\documentclass[10pt]{article}

\usepackage{amsthm,amssymb,amsbsy,amsmath,amsfonts,amssymb,amscd}
\usepackage{subfigure, hyperref}
\usepackage{mypackage}
\usepackage{setspace}
\usepackage{url}
\theoremstyle{plain}

\theoremstyle{definition}
\newtheorem{defn}[theorem]{Definition}

\theoremstyle{remark}
\hypersetup{ 
backref=true, 
pagebackref=true,
hyperindex=true, 
colorlinks=true, 
breaklinks=true, 
urlcolor= blue, 
linkcolor= blue, 
bookmarks=true, 
bookmarksopen=true, 
bookmarksnumbered=true,
pdftitle={On the locomotion and control of a self-propelled shape-changing body in a fluid}, 
pdfauthor={T. Chambrion and A. Munnier}, 
pdfsubject={This paper provides a rigourous framework for the study of the controllability of swimming devices.},
pdfkeywords={Control of PDEs, perfect fluid, locomotion, geometric control theory.} 
}

\date{\today}
\author{Thomas
Chambrion\footnote{E-mail: \texttt{thomas.chambrion@iecn.u-nancy.fr} \hfill and \hfill \texttt{alexandre.munnier@iecn.u-nancy.fr}} \and Alexandre
Munnier$^{\ast}$\footnote{Both authors are with Institut Elie Cartan UMR 7502, Nancy-Universit\'e,
CNRS, INRIA, B.P.~239, F-54506 Vandoeuvre-l\`es-Nancy Cedex,
France, and INRIA
Lorraine, Projet CORIDA. }}
\title{On the locomotion and control of a self-propelled shape-changing body in a fluid}
\begin{document}
\maketitle
\begin{abstract}

In this paper we study the locomotion of a shape-changing body swimming in a two-dimensional perfect fluid of infinite extent. The shape-changes are prescribed as functions of time satisfying constraints ensuring that they result from the work of internal forces only: conditions necessary for the locomotion to be termed {\it self-propelled}. The net {\it rigid motion} of the body results from the exchange of momentum between these shape-changes and the surrounding fluid.

The aim of this paper is three-folds: First, it describes a rigorous framework for the study of animal locomotion in fluid. Our model differs from previous ones mostly in that the number of degrees of freedom related to the shape-changes is infinite. The Euler-Lagrange equations are obtained by applying the {\it least action principle} to the system body-fluid. The formalism of Analytic Mechanics provides a simple way to handle the strong coupling between the internal dynamic of the body causing the shape-changes and the dynamic of the fluid.  The Euler-Lagrange equations take the form of a coupled system of ordinary differential equations (ODEs) and partial differential equations (PDEs). The existence and uniqueness of solutions for this system are rigorously proved.

Second, we are interested in making clear the connection between shape-changes and internal forces. Although classical, it can be quite surprising to select the shape-changes to play the role of control because the internal forces they are due to seem to be a more natural and realistic choice. We prove that, when the number of degrees of freedom relating to the shape-changes is finite, both choices are actually equivalent in the sense that there is a one-to-one relation between shape-changes and internal forces.

Third, we show how the control problem, consisting in associating with each shape-change the resulting trajectory of the swimming body, can be analysed within the framework of geometric control theory. This allows us to take advantage of the powerful tools of differential geometry, such as the notion of Lie brackets or the Orbit Theorem and to obtain the first theoretical result (to our knowledge) of control for a swimming body in an ideal fluid. We derive some interesting and surprising tracking properties: For instance, for any given shape-changes producing a net displacement in the fluid (say, moving forward), we prove that there exists other shape-changes arbitrarily close to the previous ones, that leads to a completely different motion (for instance, moving backward): This phenomenon will be called {\it Moonwalking}. Most of our results are illustrated by numerical examples.
\end{abstract}
\noindent{\bf Keywords: }Locomotion, biomechanics, ideal
fluid, geometric control theory.\\

\noindent{\bf AMS Subject Classification: }74F10, 70S05, 76B03,
93B27.
\setcounter{tocdepth}{2}
\tableofcontents


\section{Introduction}

\subsection{History}
In the last decade, much work have been done by mathematicians to better understand the dynamics of swimming in a fluid. This interest has grown from the observation that fishes and aquatic mammals evolved swimming capabilities far superior to those achieved by human technology and consequently provide an attractive model for the design of biomimetic robots. Such swimming devices propelled and steered by shape-changes would be more efficient, stealthier and more maneuverable than if 
propeller-driven. This explains why, for instance, autonomous underwater vehicles are catching the attention of the petroleum industry for their possible use in the maintenance of off-shore installations. In the field of nano-technology, the design of nano-robots able to perform basic tasks of medicine is a challenge for the forthcoming years.

Significant contributions to the understanding of the biomechanics of swimming have been made by Lighthill \cite{Lighthill:1975aa}, Taylor \cite{Taylor:1951aa, Taylor:1952aa}
Childress \cite{Childress:1981aa} and Wu \cite{Wu:2001aa}. An interesting survey on the general theme of fish locomotion  written by Sparenberg \cite{Sparenberg:2002aa} is worth being mentioned as well. 

Experiments have shown that the vortices generated by the tail fins of fish play a crucial role in their locomotion and some
models incorporate artificially produced vortices \cite{Mason:aa, Zhu:2002aa, Triantafyllou:2000aa}. If we do not neglect the
viscosity effects, the relevant model incorporates the non-stationary Navier-Stokes equations for the fluid coupled with
Newton's laws for the fish-like swimming object. This perspective
is adopted by Carling, Williams and Bowtell in \cite{Carling:1998aa}, Liu and Kawachi in \cite{Liu:1999aa}, Galdi in \cite{Galdi:1999aa} or 
San Mart{\'{\i}}n, Scheid, Takahashi and Tucsnak in \cite{San-Martin:2008aa}.  
However, and contrary to some common beliefs, forces and momenta acting on the fish body by shed vortices are not solely responsible for the net locomotion
and among numerous mathematical articles studying fish
locomotion, most of them address the case of a potential flow
which is by definition vortex-free: let us mention here the works of Kelly and Murray \cite{Kelly:2000aa},  Kozlov and Onishchenko \cite{Kozlov:2003aa}, Kanso, Marsden, Rowley 
and Melli-Huber \cite{Kanso:2005aa}, Melli, Rowley and Rufat \cite{Melli:2006aa} and Munnier \cite{Munnier:2008aa, Munnier:2008ab}. This is also the point of view we
have chosen in this article.

Although crucial for the design of autonomous underwater vehicles, results on control or on motion-planning for this kind of problem are very few, most of them focus on articulated bodies as in the works of Alouges, DeSimone and Lefebvre \cite{Alouges:2008aa} (dealing with a  three spheres mechanism swimming in a viscous fluid) or those of Mason \cite{Mason:2003aa} and Melli, Rowley and Rufat \cite{Melli:2006ab}.
More authors have considered the problem of controlling immersed {\it rigid} solids. Let us mention for instance the paper \cite{San-Martin:2007aa} of  San Mart\'\i n,  Takahashi, and Tucsnak, in which the control is chosen to be the relative fluid's velocity (thrust) on the solid's boundary, while in the paper \cite{Chambrion:2008aa} of  Chambrion and Sigalotti, the control is the {\it impulse} of the fluid also on the body's boundary. Finally, in \cite{Kozlov:2003aa, Kozlov:2003ab}, the authors examine how a body with a rigid hull wich can modify the balance of its internal mass has the ability to steer and propel itself in a perfect fluid.

The shape-changing body we consider in this paper (sometimes called {\it amoeba} for its similarity with this single-celled animal) is inspired by that of Shapere and Wilczek introduced in \cite{Shapere:1989aa} and further discussed in \cite{Cherman:2000aa}.  The Euler-Lagrange equations of motion are obtained following the method described in \cite{Kanso:2005aa} for articulated rigid bodies (i.e. shape-changing bodies made of rigid solids linked by hinges) or in \cite{Munnier:2008aa, Munnier:2008ab} for more general deformations but adapted here to our infinite dimensional model.  Notice, however, that the main idea, consisting in invoking the {\it least action principle} of Lagrangian Mechanics for the overall system fluid-body, goes back to 
the works of Thomson, Tait and Kirchhoff in their studies of the motion of rigid solids in a perfect fluid. In his book~\cite[chap. VI, page 160]{Lamb:1993aa},  Lamb 
explains that: {\it The cardinal feature
of the methods followed by these writers consists in this, that
the solids and the fluid are treated as forming together one
dynamical system, and thus the troublesome calculation of the
effect of the fluid pressures on the surfaces of the solids is
avoided}.

The shape-changes are prescribed as functions of time and used as {\it controls} to propel and steer the amoeba. Since we want the motion to be self-propelled, these deformations have to result from the work of internal forces and torques only. It entails some physical constraints previously discussed in \cite{San-Martin:2008aa, Munnier:2008aa, Munnier:2008ab}. Obviously, the fluid domain changes along with the shape of the amoeba. As in 
\cite{Mason:aa} and mostly in \cite{Shapere:1989aa}, these changes are described by means of conformal mappings allowing the explicit computation of the fluid potential. Indeed, one of the main difficulties we are faced with in studying fish-like locomotion is the precise analysis of the fluid potential with respect to the variations of the fluid domain.
\subsection{Main results}
The main results of this paper adress, on the one hand, the modelling  and the well-posedness of the Euler-Lagrange equations and, on the other, the associated control problem. 

Our model for a swimming shape-changing body, thoroughly described in Section~\ref{SEC_Setting}, constitutes the first novel concept set out in this paper. Although profoundly inspired by \cite{Cherman:2000aa}, it has been substantially improved. In particular, in the article of Shapere and Wilczek, the inertia of the amoeba is neglected whereas here, the time-evolving mass-distribution inside the animal is taken into account. Further, the shape-changes we consider are richer than those of the model in \cite{Shapere:1989aa}, for they have an infinite number of degrees of freedom. This latter improvement leads us, in Proposition~\ref{prop:decomp:pot}, to extend the validity of Kirchhoff's law  to the case where the decomposition of the fluid potential into a linear combination of elementary potentials involves an infinite number of terms. A generalization of the classical notion of mass matrix is also required and this task is carried out in Subsection~\ref{sec:mass:matrix}. The elementary potentials are solutions of boundary value problems set on the fluid domain which changes along with the shape of the amoeba. We prove in Theorem~\ref{sum_theo} that the elementary potentials, seen as functions of the {\it shape} of the body and valued into suitable Sobolev spaces, are smooth. As a straightforward consequence, we deduce in Theorem~\ref{reg:matrices}  that the mass matrices and the Lagrangian function, seen also as functions of the {\it shape} of the amoeba and valued into suitable spaces of bilinear mappings, are also smooth. This result is used in Subsection~\ref{subsec:equationofmotion} to derive the Euler-Lagrange equation \eqref{equation_of_motion:1}, a system of ODEs that governs the dynamic of the system fluid-body. It is given in a form very convenient to study locomotion problems for it gives the {\it rigid motion} of the body with respect to the shape-changes. The sharp regularity results obtained for the mass matrices are also required to prove the well-posedness of this system of ODEs, see in Proposition~\ref{PRO_well_posedness_eq_of_motion}.

We next study the connection between shape-changes and internal forces. In Section~\ref{SEC:control_with_internal_forces}, we show how the internal forces causing a given shape-change can be computed. In Theorem~\ref{theo:equivalence}, we prove that if we consider only finite dimensional deformations, there is a one-to-one relation between shape-changes and  internal forces: in other words, it is equivalent, in this case, to select for controlling the swimming animal either the shape-changes or the internal forces. 

The usual control problems associated with our dynamical system could be stated as follows: 
\begin{itemize}
\item {\bf Controllability:} Is it possible to find shape-changes that propel the swimming animal from a given starting position to a specific end one?
\item {\bf Tracking:} Is it possible to find shape-changes that allow the amoeba to follow (approximately) any given trajectory? 
\end{itemize}
We show that the answer to these two questions is positive and actually we prove more.  In Theorem~\ref{THE_diminf_tracking}, we claim that the swimming animal can not only  follow approximately any prescribed trajectory, but also while undergoing approximately any prescribed shape-changes. This surprising result is, to our knowledge, the first theoretical controllability result obtained for a body swimming in a perfect fluid.

Most of these results are illustrated by numerical examples in Section~\ref{SEC_numerical}.
\subsection{Outline of the paper}
The modeling is performed in Section~\ref{SEC_Setting}, which deals mainly with the description of the shape-changes and the kinematics of the problem. Dynamics is treated in Section~\ref{SEC_Eumer_Lagrange} where we derive the Euler-Lagrange equation. In Subsection~\ref{SEC:control_with_internal_forces} we discuss the equivalence between controlling by shape-changes and by internal forces.
Section~\ref{SEC_control_abstract} is devoted to the control problem: the main controllability results are stated (with comments) in Subsection~\ref{mainandcomments} and then proved in the remaining Subsections. Finally, Section~\ref{SEC_numerical} is dedicated to the study of trajectory-planning and to numerical simulations. 

Appendices~\ref{additional-material} provides additional material used in  Sections~\ref{SEC_Setting} and \ref{SEC_Eumer_Lagrange} and Appendix~\ref{SEC_Appendix_Orbit_Th} contains classical results of geometric control theory used in Section~\ref{SEC_control_abstract}.
\section{Setting of the problem} \label{SEC_Setting}
\subsection{Notation}
\label{notation}
In this Subsection, we introduce the main notation and the functional framework.
\subsubsection{Systems of coordinates}
Let $(\mathbf e_1,\mathbf e_2)$ denote a reference Galilean frame by which we identify the {\it physical space} to ${\mathbf R}^2$.
At any time the amoeba occupies an open smooth connected domain $\mathcal A$ and we denote by $\mathcal F:={\mathbf R}^2\setminus\bar{\mathcal A}$ the open connected domain of the surrounding fluid. 
The coordinates in $(\mathbf e_1,\mathbf e_2)$ are denoted with lowercase letters $x=(x_1,x_2)^T$ (the superscript $T$ standing for the matrices or vectors transpose) and are commonly called the {\it spatial coordinates} (see for instance \cite[chap 15]{Marsden:1999aa}). For any $x=(x_1,x_2)^T\in{\mathbf R}^2$, $x^\perp:=(-x_2,x_1)^T$ stands for the vector $x$ positively quarter turned.

Attached to the microorganism, we define also a moving frame $(\mathbf e_1^{\ast},\mathbf e_2^{\ast})$. We choose it such that its origin coincides
at any time with the center of mass of the swimming animal and we denote by $x^\ast=(x_1^\ast,x_2^\ast)^T$ the related 
so-called {\it body coordinates}. In this frame and at any time the amoeba occupies the region 
$\mathcal A^\ast$ and the fluid the domain ${\mathcal F}^\ast:={\mathbf R}^2\setminus\bar{\mathcal A}^\ast$ (see figure~\ref{fig2}). More generally, quantities with an asterisk are expressed in the moving frame.

We define also the {\it computational space} endowed with the frame $(\mathbf E_1,\mathbf E_2)$ and in which the coordinates are denoted $z=(z_1,z_2)^T$. 
In this space, $D$ is the unitary disk and $\Omega:=\mathbf R^2\setminus\bar D$.

Throughout this paper, we will use the same notation $n$ to represent the unitary normal to $\partial\mathcal A=\partial\mathcal F$ or $\partial\mathcal A^\ast=\partial\mathcal F^\ast$ directed toward the interior of the amoeba. 
We will sometimes use complex analysis and identify the space we are working in with the complex field ${\mathbf C}$. In this case the notations introduced above will turn into $x=x_1+ix_2$, $x^\ast=x^\ast_1+ix^\ast_2$ or $z=z_1+iz_2$ with $i^2=-1$ and $x_k,\,x^\ast_k,\, z_k\in{\mathbf R}$ ($k=1,2$). The complex conjugate of $z=z_1+iz_2$ is $\bar z=z_1-iz_2$. We will sometimes also mix this notation with the polar coordinates $(r,\theta)\in\mathbf R_+\times\mathbf R/2\pi$, $z=re^{i\theta}$.
\subsubsection{Function spaces}

Let $E$ and $F$ be Banach spaces and assume that $K$ is a compact subset of $E$. The vector space $\mathcal C^m(K,F)$ ($m$ an integer, $m\geq 1$) of the functions $m$ times continuously differentiable from $K$ into $F$ is a Banach space once endowed with the norm of $W^{m,\infty}(K,F)$ (uniform convergence in $K$ of the function and all of its partial derivatives up to the order $m$).

Weighted Sobolev and Lebesgue spaces required in the resolution of boundary value problems are introduced in the Appendix, Subsection~\ref{neumann:boun}.
\subsubsection{Multilinear, polynomial and analytic functions}
Let $E_1,\ldots,E_k$ ($k\geq 1$) be Banach spaces. The set consisting of all the continuous, $k$-linear mappings from $E:=E_1\times\ldots\times E_k$ into $F$ is denoted $\mathcal L_k( E,F)$ (we will drop the subscript $k$ when $k=1$). It is a Banach space whose norm is classically defined by:
$$\|\Lambda\|_{\mathcal L_k( E,F)}:=\hspace{-0.9cm}\sup_{(e_1,\ldots,e_k)\in E,\atop \|e_j\|_{E_j}=1,\,(j=1,\ldots,k)}\hspace{-0.5cm}\|\langle \Lambda,e_1,\ldots,e_k\rangle\|_{F},\quad(\Lambda\in \mathcal L_k( E,F)).$$
When $F={\mathbf R}$, we are dealing with multilinear continuous forms and we denote merely $\mathcal L_k(E):=\mathcal L_k(E,F)$.
  
 We call {\it polynomial function} from a Banach space $E$ (or from only a subset of this space) into a Banach space $F$ any function $P$ such that there exists an integer $p\in{\mathbf N}$ (the degree of the polynomial), $A_0\in F$ and $p$ mappings $A_k\in \mathcal L_k(E^k,F)$ ($k=1,\ldots,p$) such that:
 $$P(e):=A_0+\sum_{k=1}^p\langle A_k,e,\ldots,e\rangle,\quad\forall\,e\in E.$$
We denote $\mathcal P(E,F)$ the set of all the polynomial functions from $E$ into $F$. Observe that in particular $\mathcal L(E,F)\subset\mathcal P(E,F)$. We easily prove that if $E_1$, $E_2$ and $E_3$ are three Banach spaces and $P_1\in\mathcal P(E_1,E_2)$, $P_2\in\mathcal P(E_2,E_3)$ then $P_2\circ P_1\in\mathcal P(E_1,E_3)$.

We call {\it analytic function} from a Banach space $E$ into a Banach space $F$ any function $f$ such that there exists $R>0$ (the radius of convergence), $A_0\in F$ and a sequence $(A_k)_{k\geq 1}$ with $A_k\in\mathcal L_k(E^k,F)$ satisfying:
$$\sum_{k\geq 1}|\lambda|^k\|A_k\|_{\mathcal L_k(E^k,F)}<+\infty,\quad\forall\,\lambda\in\mathbf R,\,|\lambda|<R,$$
and
$$f(e):=A_0+\sum_{k\geq 1}\langle A_k,e,\ldots,e\rangle,\quad\forall\,e\in E,\,\|e\|_E<R.$$
We refer to the book \cite[\S 4]{Cartan:1961aa},  for further details on analytic functions in Banach spaces.
 \subsubsection{Banach spaces of series}
We denote any complex series by $\mathbf c:=(c_k)_{k\geq 1}$ where for any $k\geq 1$, $c_k:=a_k+ib_k$, $a_k, b_k\in{\mathbf R}$. Most of the complex series we will consider in this article live in the Banach space:
$$\mathcal S:=\Big\{(c_k)_{k\geq 1}\in{\mathbf C}^{\mathbf N}\,:\,\sum_{k\geq 1}k(|a_k|+|b_k|)<+\infty\Big\},$$
endowed with its natural norm $\|\mathbf c\|_{\mathcal S}:=\sum_{k\geq 1} k(|a_k|+|b_k|)$.
The unitary ball of $\mathcal S$ is denoted $\mathcal B$ and we will require the following open subset of $\mathcal S$, containing $\mathcal B$: 
$$\mathcal D:=\Big\{\mathbf c:=(c_k)_{k\geq 1}\in\mathcal S\,:\,\sup_{z\in\partial D}\Big|\sum_{k\geq 0}(k+1) c_{k+1} z^{k}\Big|<1\Big\}.$$
Let us introduce also the Hilbert space:
$$\mathcal T:=\Big\{(c_k)_{k\geq 1}\in{\mathbf C}^{\mathbf N}\,:\,\sum_{k\geq 1}k(|a_k|^2+|b_k|^2)<+\infty\Big\},$$
whose natural norm is denoted $\|\mathbf c\|_{\mathcal T}$.
From the obvious identity $\|\mathbf c\|_{\mathcal T}\leq \|\mathbf c\|_{\mathcal S}$, we deduce that $\mathcal S\subset \mathcal T$. 
Further elementary results about these spaces are given in the Appendix, Subsection~\ref{banach:series}.
\subsection{Kinematics of the shape-changing amoeba} 
\subsubsection{Description of the shape-changes}
The shape-changes of the amoeba are described with respect to the moving frame $(\mathbf e^\ast_1,\mathbf e^\ast_2)$ by a $\mathcal C^1$ diffeormorphism $\chi(\mathbf c)$, depending on a {\it shape 
(or control) variable}
${\mathbf c}\in\mathcal D$, which maps the unitary closed disk $\bar D$ of the computational space onto the domain $\bar {\mathcal A}^\ast$ of the physical space. We can write, according to our notation, that for any $\mathbf c\in\mathcal D$, $
\bar{\mathcal A}^\ast=\chi(\mathbf c)(\bar D)$ and $x^\ast = \chi(\mathbf c)(z)$, ($z\in \bar D$).
\begin{figure}[h]
\centerline{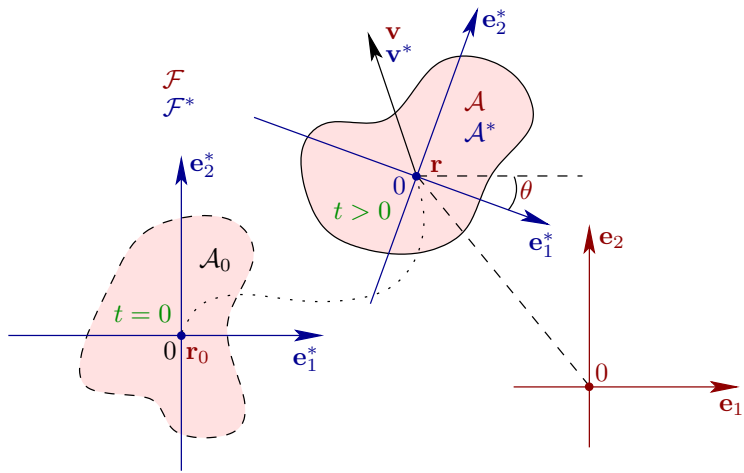}
\caption{\label{fig2}Quantities are denoted with an asterisk when they are expressed in the moving frame $(\mathbf e_j^\ast)$. }
\end{figure}
The map $\chi(\mathbf c)$ is defined for all $\mathbf c\in\mathcal D$ by:
\begin{equation}
\label{def:chi}
\chi(\mathbf c)(z):=\chi_0(z)+\langle \chi_1,\mathbf c\rangle(z),\quad(z\in \bar D),
\end{equation}
where
$\chi_0(z):=z$ and $\langle \chi_1,\mathbf c\rangle(z):=\sum_{k\geq 1}c_k\bar z^k$, ($z\in \bar D$).
Introducing in polar coordinates $(r,\theta)\in[0,1]\times {\mathbf R}/2\pi$ the vectors fields:
\begin{subequations}
\label{def:Uk}
\begin{equation}
U^a_k(r,\theta):=r^k\begin{bmatrix}
\cos(k\theta)\\
-\sin(k\theta)
\end{bmatrix}
\quad\text{and}\quad
U^b_k(r,\theta):=r^k\begin{bmatrix}
\sin(k\theta)\\
\cos(k\theta)
\end{bmatrix},
\end{equation}
we have the equivalent definition:
\begin{equation}
\chi(\mathbf c)(r,\theta):=r\begin{bmatrix}\cos(\theta)\\ \sin(\theta)
\end{bmatrix}+\sum_{k\geq 1}a_k U^a_k(r,\theta)+b_k U^b_k(r,\theta).
\end{equation}
\end{subequations}
We next introduce likewise the function $\phi(\mathbf c)$ that maps $\bar\Omega$ onto $\bar{\cal F}^\ast$. It is defined for all $\mathbf c\in{\cal D}$ by:
\begin{equation}
\label{def:2}
\phi(\mathbf c)(z):=z+\sum_{k\geq 1}c_k z^{-k},\quad (z\in\bar\Omega).
\end{equation}
Since $\bar z=1/z$ for all $z\in\partial D$, we deduce that $\chi(\mathbf c)|_{\partial D}=\phi(\mathbf c)|_{\partial\Omega}$ and the following mapping is continuous in $\mathbf C$ for all $\mathbf c\in\mathcal D$:
$$\Phi(\mathbf c)(z):=\begin{cases}\chi(\mathbf c)(z)&\text{if }z\in D,\\
\phi(\mathbf c)(z)&\text{if }z\in\bar\Omega.
\end{cases}$$
We can summarize the main properties of $\chi$ and $\phi$ in the following proposition:
\begin{prop}
For all $\mathbf c\in{\cal D}$, $\chi(\mathbf c):\bar D\to \bar{\cal A}^\ast$ and $\phi(\mathbf c):\bar\Omega\to\bar{\cal F}^\ast$ are both well-defined (the series in \eqref{def:chi} and \eqref{def:2} converge for all $z$ in  $\bar D$ and $\bar\Omega$ respectively) and invertible. Further, $\chi(\mathbf c)|_D$ is a $\mathcal C^1$ diffeomorphism, $\phi(\mathbf c)|_\Omega$ is a conformal mapping and $\Phi(\mathbf c)$ is an homeomorphism from $\mathbf C$ onto $\mathbf C$. 
\end{prop}
\begin{proof}
Let us denote $D\chi(\mathbf c)(z)$ the Jacobean matrix at a point $z\in D$ of $\chi(\mathbf c)(z)$, seen as a function from $\mathbf R^2$ into $\mathbf R^2$. Identifying the matrix-vector product $D\chi(\mathbf c)(z)h$ for all $h:=(h_1,h_2)^T\in\mathbf R^2$ with its complex expression in which $h=h_1+i h_2$, we can write that:
\begin{equation}
\label{det:chi}
D\chi(\mathbf c)(z)h=h+\Big(\sum_{k\geq 0}(k+1)c_{k+1}\bar z^k\Big)\bar h.
\end{equation}
The series in the right hand side term is the conjugate of the holomorphic function $z\in D\mapsto \sum_{k\geq 0}(k+1)\bar c_{k+1} z^k$ whose maximum is achieved, according to the maximum principle, on $\partial D$. We deduce that for all $\mathbf c\in\mathcal D$, $\chi(\mathbf c)-{\rm Id}$ is a strict contraction in $D$. Invoking the local inversion Theorem, we deduce that for all $\mathbf c\in\mathcal D$, $\chi(\mathbf c)$ is a local $\mathcal C^1$ diffeomorphism in $D$. 


We use roughly the same arguments to prove that $\phi(\mathbf c)$ is locally a conformal mapping. Starting from the expression $\phi'(\mathbf c)(z)-1=(-1/z^2)\sum_{k\geq 0}(k+1)c_{k+1} z^{-k}$, we deduce, according to the maximum principle and since $z^{-k}=\bar z^k$ for all $z\in\partial D$, that $|\phi'(\mathbf c)(z)-1|\leq \sup_{z\in\partial D}|\sum_{k\geq 1}(k+1)c_{k+1} \bar z^{k}|=\sup_{z\in\partial D}|\sum_{k\geq 1}(k+1)\bar c_{k+1} z^{k}|$. It entails that for all $\mathbf c\in\mathcal D$, $\phi(\mathbf c)-{\rm Id}$ is a strict contraction in $\Omega$ and then that $\phi(\mathbf c)$ is locally a conformal mapping in $\Omega$. 

Since both mappings $\chi(\mathbf c)-{\rm Id}$ and $\phi(\mathbf c)-{\rm Id}$ are strict contractions, we draw the same conclusion for $\Phi(\mathbf c)-{\rm Id}$ in the whole complex plane $\mathbf C$. For any $\widetilde z\in\mathbf C$, the function $z\in\mathbf C\mapsto\widetilde z-(\Phi(\mathbf c)(z)-z)$ being also a strict contraction, Banach fixed point Theorem ensures the existence and uniqueness of $z\in\mathbf C$ such that $\widetilde z-(\Phi(\mathbf c)(z)-z)=z$ i.e., $\Phi(\mathbf c)(z)=\widetilde z$. The proof is then completed.
\end{proof}
\subsubsection{Body's volume, density, mass and inertia momentum}
From the relation $x^\ast:=\chi(\mathbf c)(z)$, ($z\in D$),
we deduce that the area elements ${\rm d}x^\ast$ and ${\rm d}z$ of respectively $\mathcal A^\ast$ and $D$ can be deduced one from the other by the identity:
$${\rm d}x^\ast:=|\det\,D\chi(\mathbf c)(z)|{\rm d}z,\quad(z\in D,\,x^\ast:=\chi(\mathbf c)(z)).$$
This entails that the density  $\rho^\ast_{\mathbf c}$ of the deformed amoeba ${\mathcal A}^\ast$ can be deduced from a given constant density 
$\rho_0>0$ by 
the conservation-of-mass principle:
\begin{equation}
\label{conserve_mass}
\rho^\ast_{\mathbf c}(x^\ast)=\frac{\rho_0}{|\det\,D\chi(\mathbf c)(\chi(\mathbf c)^{-1}(x^\ast))|},\qquad(x^\ast\in{\mathcal A}^\ast).
\end{equation}
\begin{lemma}
For all $\mathbf c\in\mathcal D$, the volume of the amoeba is:
\begin{equation}
\label{express:volum}
{\rm Vol}(\mathcal A)=\pi(1-\|\mathbf c\|^2_{\mathcal T}).
\end{equation}
\end{lemma}
\begin{proof}
Starting from the expression \eqref{det:chi} of $D\chi(\mathbf c)$, we obtain that:
$$\det D\chi(\mathbf c)(z)=1-\Big(\sum_{k\geq 1}kc_k\bar z^{k-1}\Big)\Big(\sum_{k\geq 1}k\bar c_kz^{k-1}\Big),\qquad(z\in\bar D),$$
and it is clear that the right hand side term is positive if $\mathbf c\in\mathcal D$.
Expanding it, we get in polar coordinates $(r,\theta)\in [0,1]\times{\mathbf R}/2\pi$:
\begin{multline}
\label{det}
|\det\,D\chi(\mathbf c)(r,\theta)|=1-\sum_{k\geq 1}\sum_{k'\geq 1}kk'r^{k+k'-2}(a_ka_{k'}+b_kb_{k'})\cos((k-k')\theta)\\
+\sum_{k\geq 1}\sum_{k'\geq 1}kk'r^{k+k'-2}(a_kb_{k'}-a_{k'}b_{k})\sin((k-k')\theta).
\end{multline}
Since the volume of the ameoba is given by ${\rm Vol}(\mathcal A)=\int_D|\det\,D\chi(\mathbf c)(z)|{\rm d}z$,
straightforward and easy computations lead to formula \eqref{express:volum}.
\end{proof}
We can now define the element of mass in $D$ by ${\rm d}m_0:=\rho_0\,{\rm d}z$,  and likewise ${\rm d}m^\ast:=\rho^\ast_{\mathbf c}{\rm d}x^\ast$, 
is the element of mass in ${\mathcal A}^\ast$.

As the diffeomorphism $\chi(\mathbf c)$ is modeling physical shape-changes, it has to satisfy some constraints. 
Let us consider a continuous and piecewise $\mathcal C^1$ control function $t\in[0,T]\mapsto \mathbf c(t)\in\mathcal D$ ($T>0$). This regularity entails that its time derivative $\dot{\mathbf c}(t)$ exists for all but a finite number of $t$ in $[0,T]$. 
We denote $\dot{\mathbf c}=(\dot c_k)_{k\geq 1}\in\mathcal S$ and for all $k\geq 1$, $\dot c_k:=\dot a_k+i\dot b_k$. We introduce as well:
\begin{equation}
\label{def:chi:dot}
\dot\chi(\mathbf c)(z):=\frac{d\chi({\mathbf c})}{dt}(z)=\langle \chi_1,\dot{\mathbf c}\rangle,\quad(z\in\bar D,\,t\geq 0).
\end{equation}
Because of the incompressibility of the fluid, its volume has to be constant. We draw the same conclusion for the volume of the amoeba because its volume is nothing but the complementary of the volume of the fluid. According to \eqref{express:volum}, it means that the function $t\in[0,T]\mapsto\|\mathbf c(t)\|_{\mathcal T}$ has to be constant for the control function to be physically allowable. We introduce the notation $\mu:=\|\mathbf c(0)\|_{\mathcal T}$ and the subset of $\mathcal S$:
\begin{equation}
\mathcal E(\mu):=\big\{\mathbf c\in \mathcal S\,:\,\|\mathbf c\|_{\mathcal T}=\mu\big\}.
\end{equation}
Remember that for the map $\chi(\mathbf c)$ to be injective, $\mathbf c$ has to belong to $\mathcal D$. We define then also:
\begin{equation}
\label{delellipse}
\mathcal E^\bullet(\mu):=\big\{\mathbf c\in \mathcal D\,:\,\|\mathbf c\|_{\mathcal T}=\mu\big\}.
\end{equation}
According to Parseval's identity:
\begin{align*}
\sum_{k\geq 1}k|c_k|^2\leq \sum_{k\geq 1}k^2|c_k|^2=&\frac{1}{2\pi}\int_0^{2\pi}\Big(\sum_{k\geq 0} (k+1)c_{k+1} e^{-ik\theta}\Big)^2{\rm d}\theta\\
&\leq \sup_{z\in\partial D}\Big|\sum_{k\geq 0} (k+1)\bar c_{k+1} z^k\Big|^2,
\end{align*}
and we deduce that $\mathcal E^\bullet(\mu)$ is non empty if and only if $\mu<1$. The constant volume of the amoeba can next be rewritten:
\begin{equation}
\label{express:vol}
{\rm Vol}(\mathcal A)=\pi(1-\mu^2),\quad(t\geq 0).
\end{equation}
Differentiating with respect to time identity \eqref{express:volum}, we get an equivalent formulation for the conservation of the amoeba's volume:
\begin{subequations}
\label{const_volume}
\begin{equation}
\label{const_volume:1}
\sum_{k\geq 1}k(\dot a_k a_k+\dot b_k b_k)=0,\quad(t\geq 0),
\end{equation}
or, with notation of Subsection~\ref{banach:series} of the Appendix:
\begin{equation}
\label{const_volume:2}
\langle G(\mathbf c),\dot{\mathbf c}\rangle = 0,\quad(t\geq 0).
\end{equation}
\end{subequations}
\begin{rem}
Quite surprisingly, the constraint on the body's volume is not required any longer in dimension 3 as proved in \cite{Munnier:2008ab}. A physical explanation is that the 2d case can be seen as a 3d model in which we consider the swimming animal as an infinite cylinder of section $\mathcal A$. Hence, any shape-change of the body which does not preserve its volume, entails an infinite variation of the fluid's volume - which is impossible because the fluid is incompressible. With a real 3d model, although the fluid may be also incompressible, the finite variations of the fluid's volume due to the shape-changes can, in some sense, be {\it neglected} when compared with the infinite overall amount of fluid. From a mathematical point of view, the boundary value problem driving the motion of the fluid requires 
condition~\ref{const_volume:1} (or \ref{const_volume:2}) to be well-posed in dimension 2 and not any longer in dimension 3.
\end{rem}
The mass of the amoeba is: 
\begin{equation}
m:=\int_{\mathcal A^\ast}{\rm d}m^\ast=\int_D{\rm d}m_0=\pi\rho_0. \label{mass}
\end{equation}
If we assume that the body is neutrally buoyant, we get the equality $\rho_f{\rm Vol}(\mathcal A)=m$ 
where $\rho_f>0$ is the given constant density of the fluid and we deduce that the densities $\rho_f$ and $\rho_0$ are linked  by the relation:
\begin{equation}
\label{exp:rho0}
\rho_f(1-\mu^2)=\rho_0.
\end{equation}
The inertia momentum depends on the shape of the amoeba and reads $I(\mathbf c):=\int_{\mathcal A^\ast}|x^\ast|^2\,{\rm d}m^\ast$,
or equivalently upon a change of variables $I(\mathbf c):=\int_{D}|\chi(\mathbf c)(z)|^2\,{\rm d}m_0$.
It can also be easily computed in terms of the control variable:
\begin{equation}
\label{exp:I}
I(\mathbf c)=\pi\rho_0\Big(\frac{1}{2}+\sum_{k\geq 1}\frac{1}{k+1}|c_k|^2\Big).
\end{equation}
In this form, we see that $I$, seen as a function of $\mathbf c$ valued into $\mathbf R$, belongs to $\mathcal P(\mathcal D,{\mathbf R})$. 
\subsubsection{Constraints where the motion is self-propelled}
Since we assume that the shape-changes of the amoeba are produced by internal forces and torques only (by definition of a {\it self-propelled motion}), Newton's laws ensure that the linear and the angular momenta of the animal {\bf with respect to its attached frame} $(\mathbf e^{\ast}_1, \mathbf e^{\ast}_2)$ have to
remain unchanged when it undergoes shape-changes. Hence, we get the condition:
$$\frac{d}{dt}\left(\int_{\mathcal A^\ast}x^\ast\,{\rm d}m^\ast\right)=\mathbf 0,\quad(t\geq 0),$$
which can be rewritten, upon a change of variables and taking into account \eqref{conserve_mass}:
\begin{equation}
\label{cond_phi:a}
\int_{D}\dot\chi(\mathbf c)\,{\rm d}m_0=\mathbf 0,\quad(t\geq 0).
\end{equation}
For the angular momentum, the same arguments yield:
\begin{equation}
\label{cond_phi:b}
\int_{D}\dot\chi(\mathbf c)
\cdot \chi({\mathbf c})^\perp\,{\rm d}m_0=0,\quad(t\geq 0).
\end{equation}
Condition \eqref{cond_phi:a} is actually intrinsically satisfied for any control function $t\mapsto\mathbf c(t)$. 
For condition \eqref{cond_phi:b}, observe first that, with definition \eqref{def:Uk}, $(U_k^a)^\perp=U_k^b$ and $(U_k^b)^\perp=-U_k^a$,
which leads after some basic algebra to the identity:
\begin{subequations}
\label{C1}
\begin{equation}
\label{C1:1}
\sum_{k\geq 1}\frac{1}{k+1}(\dot b_{k}a_{k}-\dot a_{k}b_{k})=0,\quad(t\geq 0),
\end{equation}
or equivalently, with the notation of Subsection~\ref{banach:series}:
\begin{equation}
\label{C1:2}
\langle F({\mathbf c}),\dot{\mathbf c}\rangle =0,\quad(t\geq 0).
\end{equation}
\end{subequations}
This constraint together with \eqref{const_volume:2} lead us to introduce the notion of {\it allowable} control:
\begin{definition}[Physically allowable control]
\label{alow:cont}
A smooth function $t\in[0,T]\mapsto\mathbf c(t)\in\mathcal S$ (for some real positive $T$) is said to be Physically allowable when:
\begin{itemize}
\item There exists $\mu>0$ such that $\mathbf c(t)\in \mathcal E^\bullet(\mu)$ for all $t\geq 0$.
\item Constraint \eqref{C1:2} is satisfied for all $t\in]0,T[$.
\end{itemize}
\end{definition}
\begin{figure}[H]     
     \centering
     \begin{tabular}{|c|c|c|c|c|}
     \hline
     \subfigure
     {\includegraphics[width=.2\textwidth]{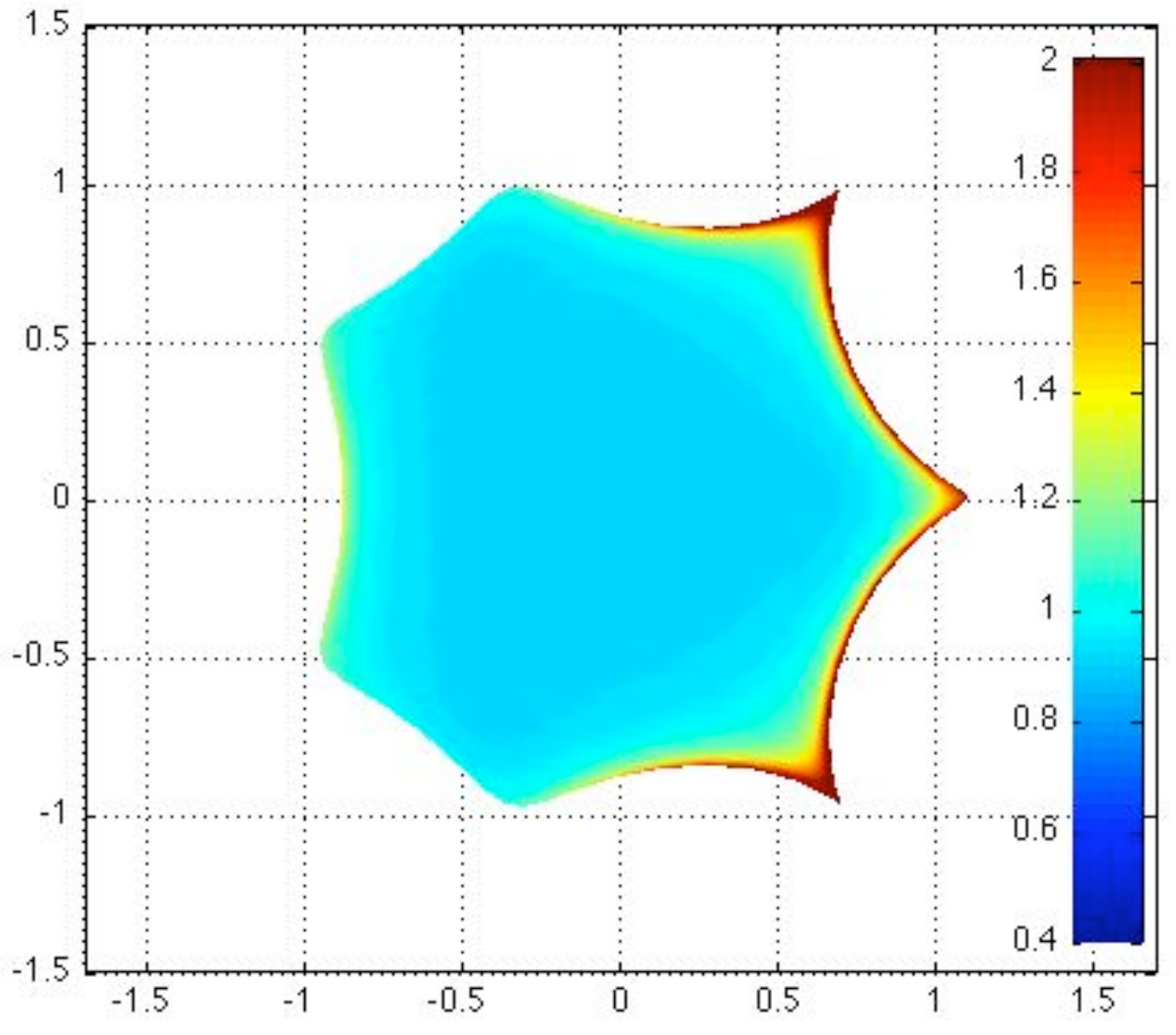}}
     &
     \subfigure
     {\includegraphics[width=.2\textwidth]{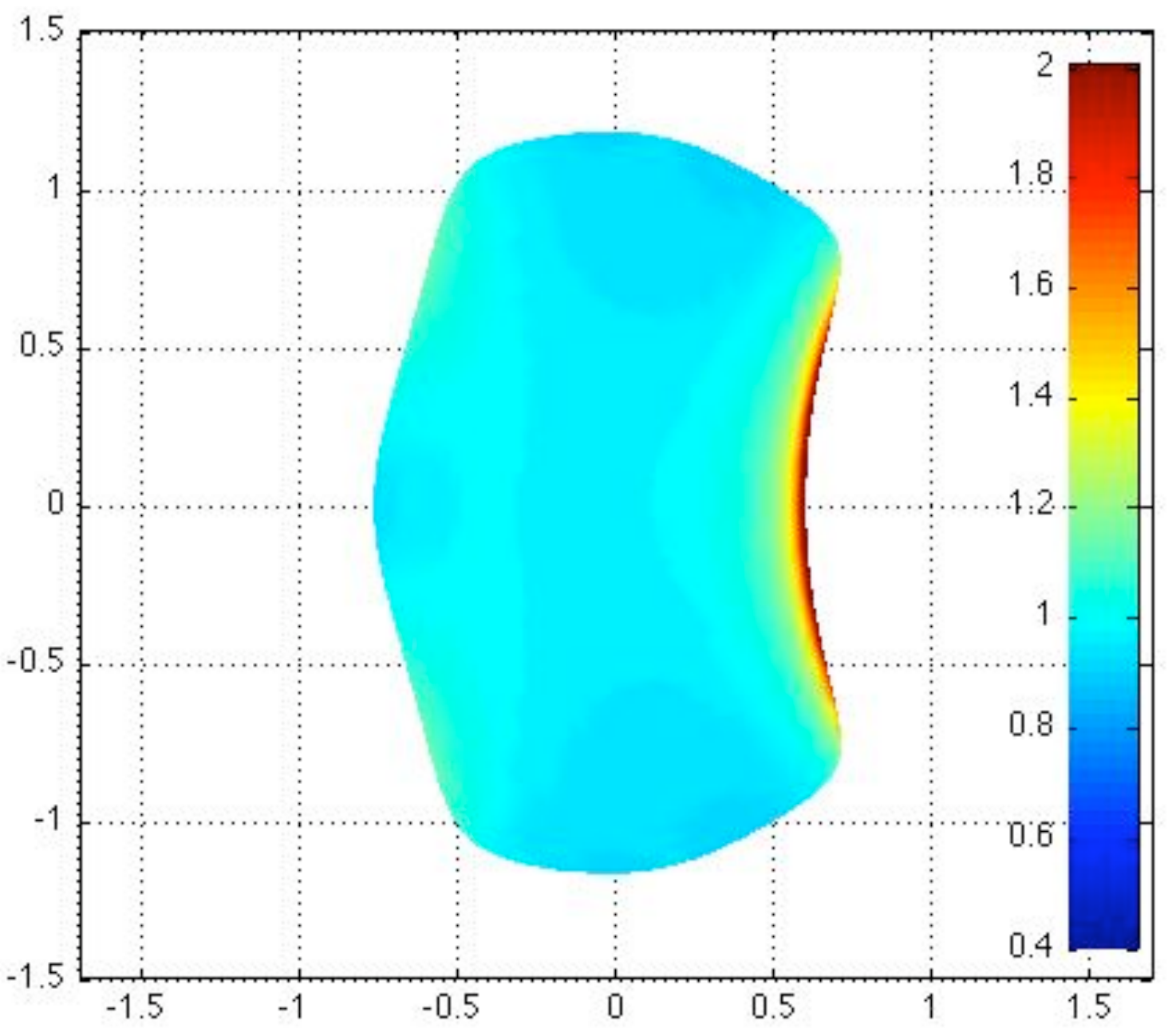}}
     &
      \subfigure
     {\includegraphics[width=.2\textwidth]{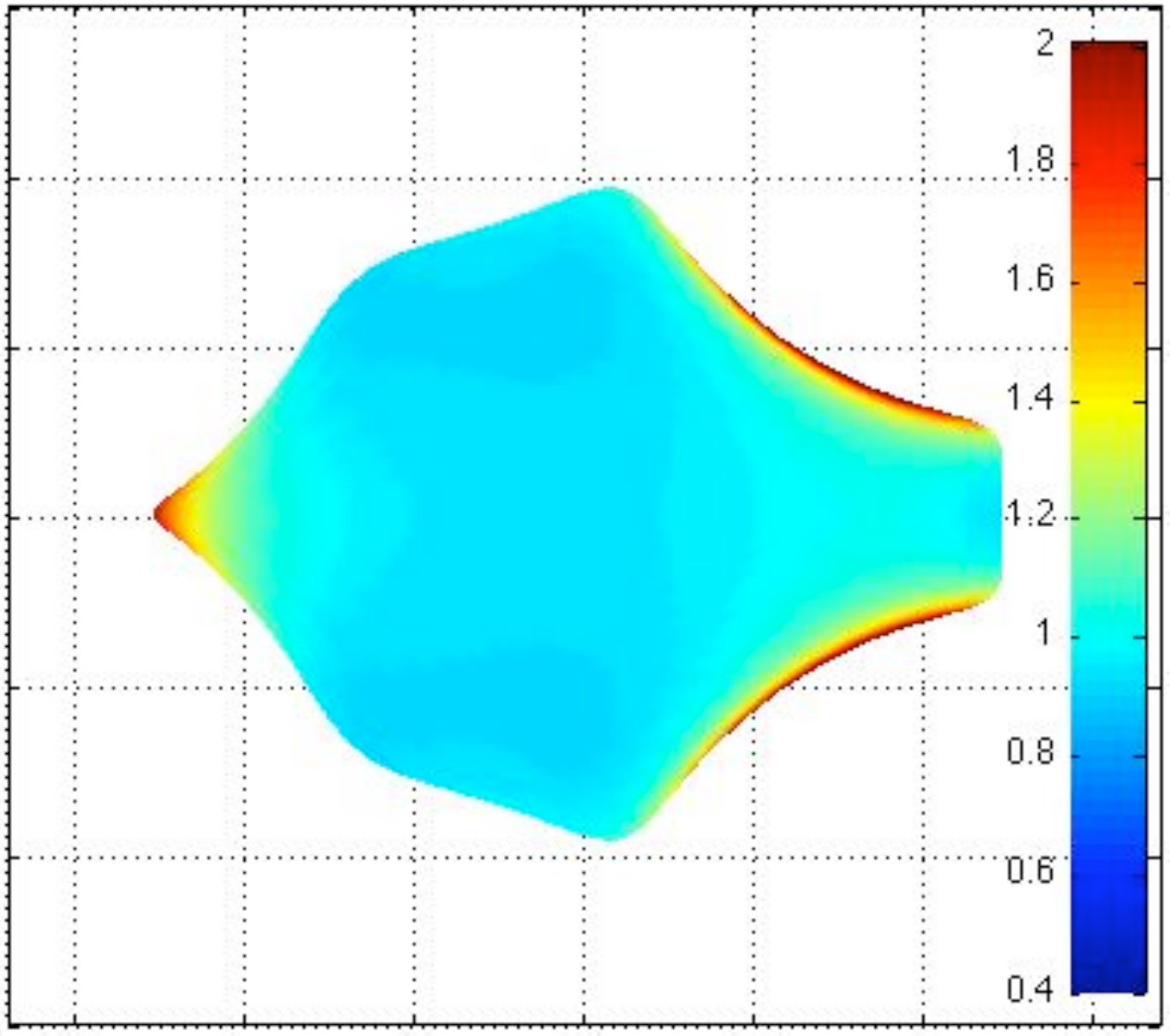}}
     &
     \subfigure
     {\includegraphics[width=.2\textwidth]{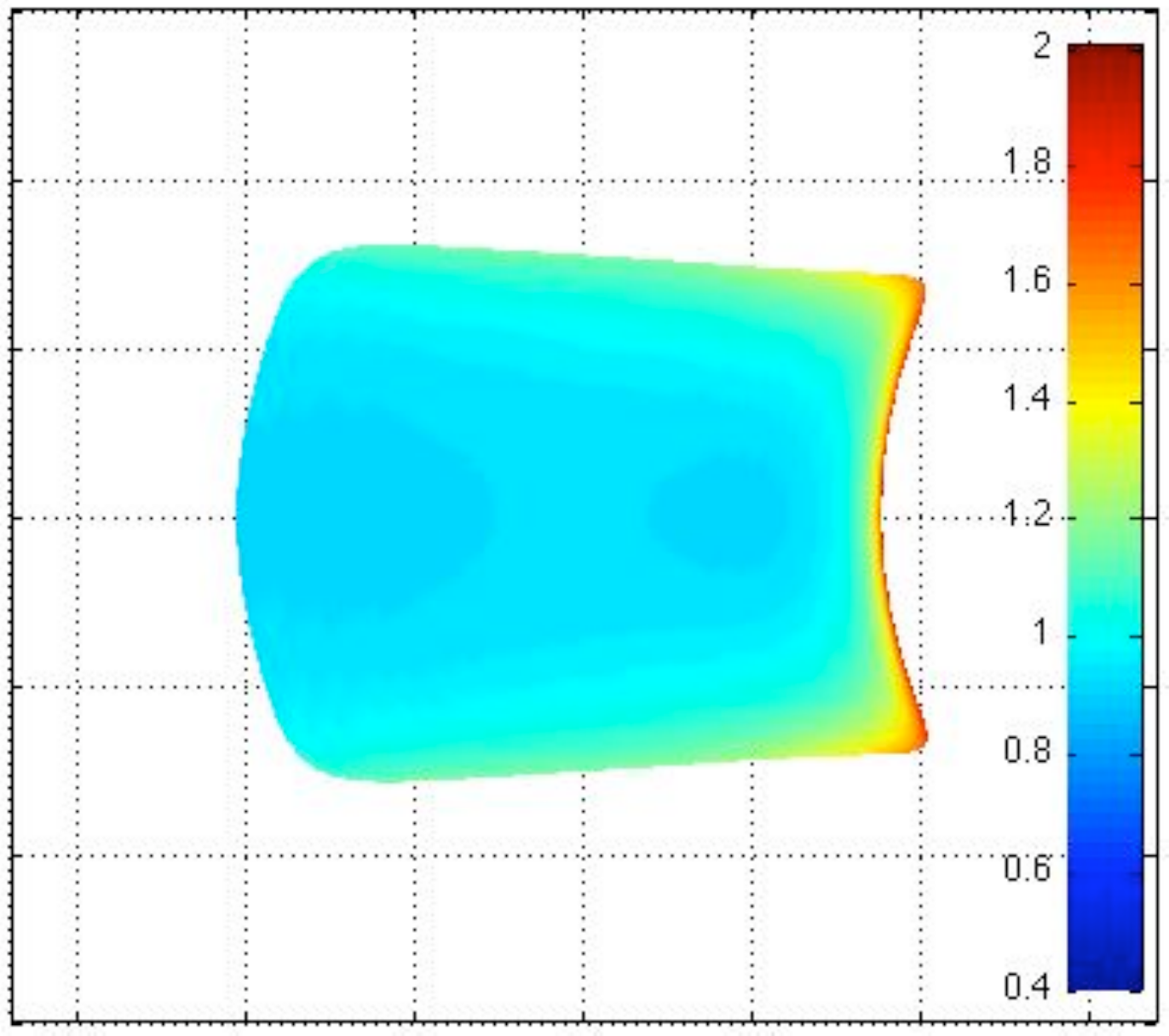}}\\
     \hline
\end{tabular}
\caption{Examples of physically allowable shape-changes ($\mu=1/\sqrt{12}$, $\|\mathbf c\|_{\mathcal S}<1$). The colors gives the value of the density inside the animal (neutrally-buoyant case).}
\end{figure}

\subsection{Rigid Motion, Velocity}
The overall motion of the amoeba in the fluid consists in the superimposition of its shape-changes with a rigid motion. The (prescribed) shape-changes have been described in the preceding Subsection along with
the (unknown) net {\it rigid motion} results from the exchange of momentum between the shape-changes with the surrounding fluid.
It is described by elements $\mathbf q:=(\mathbf r,\theta)^T$ of $\mathcal Q:=\mathbf R^2\times\mathbf R/2\pi$ where $\mathbf r:=(r_1,r_2)^T\in\mathbf R^2$ is a vector giving the position of the center-of-mass of the body and $\theta\in\mathbf R/2\pi$ an angle giving its orientation with respect to $(\mathbf e_1,\mathbf e_2)$; see Fig.~\ref{fig2}. If we denote by $R(\theta)\in\,$SO(2) the rotation matrix of angle $\theta$, then we have the relations $R(\theta)\mathbf e_j=\mathbf e_j^\ast$ ($j=1,2$).

Let the shape-changes be frozen for a while and consider a physical point attached to the body undergoing only a rigid motion.  Then, there exists a smooth function $t\in[0,T]\mapsto \mathbf q(t):=(\mathbf r(t),\theta(t))\in \mathcal Q$ ($T>0$) such that the point's coordinates in $(\mathbf e_1,\mathbf e_2)$ be given by $x=R(\theta)x_0+\mathbf r$ ($x_0\in\mathbf R^2$ being the coordinates at the time $t=0$). Next, compute the time derivative of this expression and denote by $\dot{\mathbf q}:=(\dot{\mathbf r},\omega)\in\mathbf R^3$ the time derivative of $\mathbf q$. Since we classically have $\partial_\theta R(\theta)R(\theta)^Tx=x^\perp$ for all $x=(x_1,x_2)^T\in{\mathbf R}^2$, we deduce that the {\it Eulerian} velocity of the point is $\mathbf v_r(x)=\omega(x-\mathbf r)^\perp+\dot{\mathbf r}$. It can also be expressed in the moving frame $(\mathbf e_1^\ast,\mathbf e_2^\ast)$ and reads $\mathbf v_r^\ast(x^\ast)=\omega(x^\ast)^\perp+\dot{\mathbf r}^\ast$ where $\dot{\mathbf r}^\ast:=R(\theta)^T\dot{\mathbf r}$. This leads us to introduce also the notation $\dot{\mathbf q}^\ast:=(\dot{\mathbf r}^\ast,\omega)^T\in\mathbf R^3$.

Let us return to the general case where the shape-changes are taken into account. The coordinates in $(\mathbf e_1,\mathbf e_2)$ of a physical point attached to the amoeba are given by:
$x=R(\theta)\chi(\mathbf c)(z_0)+\mathbf r$, where at the time $t=0$, $\mathbf c(0)=\mathbf c_0$ and $x(0)=x_0=\chi(\mathbf c_0)(z_0)$ for some $z_0\in D$. Observe that we can always assume that at the time $t=0$, $\mathbf q=\mathbf q_0:=(\mathbf 0,0)^T$.
We deduce that the {\it Eulerian velocity} at a point $x$ of ${\mathcal A}$ is:
$$\mathbf v(x)=\omega(x-\mathbf r)^\perp+\dot{\mathbf r}+R(\theta)\dot\chi(\mathbf c)[\chi(\mathbf c)^{-1}(R(\theta)^T(x-\mathbf r)))].$$
When expressed in the moving frame $(\mathbf e_1^\ast,\mathbf e_2^\ast)$ it reads:
\begin{equation}
\label{convective}
\mathbf v^\ast(x^\ast)=
(\omega\, x^{\ast\perp}+\dot{\mathbf r}^\ast)+\dot\chi(\mathbf c)(\chi(\mathbf c)^{-1}(x^\ast)),\quad(x^\ast\in\mathcal A^\ast).
\end{equation}
\subsection{Potential flow}
The fluid is assumed to be incompressible and inviscid. We denote $\rho_f>0$ its constant density and we set ${\rm d}m^\ast_f:=\rho_f{\rm d}x^\ast$
the element of mass in ${\mathcal F}^\ast$ and ${\rm d}m^0_f:=\rho_f{\rm d}z$
the element of mass in $\Omega$. 
We seek the Eulerian velocity $\mathbf u^\ast$ of the fluid, expressed in $(\mathbf e_1^\ast,\mathbf e_2^\ast)$, as the gradient of a {\it potential function} $\varphi$:
\begin{equation}
\mathbf u^\ast:=\nabla\varphi\quad\text{ in }{\mathcal F^\ast}.
\end{equation}
The incompressibility of the fluid entails that $\nabla\cdot \mathbf u^\ast=0$ and hence:
\begin{subequations}
\label{potential_function}
\begin{equation}
\Delta \varphi^\ast=0\quad\text{ in }{\mathcal F^\ast}.
\end{equation}
The classical {\it non penetrating} (or {\it slip}) condition for inviscid fluid leads to the equality:
$$\mathbf u^\ast\cdot\ n=\mathbf v^\ast\cdot n\quad\text{on }\partial{\mathcal F^\ast},$$ 
and yields the following Neumann boundary condition for $\varphi$:
\begin{equation}
\label{the_boundary_cond}
\partial_n\varphi=\mathbf v^\ast\cdot n\quad\text{ on }\partial{\mathcal F^\ast}.
\end{equation}
\end{subequations}
The boundary value problem \eqref{potential_function} admits a weak (or variational) solution in the weighted Sobolev space $H^1_N(\mathcal F^\ast)$ (see Appendix, Subsection~\ref{neumann:boun}).
The potential function is actually only defined up to an additive constant. It does not matter since we are only interested in $\nabla\varphi$ which is uniquely 
determined. Note that the potential function does depends on both $\dot{\mathbf c}$  (linearly through the boundary data) and $\mathbf c$ (through the domain $\mathcal F^\ast$). 
\subsection{Lagrangian of the system fluid-amoeba}
\label{def:lagr}
Since we neglect gravity, the Lagrangian function reduces to the kinetic energy of the system fluid-body.
Because of relations \eqref{cond_phi:a} and \eqref{cond_phi:b}, there is a decoupling between the kinetic energy of the body due to its rigid motion and that due to its shape-changes:
$$K^b:=\frac{1}{2}m|\dot{\mathbf r}^\ast|^2+\frac{1}{2}|\omega|^2 I(\mathbf c)+\frac{1}{2}\int_{\mathcal A^\ast}|\dot\chi(\mathbf c)(\chi({\mathbf c})^{-1}(x^\ast))|^2\,{\rm d}m^\ast,$$
the last term being the kinetic energy of deformation. It can be further computed as follows:
\begin{equation}
\int_{\mathcal A^\ast}|\dot\chi(\mathbf c)(\chi({\mathbf c})^{-1}(x^\ast))|^2\,{\rm d}m^\ast=\int_{D}|\dot\chi(\mathbf c)(z)|^2\,{\rm d}m_0=\pi\rho_0\sum_{k\geq 1}\frac{|\dot c_k|^2}{k+1}.\label{moment:of:inertia}
\end{equation}
On the other hand, the kinetic energy of the fluid reads:
\begin{equation}
\label{kine:fluid}
K^f:=\frac{1}{2}\int_{\mathcal F^\ast}|\mathbf u^\ast|^2\,{\rm d}m_f^\ast=\frac{1}{2}\int_{\mathcal F^\ast}|\nabla\varphi|^2\,{\rm d}m_f^\ast.
\end{equation}
The Lagrangian function of the system fluid-amoeba next reduces to:
\begin{equation}
\label{express:lagrange}
L:=K^b+K^f,
\end{equation}
and turns out to be a function of $(\dot{\mathbf q}^\ast,\mathbf c,\dot{\mathbf c})\in({\mathbf R}^2\times{\mathbf R})\times\mathcal D\times\mathcal S$. More precisely, for any fixed $\mathbf c\in\mathcal D$, $L(\mathbf c)$ is a quadratic form in $(\dot{\mathbf q}^\ast,\dot{\mathbf c})$. It is worth remarking that it does not depend on $\mathbf r$ and $\theta$ due to the isotropy of our model with respect to the position and orientation of the body in the fluid.
\begin{rem}
As already mentioned, if we do not neglect gravity but rather assume that the body is neutrally-buoyant we have to add to our model the relation \eqref{exp:rho0} linking $\rho_f$ and $\rho_0$. However, \eqref{exp:rho0} only ensures that the upthrust is null. The torque applied on the amoeba by the buoyant force is equal to:
$$mg(\mathbf r_f-\mathbf r)^\perp\cdot \mathbf e_2,$$
where $g$ is the standard gravity and $\mathbf r_f$ is the center of buoyancy given by:
$$\mathbf r_f:=\frac{1}{m}\int_{\mathcal A}{\rm d}m_f.$$
The torque is null only when $\mathbf r_f=\mathbf r$ what is not always verified under solely assumption \eqref{exp:rho0}. 

The main consequence of taking into account buoyancy is to lose the isotropy of our model (as seen by an observer
attached to the body and, without buoyancy, all of the positions and directions in the fluid are equivalent). In this case, the Lagrangian function depends also on $\mathbf q$ and the Euler-Lagrange equations turn out to be much more involved. Models with buoyancy, several swimming bodies and bounded or partially bounded fluid domains are studied in \cite{Munnier:2008ab}. 
 \end{rem} 
\section{Euler-Lagrange equations}\label{SEC_Eumer_Lagrange}
The aim of the Section is to obtain the Euler-Lagrange equation that governs the dynamics of our system.
\subsection{Elementary potentials}\label{SEC_Regularity_potentials}
Kirchhoff's law states that the potential function can be decomposed into a linear combination of elementary potentials, each one being associated with a degree of freedom of the system (which are here: the translations of the body along $\mathbf e_j$ ($j=1,2$), the rotation and all of the elementary shape-changes governed by the variables $c_k$ ($k\geq 1$)). This law is classical when the number of degrees of freedom is finite but it must be adapted to our infinite dimensional model. This is the first goal of this subsection.

The elementary potentials, as being solutions of Neumann boundary value problems set on the fluid domain, depend implicitly on the shape of the amoeba, i.e. on the control variable $\mathbf c$. The second issue we will address in this subsection is to examine how smooth is this dependence. To carry out this task, we will use a conformal mapping $\phi(\mathbf c)$ that maps the fixed domain $\Omega$ (the exterior of the unitary disk) of the computational space onto the fluid domain $\mathcal F^\ast$ of the physical space. Each elementary potential, once composed with $\phi(\mathbf c)$ will yield a function defined in the fixed domain $\Omega$ and whose dependence with respect to $\mathbf c$ is explicit. 
\subsubsection{Definitions, Kirchhoff's law}   
\label{def:first}
We begin with decomposing $\varphi$ into the sum of $\varphi^r$ (the potential associated with the rigid motion) and $\varphi^d$ (the potential associated with the shape-changes). Both are harmonic in $\mathcal F^\ast$ and satisfy the following Neumann boundary conditions:
$\partial_n\varphi^r=(\omega x^{\ast\perp}+\dot{\mathbf r}^\ast)\cdot n$ and
$\partial_n\varphi^d=(\dot\chi(\mathbf c)(\chi(\mathbf c)^{-1}(x^\ast)))\cdot n$ on $\partial\mathcal F^\ast$.
We wish now to obtain a further decomposition of $\varphi^d$ into an infinite linear combination (i.e. a series) of elementary potentials each one associated with an elementary shape-change.
To this end, we introduce the functions $\varphi^a_k$ and $\varphi^b_k$ ($k\geq 1$), all of them being harmonic in $\mathcal F^\ast$ and which satisfy the following Neumann boundary conditions:
\begin{alignat}{3}
\partial_n\varphi^a_k&:=U^a_k(\chi(\mathbf c)^{-1}(x^\ast))\cdot n- ka_k&~&\text{ on }\partial{\mathcal F}^\ast&~&\text{ for all }k\geq 1,\label{extra:1}\\
\partial_n\varphi^b_k&:=U^b_k(\chi(\mathbf c)^{-1}(x^\ast))\cdot n- kb_k&~&\text{ on }\partial{\mathcal F}^\ast&~&\text{ for all }k\geq 1.\label{extra:2}
\end{alignat}
The extra terms $ka_k$ and $kb_k$ have to be added for the boundary data to satisfy a so-called {\it compatibility condition} necessary to ensure the well-posedness of the Neumann problems (see Subsection~\ref{neumann:boun}). From a more physical point of view, the elementary shape-changes driven by the shape variables $a_k$ and $b_k$ do not preserve the volume of the amoeba and has to be  suitably modified. Observe however that under condition \eqref{const_volume}, we recover 
\begin{equation}
\label{kirch}
\sum_{k\geq 1}\dot a_k\partial_n\varphi^a_k+\dot b_k\partial_n\varphi^b_k=\partial_n\varphi^d.
\end{equation}
The question is now: does this identity entails the equality $\sum_{k\geq 1}\dot a_k\varphi^a_k+\dot b_k\varphi^b_k=\varphi^d$? When the number of terms in the sum is finite, it is nothing but classical Kirchhoff's law. The question can hence be simplified into: what is the topology the left hand side series of functions in \eqref{kirch} has to converge for, that ensures the equality?
And the answer is: the topology of $L^2(\partial\mathcal F^\ast)$, because solutions of Neumann problems in $H^1_N(\mathcal F^\ast)$ depend linearly and continuously on their boundary data in $L^2(\partial\mathcal F^\ast)$, as detailed in the Appendix, Subsection~\ref{neumann:boun}. Actually, for any $\mathbf c\in\mathcal D$ and $\dot{\mathbf c}\in\mathcal S$, we easily confirm that the series of functions in the left hand side of \eqref{kirch} converge normally on $\partial\mathcal F^\ast$ and hence also in $L^2(\partial\mathcal F^\ast)$.

To emphasize the dependence of $\varphi^d$ with respect to $\mathbf c$ and its linear dependence with respect to $\dot{\mathbf c}$, we denote it rather 
$\langle\varphi^d(\mathbf c),\dot{\mathbf c}\rangle$.

It remains for us to introduce a decomposition for the potential $\varphi^r$. So, let us define the elementary potentials $\varphi^r_j$ ($j=1,2,3$) as being harmonic functions in $\mathcal F^\ast$ satisfying the Neumann boundary conditions:
$\partial_n\varphi^r_j=n\cdot\mathbf e_j^\ast$ ($j=1,2$) and $\partial_n\varphi^r_3=(x^\ast)^\perp\cdot n$ on $\partial\mathcal F^\ast$. We can now state:

\begin{prop}[Potentials decomposition]
\label{prop:decomp:pot}
For any allowable control (in the sense of Definition~\ref{alow:cont}), we have the identities in $H^1_N(\mathcal F^\ast)$:
\begin{subequations}
\label{decomp:prop1}
\begin{align}
\varphi^r&=\dot r_1^\ast\varphi^r_1+\dot r_2^\ast\varphi^r_2+\omega\varphi^r_3,\\
\langle\varphi^d(\mathbf c),\dot{\mathbf c}\rangle&=\sum_{k\geq 1}\dot a_k\varphi^a_k+\dot b_k\varphi^b_k,\label{decomp:prop1:1}\\
\varphi&=\varphi^r+\langle\varphi^d(\mathbf c),\dot{\mathbf c}\rangle.\label{decomp:prop1:2}
\end{align}
\end{subequations}
\end{prop}
Although it does not appear in the notation, the potentials $\varphi^r_i$ ($i=1,2,3$) and $\varphi^a_k$, $\varphi^b_k$ ($k\geq 1$) do obviously depend on $\mathbf c$ since the domain $\mathcal F^\ast$ does.
\subsubsection{Dependence of the elementary potentials with respect to $\mathbf c$}
We use complex analysis to compute  the elementary potential
functions so we identify both the computational and the physical spaces with the complex plane $\mathbf C$. As already mentioned, so as not to overload the notation we will mix the complex notation
$z=z_1+iz_2$ or $x^\ast=x^\ast_1+ix^\ast_2$ ($i^2=-1$) with the real one
$z=(z_1,z_2)^T$, $x^\ast=(x_1^\ast,x_2^\ast)^T$ and even with the polar coordinates
$(r,\theta)$, $r=|z|$ and $\theta={\rm Arg}\,(z)$ (i.e.
$z=re^{i\theta}$) when necessary. Remember that $D$ is the unitary disk of the computational space, $\Omega:=\mathbf C\setminus \bar D$ and that for all $\mathbf c\in\mathcal D$, the mapping $\phi(\mathbf c)$ is defined by \eqref{def:2}. We get the following expression for the unitary normal to $\partial{\mathcal A}^\ast$:
$$n(x^\ast):=n_1(x^\ast)+in_2(x^\ast)=-z\frac{\phi'({\mathbf c})(z)}{|\phi'({\mathbf c})(z)|},\quad(x^\ast=\phi(\mathbf c)(z),\,z\in \partial D),$$
where $\phi'(\mathbf c)$ is the complex derivative of $\phi(\mathbf c)$.
We introduce next the functions $\xi^r_j(\mathbf c)$ ($j=1,2,3$) defined by:
\begin{equation}
\label{defxiij}
\xi_j^r(\mathbf c)(z):=\varphi_j^r(x^\ast),\quad(x^\ast=\phi(\mathbf c)(z),\,z\in \Omega).
\end{equation}
Since $\phi(\mathbf c)$ is a conformal mapping, the functions $\xi_j^r(\mathbf c)$ are harmonic in $\Omega$ and we compute that:
\begin{equation}
\label{change_variable_boundary}
\frac{1}{|\phi(\mathbf c)'(z)|}\partial_n \xi_j^r(\mathbf c)(z)=\partial_n\varphi_j^r(x^\ast),\quad(x^\ast=\phi(\mathbf c)(z),\,z\in \Omega).
\end{equation}
The Neumann boundary conditions for $\varphi^r_j$ lead to the following boundary conditions for $\xi_j^r$ ($j=1,2,3$):
\begin{subequations}
\label{bound:cond:neum}
\begin{align}
\partial_n\xi^r_1(\mathbf c)(z)&=-\Re(z\phi'({\mathbf c})(z)),\\
\partial_n\xi^r_2(\mathbf c)(z)&=-\Im(z\phi'({\mathbf c})(z)),\\
\partial_n\xi^r_3(\mathbf c)(z)&=-\Im(\overline{\phi(\mathbf c)(z)}z\phi'({\mathbf c})(z)),\quad(z\in\partial D).
\end{align}
\end{subequations}
We proceed likewise for the elementary potentials related to the shape-changes: We define $\xi^a_k(\mathbf c)$ and $\xi^b_k(\mathbf c)$ for all $k\geq 1$ by:
\begin{subequations}
\label{def:xikab}
\begin{align}
\xi^a_k(\mathbf c)(z)&:=\varphi^a_k(x^\ast),\label{def:xika}\\
\xi^b_k(\mathbf c)(z)&:=\varphi^b_k(x^\ast),\quad(x^\ast=\phi(\mathbf c)(z),\,z\in\Omega).\label{def:xikb}
\end{align}
\end{subequations}
These functions are harmonic 
in $\Omega$ and satisfy the Neumann boundary conditions:
\begin{subequations}
\label{def:xi_d}
\begin{align}
\partial_n\xi^a_k(\mathbf c)(z)&=-\Re(z^{k+1}\phi'({\mathbf c})(z))- ka_k,\\
\partial_n\xi^b_k(\mathbf c)(z)&=-\Im(z^{k+1}\phi'({\mathbf c})(z))- kb_k,\quad(z\in\partial D).
\end{align}
\end{subequations}
It is clear, applying the results of Subsection~\ref{neumann:boun}, that all of these functions are well-defined in the weighted Sobolev space $H^1_N(\Omega)$. What we are interested in is to study their regularities, seen as functions of $\mathbf c$ valued in $H^1_N(\Omega)$.
We invoke again the linear-continuous dependence of the solution in $H^1_N(\Omega)$ of a Neumann boundary value problem with respect to its boundary data in $L^2(\partial D)$ and the problem is reduced to the study of the dependence of the boundary data \eqref{bound:cond:neum} and \eqref{def:xi_d} in $L^2(\partial D)$ with respect to $\mathbf c\in \mathcal D$. Some simple estimated based on the identities:
\begin{subequations}
\label{express:poten:complex}
\begin{align}
z^{k+1}\phi'({\mathbf c})(z)&=z^{k+1}-\sum_{j\geq 1-k}(j+k)\frac{c_{j+k}}{z^j},\\
\overline{\phi(\mathbf c)(z)}z\phi'(\mathbf c)(z)&=1+\sum_{k\geq 1}\bar c_k z^{k+1}-\sum_{k\geq 1}k \frac{c_k}{z^{k+1}}-\sum_{k\geq 1}\sum_{j\geq 1}k\bar c_j c_k
z^{j-k},
\end{align}
\end{subequations}
available for all $z\in\partial D$, lead us to deduce that this dependence is polynomial. It entails:
\begin{lemma}
The $\xi$-type functions defined in \eqref{defxiij} and \eqref{def:xikab}, seen as functions of $\mathbf c$ valued in $H^1_N(\Omega)$ belong to $\mathcal P(\mathcal D,H^1_N(\Omega))$.
\end{lemma}
However, we need to prove a little bit more.  Indeed, let us define also $\xi^d$ by $\xi^d(z):=\varphi^d(\phi(\mathbf c)(z))$ for all $z\in D$. Again, to emphasize the dependence of $\xi^d$ with respect to $\mathbf c$ and $\dot{\mathbf c}$, we denote it rather $\langle \xi^d(\mathbf c),\dot{\mathbf c}\rangle$. According to identity \eqref{decomp:prop1:1}, we get the decomposition in $H^1_N(\Omega)$, for all $\mathbf c\in\mathcal D$ and all $\dot{\mathbf c}\in\mathcal S$:
\begin{subequations}
\label{decomp:prop2}
\begin{equation}
\langle\xi^d(\mathbf c),\dot{\mathbf c}\rangle=\sum_{k\geq 1}\dot a_k\xi^a_k(\mathbf c)+\dot b_k\xi^b_k(\mathbf c),\label{decomp:prop2:1}
\end{equation}
and we deduce that $\xi^d$ can be seen as a function of $\mathbf c$ valued in $\mathcal L(\mathcal S,H^1_N(\Omega))$. We prove in the Appendix, Subsection~\ref{regularity_of_xid} that the function $\mathbf c\in\mathcal D\mapsto \xi^d\in \mathcal L(\mathcal S,H^1_N(\Omega))$ is also polynomial. 

Finally, let us introduce $\xi$ and $\xi^r$ defined by
$\xi(z):=\varphi(\phi(\mathbf c)(z))$ and $\xi^r(z):=\varphi^r(\phi(\mathbf c)(z))$. According to Proposition~\ref{prop:decomp:pot}, we can state that, for all $\mathbf c\in\mathcal D$ and all $\dot{\mathbf c}\in\mathcal S$:
\begin{align}
\xi^r&=\dot r_1^\ast\xi^r_1+\dot r_2^\ast\xi^r_2+\omega\xi^r_3,\label{decomp:rigiddd}\\
\xi&=\xi^r+\langle\xi^d(\mathbf c),\dot{\mathbf c}\rangle.\label{decomp:prop2:2}
\end{align}
\end{subequations}
We can now summarize all of the results obtained in this Subsection:
\begin{theorem}
\label{sum_theo}
\begin{itemize}
\item {\bf Well-posedness: } For any $\mathbf c\in{\mathcal D}$ and any $\widetilde{\mathbf c}\in\mathcal S$, the functions $\xi$, $\langle\xi^d(\mathbf c),\widetilde{\mathbf c}\rangle$, $\xi_j^r(\mathbf c)$ ($j=1,2,3$), $\xi^a_k(\mathbf c)$ and $\xi^b_k(\mathbf c)$ ($k\geq 1$) are well defined as elements of $H^1_N(\Omega)$.
\item {\bf Decomposition: }For any allowable control $(\mathbf c,\dot{\mathbf c})\in{\mathcal D}\times\mathcal S$, identities \eqref{decomp:prop2} hold.
\item {\bf Regularity: }$\xi^r_{j}\in\mathcal P(\mathcal D,H^1_N(\Omega))$ ($j=1,2,3$), $\xi^a_{k}$, $\xi^b_k\in\mathcal P(\mathcal D,H^1_N(\Omega))$ 
($k\geq 1$) and $\xi^d\in\mathcal P(\mathcal D,\mathcal L(\mathcal S,H^1_N(\Omega)))$.
\end{itemize}
\end{theorem}
The properties of conformal mappings  allow to write that 
$$K^f=\frac{1}{2}\int_{\mathcal F^\ast}|\nabla\varphi|^2{\rm d}m_f=\frac{1}{2}\int_{\Omega}|\nabla\xi|^2{\rm d}m_f^0.$$
So from now on we will refer to $\xi$ as being the potential function in place of $\varphi$ and likewise, we will call $\xi_j^r(\mathbf c)$ ($j=1,2,3$), $\xi^a_k(\mathbf c)$ and $\xi^b_k(\mathbf c)$ ($k\geq 1$), the {\it elementary potentials}.
\subsection{Mass matrices}
\label{sec:mass:matrix}
{\it Mass matrix} is a central notion in the modeling of fluid-structure interaction problems. It can be defined as the polarization of the kinetic energy of the system, seen as a quadratic form. Remember that in our case, the kinetic energy coincides with the Lagrangian function defined in Subsection~\ref{def:lagr} and that, for any fixed $\mathbf c\in\mathcal D$, $L(\mathbf c)$ is a quadratic form in $(\dot{\mathbf q}^\ast,\dot{\mathbf c})\in\mathbf R^3\times \mathcal S$. We define then $\mathbb M(\mathbf c)$ as being the bilinear symmetric form on $(\mathbf R^3\times \mathcal S)\times(\mathbf R^3\times \mathcal S)$ such that:
$$L(\mathbf c,\dot{\mathbf q}^\ast,\dot{\mathbf c})=\frac{1}{2}\langle\mathbb M(\mathbf c),(\dot{\mathbf q}^\ast,\dot{\mathbf c}),(\dot{\mathbf q}^\ast,\dot{\mathbf c})\rangle.$$
We next decompose it into $\mathbb M^r(\mathbf c)$, a bilinear symmetric form on $\mathbf R^3\times\mathbf R^3$ (that can be identified with an actual $3\times 3$ symmetric matrix), $\mathbb N(\mathbf c)$ a bilinear form on $\mathcal S\times\mathbf R^3$ and $\mathbb M^d(\mathbf c)$ a  bilinear symmetric form on $\mathcal S\times\mathcal S$ such that:
$$\langle \mathbb M(\mathbf c),(\dot{\mathbf c},\dot{\mathbf q}^\ast),(\dot{\mathbf c},\dot{\mathbf q}^\ast)\rangle=\langle \mathbb M^r(\mathbf c),\dot{\mathbf q}^\ast,\dot{\mathbf q}^\ast\rangle+\langle \mathbb M^d(\mathbf c),\dot{\mathbf c},\dot{\mathbf c}\rangle+2\langle \mathbb N(\mathbf c),\dot{\mathbf c},\dot{\mathbf q}^\ast\rangle.
$$
We are interested in finding explicit expressions for $\mathbb M^r(\mathbf c)$, $\mathbb N(\mathbf c)$ and $\mathbb M^d(\mathbf c)$ and in studying their dependence with respect to the control variable $\mathbf c$.
\subsubsection{Mass matrix related to the rigid motion}\label{SEC_Mass_Matrix_Elliptic}
We consider first $\mathbb M^r(\mathbf c)$, the classical mass matrix of the amoeba associated to its rigid motion. The decomposition of the potential function given in \eqref{decomp:rigiddd} leads us to introduce the symmetric $3\times3$ matrix: \begin{multline}
\label{def:Mr}
{\mathbb M}^{r}(\mathbf c):=\\
\begin{bmatrix}m&0&0\\
0&m&0\\
0&0&I(\mathbf c)
\end{bmatrix}+
\begin{bmatrix}
\int_{\Omega}\!\!\nabla\xi^r_{1}(\mathbf c)\cdot\nabla\xi^r_{1}(\mathbf c)\,{\rm d}m_f^0&\hspace{-0.2cm}\cdots\hspace{-0.2cm}&\int_{\Omega}\!\!\nabla\xi^r_{1}(\mathbf c)\cdot\nabla\xi^r_{3}(\mathbf c)\,{\rm d}m_f^0\\
\vdots&&\vdots\\
\int_{\Omega}\!\!\nabla\xi^r_{3}(\mathbf c)\cdot\nabla\xi^r_{1}(\mathbf c)\,{\rm d}m_f^0&\hspace{-0.2cm}\cdots\hspace{-0.2cm}&\int_{\Omega}\!\!\nabla\xi^r_{3}(\mathbf c)\cdot\nabla\xi^r_{3}(\mathbf c)\,{\rm d}m_f^0
\end{bmatrix},
\end{multline}
where we recall that $m>0$ is the constant mass of the amoeba and $I(\mathbf c)$ is its inertia momentum given in \eqref{exp:I}. The kinetic energy due to the rigid displacement of the amoeba can be written as the matrix-vectors product: $(1/2)(\dot{\mathbf r}^\ast,\omega)\mathbb M^r(\mathbf c)(\dot{\mathbf r}^\ast,\omega)^T$.
The latter matrix in the right hand side of \eqref{def:Mr} is usually referred to as an {\it added mass matrix}, relating here to the rigid motion of the animal. 

We deal now with the kinetic energy due to the shape-changes, considering separately the infinite and finite dimensional cases.
\subsubsection{Infinite dimensional case}
To define the elements of the matrices $\mathbb N(\mathbf c)$ and $\mathbb M^d(\mathbf c)$ we use the canonical basis $\{\mathbf f^1,\mathbf f^2,\mathbf f^3\}$ of $\mathbf R^3$ and the Shauder basis $\{\mathbf a^j,\,\mathbf b^j, j\geq 1\}$ of $\mathcal S$ defined in the Appendix, Subsection~\ref{banach:series}.
We have, for any $1\leq k\leq 3$ and $j\geq 1$:
\begin{subequations}
\label{matrix_entries_N}
\begin{align}
\langle \mathbb N(\mathbf c),\mathbf a^j,\mathbf f^k\rangle&:=\int_{\Omega}\nabla\xi^r_{k}(\mathbf c)\cdot\nabla\xi^a_j(\mathbf c){\rm d}m_f^0,\\
\langle \mathbb N(\mathbf c),\mathbf b^j,\mathbf f^k\rangle&:=\int_{\Omega}\nabla\xi^r_{k}(\mathbf c)\cdot\nabla\xi^b_j(\mathbf c){\rm d}m_f^0.
\end{align} 
\end{subequations}
Denoting $\delta^j_k$ the Kronecker symbol, the entries of $\mathbb M^d(\mathbf c)$ are, for all $j,k\geq 1$:
\begin{subequations}
\label{matrix_entries_Md}
\begin{align}
\langle \mathbb M^d(\mathbf c),\mathbf a^j,\mathbf a^k\rangle&:=\int_{\Omega}\nabla\xi^a_{j}(\mathbf c)\cdot\nabla\xi^a_k(\mathbf c){\rm d}m_f^0+\frac{\pi\rho_0\delta^j_k}{k+1},\\
\langle \mathbb M^d(\mathbf c),\mathbf b^j,\mathbf b^k\rangle&:=\int_{\Omega}\nabla\xi^b_{j}(\mathbf c)\cdot\nabla\xi^b_k(\mathbf c){\rm d}m_f^0+\frac{\pi\rho_0\delta^j_k}{k+1},\\
\langle \mathbb M^d(\mathbf c),\mathbf a^j,\mathbf b^k\rangle&:=\int_{\Omega}\nabla\xi^a_{j}(\mathbf c)\cdot\nabla\xi^b_k(\mathbf c){\rm d}m_f^0.
\end{align} 
\end{subequations}
\subsubsection{Finite dimensional case}
\label{finite:dim:case}
We assume now that the rate-of-shape-changes variable $\dot{\mathbf c}$ has only a finite number of non-zero elements, say the $N$ firsts ($N\geq 1$).
As explained in the Appendix, Subsection~\ref{banach:series}, in this case we introduce the projector $\Pi_N$ defined in \eqref{defpiN} and we identify $\mathcal S_N=\Pi_N(\mathcal S)$ with $\mathbf R^{2N}$. Upon this identification, the bilinear mappings $\mathbb N(\mathbf c)$ and $\mathbb M^d(\mathbf c)$ can be identified with actual matrices of sizes $3\times 2N$ and $2N\times 2N$ respectively. Thus, we have:
\begin{equation}
\label{N:finite:dim}
\mathbb N(\mathbf c):=
\begin{bmatrix}
\int_{\Omega}\nabla\xi^r_{1}(\mathbf c)\cdot\nabla\xi^a_{1}(\mathbf c)\,{\rm d}m_f^0&\hspace{-0.2cm}\cdots\hspace{-0.2cm}&\int_{\Omega}\nabla\xi^r_{1}(\mathbf c)\cdot\nabla\xi^b_{N}(\mathbf c)\,{\rm d}m_f^0\\[2mm]
\int_{\Omega}\nabla\xi^r_{2}(\mathbf c)\cdot\nabla\xi^a_{1}(\mathbf c)\,{\rm d}m_f^0&\hspace{-0.2cm}\cdots\hspace{-0.2cm}&\int_{\Omega}\nabla\xi^r_{2}(\mathbf c)\cdot\nabla\xi^b_{N}(\mathbf c)\,{\rm d}m_f^0\\[2mm]
\int_{\Omega}\nabla\xi^r_{3}(\mathbf c)\cdot\nabla\xi^a_{1}(\mathbf c)\,{\rm d}m_f^0&\hspace{-0.2cm}\cdots\hspace{-0.2cm}&\int_{\Omega}\nabla\xi^r_{3}(\mathbf c)\cdot\nabla\xi^b_{N}(\mathbf c)\,{\rm d}m_f^0
\end{bmatrix},
\end{equation}
while ${\mathbb M}^{d}(\mathbf c)$ reads:
\begin{multline}
\label{Md:finite:dim}
{\mathbb M}^{d}(\mathbf c):=\begin{bmatrix}
\int_{\Omega}\nabla\xi^a_{1}(\mathbf c)\cdot\nabla\xi^a_{1}(\mathbf c)\,{\rm d}m_f^0&\hspace{-0.2cm}\cdots\hspace{-0.2cm}&\int_{\Omega}\nabla\xi^a_{1}(\mathbf c)\cdot\nabla\xi^b_{N}(\mathbf c)\,{\rm d}m_f^0\\[2mm]
\int_{\Omega}\nabla\xi^b_{1}(\mathbf c)\cdot\nabla\xi^a_{1}(\mathbf c)\,{\rm d}m_f^0 &\hspace{-0.2cm}\cdots\hspace{-0.2cm}&\int_{\Omega}\nabla\xi^b_{1}(\mathbf c)\cdot\nabla\xi^b_{N}(\mathbf c)\,{\rm d}m_f^0\\
\vdots&&\vdots\\
\int_{\Omega}\nabla\xi^a_{N}(\mathbf c)\cdot\nabla\xi^a_{1}(\mathbf c)\,{\rm d}m_f^0&\hspace{-0.2cm}\cdots\hspace{-0.2cm}&\int_{\Omega}\nabla\xi^a_{N}(\mathbf c)\cdot\nabla\xi^b_{N}(\mathbf c)\,{\rm d}m_f^0\\[2mm]
\int_{\Omega}\nabla\xi^b_{N}(\mathbf c)\cdot\nabla\xi^a_{1}(\mathbf c)\,{\rm d}m_f^0&\hspace{-0.2cm}\cdots\hspace{-0.2cm}&\int_{\Omega}\nabla\xi^b_{N}(\mathbf c)\cdot\nabla\xi^b_{N}(\mathbf c)\,{\rm d}m_f^0
\end{bmatrix}+\\
\pi\rho_0\begin{bmatrix}
{1}/{2}&0&\ldots&0&0\\
0&{1}/{2}&\ldots&0&0\\
\vdots&&&&\vdots\\
0&0&\ldots&{1}/{N+1}&0\\
0&0&\ldots&0&{1}/{N+1}\\
\end{bmatrix}
.\end{multline}
Note that in the finite dimensional case, the overall mass matrix $\mathbb M(\mathbf c)$ can also be identified with an actual $(3+2N)\times(3+2N)$ block matrix defined by:
\begin{equation}
\label{def:matrix:M}
\mathbb M(\mathbf c):=\begin{bmatrix}\mathbb M^r(\mathbf c)&\mathbb N(\mathbf c)\\
\mathbb N(\mathbf c)^T&\mathbb M^d(\mathbf c)
\end{bmatrix}.
\end{equation}
\subsubsection{Regularity of the Mass Matrices and Lagrangian function}
As a straightforward consequence of Theorem~\ref{sum_theo} we get:
\begin{theorem}
\label{reg:matrices}
$\mathbb M^r\in\mathcal P(\mathcal D,\mathcal L_2({\mathbf R}^3\times{\mathbf R}^3))$, $\mathbb N\in\mathcal P(\mathcal D,\mathcal L_2(\mathbf R^3\times\mathcal S))$ and $\mathbb M^d\in\mathcal P(\mathcal D,\mathcal L_2(\mathcal S\times\mathcal S))$. It entails that $\mathbb M\in \mathcal P(\mathcal D,\mathcal L_2((\mathbf R^3\times\mathcal S)\times(\mathbf R^3\times\mathcal S)))$ and the Lagrangian function is smooth in all of its variables.
\end{theorem}
Notice that the mass matrix $\mathbb M^r(\mathbf c)$ is the sum of $\mathbb M^r_1(\mathbf c):={\rm diag}\,(m,m,I(\mathbf c))$ and 
$$\mathbb M^r_2(\mathbf c):=\begin{bmatrix}
\int_{\Omega}\!\!\nabla\xi^r_{1}\cdot\nabla\xi^r_{1}\,{\rm d}m_f&\cdots&\int_{\Omega}\!\!\nabla\xi^r_{1}\cdot\nabla\xi^r_{3}\,{\rm d}m_f\\
\vdots&&\vdots\\
\int_{\Omega}\!\!\nabla\xi^r_{3}\cdot\nabla\xi^r_{1}\,{\rm d}m_f&\cdots&\int_{\Omega}\!\!\nabla\xi^r_{3}\cdot\nabla\xi^r_{3}\,{\rm d}m_f
\end{bmatrix},$$
which is positive for all $\mathbf c\in\mathcal D$. We deduce that, for all $\mathbf c\in\mathcal D$:
\begin{equation}
\label{invertible:Mr}
\det\mathbb M^r(\mathbf c)\geq \det \mathbb M^r_1(\mathbf c)\geq m^2\pi\frac{\rho_0}{2},
\end{equation}
and therefore that $\mathbb M^r(\mathbf c)$ is invertible for all $\mathbf c\in\mathcal D$. Further, according to the classical formula:
\begin{equation}
\label{inverse:Mr}
\left ( \mathbb{M}^r(\mathbf{c}) \right )^{-1}=\frac{1}{\det \mathbb{M}^r(\mathbf{c})} \mathrm{co}(\mathbb{M}^r(\mathbf{c}))^T,
\end{equation}
where $\mathrm{co}(\mathbb{M}^r(\mathbf{c}))^T$ stands for the transpose of the co-matrix, we deduce from estimate \eqref{invertible:Mr} and Theorem~\ref{reg:matrices}:
\begin{prop}\label{PRO_M_Lipschitz}
The application $ \mathbf{c}\in\mathcal D\mapsto \mathbb{M}^r  (\mathbf{c} )^{-1}\in\mathcal L(\mathbf R^3,\mathbf R^3)$ is well-defined and analytic with radius of  convergence $R\geq 1$. 
\end{prop}
\subsection{Explicit computation of the mass matrices} 
\label{SEC_Expressions_mass_matrices}
\subsubsection{Elementary potentials}
In this Subsection, our aim is to compute explicitly the elementary potentials defined in Subsection~\ref{SEC_Regularity_potentials}. 
According to \eqref{bound:cond:neum}, \eqref{def:xi_d} together with expressions \eqref{express:poten:complex} and after some algebra, we get in polar coordinates:
\begin{subequations}
\begin{align}
\partial_n\xi^r_1(\mathbf c)(\theta,r)&=\sum_{j\geq 1}j\mu^1_{0,j}\cos(j\theta)+j\mu^2_{0,j}\sin(j\theta),\\
\partial_n\xi^r_2(\mathbf c)(\theta,r)&=\sum_{j\geq 1}j\nu^1_{0,j}\cos(j\theta)+j\nu^2_{0,j}\sin(j\theta),\\
\partial_n\xi^r_3(\mathbf c)(\theta,r)&=\sum_{k\geq 1}k\alpha_k^1\cos(k\theta)+k\alpha_k^2\sin(k\theta),
\end{align}
and for $k\geq 1$:
\begin{align}
\partial_n\xi^a_k(\mathbf c)(r,\theta)&=\sum_{j\geq 1}j\mu^1_{k,j}\cos(j\theta)+j\mu^2_{k,j}\sin(j\theta),\\
\partial_n\xi^b_k(\mathbf c)(r,\theta)&=\sum_{j\geq 1}j\nu^1_{k,j}\cos(j\theta)+j\nu^2_{k,j}\sin(j\theta),
\end{align}
where the sequences of real numbers $(\alpha_k^1)_{k\geq 1}$ and $(\alpha_k^2)_{k\geq 0}$ are defined by:
\begin{align}
\alpha^1_1&=\sum_{j\geq 1}b_{j+1}a_j-a_{j+1}b_j,\\
\alpha^1_k&=b_{k-1}+\sum_{j\geq 1}b_{j+k}a_j-a_{j+k}b_j,\quad (k\geq 2),
\end{align}
and
\begin{align}
\alpha^2_1&=-\sum_{j\geq 1}a_{j+1}a_j+b_{j+1}b_j,\\
\alpha^2_k&=-a_{k-1}-\sum_{j\geq 1}a_{j+k}a_j+b_{j+k}b_j,\quad (k\geq 2),
\end{align}
\end{subequations}
and the sequences $(\mu^l_{k,j})_{j\geq 1}$ and $(\mu^l_{k,j})_{j\geq 1}$ ($l=1,2$, $k\geq 1$) by:
\begin{subequations}
\begin{align}
\mu_{k,j}^1&=\begin{cases}
(k/j+1)a_{k+j}+(k/j-1)a_{k-j}&\text{if }1\leq j\leq k-1,\\
(k/(k+1)+1)a_{2k+1}-1/(k+1)&\text{if }j=k+1,\\
(k/j+1)a_{k+j}&\text{if }j=k\text{ or }j\geq k+2,
\end{cases}\\
\mu_{k,j}^2&=\begin{cases}
(k/j+1)b_{k+j}-(k/j-1)b_{k-j}&\text{if }1\leq j\leq k-1,\\
(k/j+1)b_{k+j}&\text{if }j\geq k,
\end{cases}
\intertext{and}
\nu_{k,j}^1&=\begin{cases}
(k/j+1)b_{k+j}+(k/j-1)b_{k-j}&\text{if }1\leq j\leq k-1,\\
(k/j+1)b_{k+j}&\text{if }j\geq k,
\end{cases}\\
\nu_{k,j}^2&=\begin{cases}
-(k/j+1)a_{k+j}+(k/j-1)a_{k-j}&\text{if }1\leq j\leq k-1,\\
-(k/(k+1)+1)a_{2k+1}-1/(k+1)&\text{if }j=k+1,\\
-(k/j+1)a_{k+j}&\text{if }j=k\text{ or }j\geq k+2.
\end{cases}
\end{align}
\end{subequations}
We deduce the following expressions for the elementary potentials defined in $\Omega:={\mathbf R}^2\setminus\bar D$  in polar coordinates:
\begin{subequations}
\label{exp:poten}
\begin{align}
\xi^r_1(\mathbf c)(r,\theta)&=\sum_{j\geq 1}(\mu_{0,j}^1\cos(j\theta)+\mu_{0,j}^2\cos(j\theta))r^{-j},\\
\xi^r_2(\mathbf c)(r,\theta)&=\sum_{j\geq 1}(\nu_{0,j}^1\cos(j\theta)+\nu_{0,j}^2\cos(j\theta))r^{-j},\\
\xi^r_3(\mathbf c)(r,\theta)&=\sum_{j\geq 1}(\alpha_j^1\cos(j\theta)+\alpha_j^2\cos(j\theta))r^{-j},\\
\xi^a_k(\mathbf c)(r,\theta)&=\sum_{j\geq 1}(\mu_{k,j}^1\cos(j\theta)+\mu_{k,j}^2\cos(j\theta))r^{-j},\\
\xi^b_k(\mathbf c)(r,\theta)&=\sum_{j\geq 1}(\nu_{k,j}^1\cos(j\theta)+\nu_{k,j}^2\cos(j\theta))r^{-j}.
\end{align}
\end{subequations}
From now on, we will denote, for any $k\geq 0$, $\boldsymbol\mu_k:=((\mu^1_{k,j})_{j\geq 1},(\mu^2_{k,j})_{j\geq 1})$ (a pair of real sequences) and 
likewise $\boldsymbol\nu_k:=((\nu^1_{k,j})_{j\geq 1},(\nu^2_{k,j})_{j\geq 1})$ and $\boldsymbol\alpha:=((\alpha^1_j)_{j\geq 1},(\alpha^2_{j})_{j\geq 1})$.
\subsubsection{Mass matrices}
\label{SEC_Expressions_of_the_mass_matrices}
The entries of the mass matrices defined in Subsection~\ref{sec:mass:matrix} can be now easily derived from the expressions \eqref{exp:poten} of the elementary potentials. Indeed,
let us consider, for instance, the first element of the matrix $\mathbb M^r(\mathbf c)$. Applying Green's formula, we get:
$$\int_{\Omega}\nabla\xi^r_1(\mathbf c)\cdot\nabla\xi^r_1(\mathbf c)\,{\rm d}m_f=
-\rho_f\int_{\partial D}\xi^r_1(\mathbf c)\frac{\partial\xi^r_1}{\partial r}(\mathbf c)\,{\rm d}\sigma,$$
and then:
$$\int_{\Omega}\nabla\xi^r_1(\mathbf c)\cdot\nabla\xi^r_1(\mathbf c)\,{\rm d}m_f=\pi\rho_f\sum_{j\geq 1}{j}|\mu_{0,j}|^2.$$
For any two pair of sequences $\boldsymbol\upsilon:=((\upsilon^1_j)_{j\geq 1},(\upsilon^2_j)_{j\geq 1})$ and $\boldsymbol\varsigma:=((\varsigma^1_j)_{j\geq 1},(\varsigma^2_j)_{j\geq 1})$ of real numbers, we introduce the notation:
$$\boldsymbol\upsilon\cdot\boldsymbol\varsigma:=\sum_{k\geq 1}{k}(\upsilon^1_k\varsigma^1_k+\upsilon^2_k\varsigma^2_k)\quad\text{and}\quad
|\boldsymbol\upsilon|^2:=\boldsymbol\upsilon\cdot\boldsymbol\upsilon.$$
Taking into account the expressions \eqref{def:Mr} and \eqref{exp:I}, it allows us to give the expression of the mass matrices in a convenient short form:
\begin{equation}
\label{def:matrix_mr_finite}
\mathbb M^r(\mathbf c)=\\
\rho_0\pi
\begin{bmatrix}
1&0&0\\
0&1&0\\
0&0&\ds\frac{1}{2}+\sum_{k\geq 1}\frac{|c_k|^2}{k+1}
\end{bmatrix}+
\rho_f\pi
\begin{bmatrix}
|\boldsymbol\mu_0|^2&\boldsymbol\mu_0\cdot\boldsymbol\nu_0&\boldsymbol\mu_0\cdot\boldsymbol\alpha\\
\boldsymbol\mu_0\cdot\boldsymbol\nu_0&|\boldsymbol\nu_0|^2&\boldsymbol\nu_0\cdot\boldsymbol\alpha\\
\boldsymbol\mu_0\cdot\boldsymbol\alpha&\boldsymbol\nu_0\cdot\boldsymbol\alpha&|\boldsymbol\alpha|^2
\end{bmatrix},
\end{equation}
and likewise the elements \eqref{matrix_entries_N} of the mass matrix $\mathbb N$ read, for all $k\geq 1$:
\begin{align}
\langle \mathbb N(\mathbf c),\mathbf a^k,\mathbf f^j\rangle&=\begin{cases}\rho_f\pi \boldsymbol\mu_0\cdot\boldsymbol\mu_k&\text{if }j=1\\
\rho_f\pi \boldsymbol\nu_0\cdot\boldsymbol\mu_k&\text{if }j=2\\
\rho_f\pi \boldsymbol\alpha\cdot\boldsymbol\mu_k&\text{if }j=3,
\end{cases}\\
\langle \mathbb N(\mathbf c),\mathbf b^k,\mathbf f^j\rangle&=\begin{cases}\rho_f\pi \boldsymbol\mu_0\cdot\boldsymbol\nu_k&\text{if }j=1\\
\rho_f\pi \boldsymbol\nu_0\cdot\boldsymbol\nu_k&\text{if }j=2\\
\rho_f\pi \boldsymbol\alpha\cdot\boldsymbol\nu_k&\text{if }j=3.
\end{cases}
\end{align}
At last, the expressions of the elements \eqref{matrix_entries_Md} of $\mathbb M^d(\mathbf c)$ reads, for all $j,k\geq 1$:
\begin{subequations}
\begin{align}
\langle \mathbb M^d(\mathbf c),\mathbf a^j,\mathbf a^k\rangle&=\rho_f\pi \boldsymbol\mu_j\cdot\boldsymbol\mu_k+\frac{\pi\rho_0\delta^j_k}{k+1},\quad \langle \mathbb M^d(\mathbf c),\mathbf a^j,\mathbf b^k\rangle(\mathbf c)=\rho_f\pi \boldsymbol\mu_j\cdot\boldsymbol\nu_k,\\
\langle \mathbb M^d(\mathbf c),\mathbf b^j,\mathbf b^k\rangle&=\rho_f\pi \boldsymbol\nu_j\cdot\boldsymbol\nu_k+\frac{\pi\rho_0\delta^j_k}{k+1}.
\end{align}
\end{subequations}
In the finite dimensional case (i.e $\dot{\mathbf c}$ has only a finite number of non-zero elements, the $N$ first, $N\geq 1$), treated in Paragraph \ref{finite:dim:case}, the expressions \eqref{N:finite:dim} and \eqref{Md:finite:dim} turn out to be:
\begin{equation}\label{def:matrix_N_finite}
\mathbb N(\mathbf c)=\rho_f\pi
\begin{bmatrix}
\boldsymbol\mu_0\cdot\boldsymbol\mu_1&\boldsymbol\mu_0\cdot\boldsymbol\nu_1&\ldots&\boldsymbol\mu_0\cdot\boldsymbol\mu_N&\boldsymbol\mu_0\cdot\boldsymbol\nu_N\\
\boldsymbol\nu_0\cdot\boldsymbol\mu_1&\boldsymbol\nu_0\cdot\boldsymbol\nu_1&\ldots&\boldsymbol\nu_0\cdot\boldsymbol\mu_N&\boldsymbol\nu_0\cdot\boldsymbol\nu_N\\
\boldsymbol\alpha\cdot\boldsymbol\mu_1&\boldsymbol\alpha\cdot\boldsymbol\nu_1&\ldots&\boldsymbol\alpha\cdot\boldsymbol\mu_N&\boldsymbol\alpha\cdot\boldsymbol\nu_N
\end{bmatrix},
\end{equation}
and
\begin{multline}
\label{def:matrix_Md_finite}
\mathbb M^d(\mathbf c):=
\rho_f\pi\begin{bmatrix}
|\boldsymbol\mu_1|^2\hspace{-2mm}&\boldsymbol\mu_1\cdot\boldsymbol\nu_1&\hspace{-4mm}\ldots\hspace{-4mm}&\boldsymbol\mu_1\cdot\boldsymbol\mu_N&\hspace{-2mm}\boldsymbol\mu_1\cdot\boldsymbol\nu_N\\
\boldsymbol\nu_1\cdot\boldsymbol\mu_1\hspace{-2mm}&|\boldsymbol\nu_1|^2&\hspace{-4mm}\ldots\hspace{-4mm}&\boldsymbol\nu_1\cdot\boldsymbol\mu_N&\hspace{-2mm}\boldsymbol\nu_1\cdot\boldsymbol\nu_N\\
\vdots\hspace{-2mm}&\vdots&&\vdots&\hspace{-2mm}\vdots\\
\boldsymbol\mu_N\cdot\boldsymbol\mu_1&\boldsymbol\mu_N\cdot\boldsymbol\nu_1\hspace{-2mm}&\hspace{-4mm}\ldots\hspace{-4mm}&|\boldsymbol\mu_N|^2&\hspace{-2mm}\boldsymbol\mu_N\cdot\boldsymbol\nu_N\\
\boldsymbol\nu_N\cdot\boldsymbol\mu_1&\boldsymbol\nu_N\cdot\boldsymbol\nu_1\hspace{-2mm}&\hspace{-4mm}\ldots\hspace{-4mm}&\boldsymbol\nu_N\cdot\boldsymbol\mu_N&\hspace{-2mm}|\boldsymbol\nu_N|^2
\end{bmatrix}
+\\
\pi\rho_0\begin{bmatrix}
{1}/{2}&\hspace{-2mm}0&\hspace{-4mm}\ldots\hspace{-4mm}&0&\hspace{-2mm}0\\
0&\hspace{-2mm}1/2&\hspace{-4mm}\ldots\hspace{-4mm}&0&\hspace{-2mm}0\\
\vdots&&&&\vdots\\
0&\hspace{-2mm}0&\hspace{-4mm}\ldots\hspace{-4mm}&1/N+1&\hspace{-2mm}0\\
0&\hspace{-2mm}0&\hspace{-4mm}\ldots\hspace{-4mm}&0&\hspace{-2mm}{1}/{N+1}\\
\end{bmatrix}.
\end{multline}
\subsubsection{A special case, $N=2$}
\label{SEC_special_case}
We specify $N=2$ and we assume that both $\mathbf c$ and $\dot{\mathbf c}$ have only two non-zero elements: $c_1=a_1+i b_1$ and $c_2=a_2+i b_2$.  
The quantities arising in the expression of 
the matrix $\mathbb M^r(\mathbf c)$ are in this case:
\begin{align*}
|\boldsymbol\mu_0|^2&=(1-a_1)^2+(b_1)^2+2(a_2) ^2+2(b_2)^2,\\
\boldsymbol\mu_0\cdot\boldsymbol\nu_0&=-2b_1,\\
\boldsymbol\mu_0\cdot\boldsymbol\alpha&=3(a_2b_1-a_1b_2)-2a_1b_1a_2+b_2[(a_1)^2-(b_1)^2],\\
|\boldsymbol\nu_0|^2&=(1+a_1)^2+(b_1)^2+2(a_2) ^2+2(b_2)^2,\\
\boldsymbol\nu_0\cdot\boldsymbol\alpha&=3(b_1b_2+a_2a_1)+2a_1b_1b_2+a_2[(a_1)^2-(b_1)^2],\\
|\boldsymbol\alpha|^2&=[(a_1)^2+(b_1)^2][(a_2)^2+(b_2)^2]+2(b_1)^2+2(a_1)^2+3(b_2)^2+3(a_2)^2,
\end{align*}
while the elements of the matrix $\mathbb N(\mathbf c)$ read:
$$\begin{array}{rlrl}
\boldsymbol\mu_0\cdot\boldsymbol\mu_1&=2(a_2a_1+b_2b_1)-3a_2,&
\boldsymbol\mu_0\cdot\boldsymbol\nu_1&=2(b_2a_1-a_2b_1)-3b_2,\\
\boldsymbol\mu_0\cdot\boldsymbol\mu_2&=-a_1+(a_1)^2-(b_1)^2,&
\boldsymbol\mu_0\cdot\boldsymbol\nu_2&=-b_1+2a_1b_1,\\
\boldsymbol\nu_0\cdot\boldsymbol\mu_1&=2(a_2b_1-b_2a_1)-3b_2,&
\boldsymbol\nu_0\cdot\boldsymbol\nu_1&=2(b_2b_1+a_2a_1)+3a_2,\\
\boldsymbol\nu_0\cdot\boldsymbol\mu_2&=b_1+2a_1b_1,&
\boldsymbol\nu_0\cdot\boldsymbol\nu_2&=(b_1)^2-(a_1)^2-a_1,\\
\boldsymbol\alpha\cdot\boldsymbol\mu_1&=-b_1-2b_1[(a_2)^2+(b_2)^2],&
\boldsymbol\alpha\cdot\boldsymbol\nu_1&=a_1+2a_1[(b_2)^2+(a_2)^2],\\
\boldsymbol\alpha\cdot\boldsymbol\mu_2&=-b_2+b_2[(a_1)^2+(b_1)^2],&
\boldsymbol\alpha\cdot\boldsymbol\nu_2&=a_2-a_2[(b_1)^2+(a_1)^2].
\end{array}$$
Identity \eqref{conserve_mass} together with \eqref{det} allow computation of the expression of the density 
$\rho^\ast_{\mathbf c}$ in polar coordinates:
\begin{multline*}\rho^\ast_\mathbf c(\chi(\mathbf c,r,\theta))=\rho_0\Big[1-(a_1)^2-(b_1)^2-4(a_1a_2+b_1b_2)r\cos(\theta)\\
+4(b_1a_2-a_1b_2)r\sin(\theta)-4[(b_2)^2+(a_2)^2]r^2\Big]^{-1}.
\end{multline*}
Notice that this quantity is not required to compute the motion of the amoeba. Finally, to compute the internal forces of the swimming body, as it will be shown in Subsection~\ref{expre:internal:forces}, we need the expression of the elements of $\mathbb M^d(\mathbf c)$. We give only the non-zero elements:
$$\begin{array}{rlrl}
|\boldsymbol\mu_1|^2&=4[(a_2)^2+(b_2)^2]+\ds\frac{1}{2},&
\boldsymbol\mu_1\cdot\boldsymbol\mu_2&=2(a_2a_1-b_2b_1),\\
\boldsymbol\mu_1\cdot\boldsymbol\nu_2&=2(a_2b_1+a_1b_2),&
|\boldsymbol\mu_2|^2&=(a_1)^2+(b_1)^2+\ds\frac{1}{3},\\
\boldsymbol\mu_2\cdot\boldsymbol\nu_1&=2(a_1b_2+b_1a_2),&
|\boldsymbol\nu_1|^2&=4[(b_2)^2+(a_2)^2]+\ds\frac{1}{2},\\
\boldsymbol\nu_1\cdot\boldsymbol\nu_2&=2(b_1b_2-a_2a_1),&
|\boldsymbol\nu_2|^2&=(b_1)^2+(a_1)^2+\ds\frac{1}{3}.
\end{array}$$
\subsection{Equation of motion}
\label{subsec:equationofmotion}
Following the method explained in \cite[chap VI, pages 160-201]{Lamb:1993aa}, we introduce $\mathbf P$ and $\Pi$, the translational and angular  {\it impulses}, as well as  $\mathbf L$ and $\Lambda$, the impulses relating to the deformations:
$$\begin{bmatrix}
\mathbf P\\
\Pi
\end{bmatrix}
:=\mathbb M^r(\mathbf c)\begin{bmatrix}\dot{\mathbf r}^\ast\\ \omega\end{bmatrix}\qquad
\text{and}\qquad
\begin{bmatrix}
\mathbf L\\
\Lambda
\end{bmatrix}:=\langle\mathbb N(\mathbf c),
\dot{\mathbf c}\rangle.
$$
In these identities, both left hand side terms can be identified with elements of $\mathbf R^3$. 
We compute that for all $\dot{\mathbf p}:=(\dot{\mathbf s},\widetilde\omega)^T\in \mathbf R^3$:
$$\frac{d}{dt}\frac{\partial}{\partial\dot{\mathbf q}}\begin{bmatrix}\dot{\mathbf r}^\ast  \\ \omega \end{bmatrix}\cdot\dot{\mathbf p}
-\frac{\partial}{\partial { \mathbf q}}\begin{bmatrix} \dot{\mathbf r}^\ast  \\ \omega\end{bmatrix}\cdot \dot{\mathbf p}=
\begin{bmatrix} \widetilde\omega(\dot{\mathbf r}^\ast)^\perp-\omega(\dot{\mathbf s}^\ast)^\perp \\ 0\end{bmatrix}.$$
We next easily obtain that:
$$
\frac{d}{dt}\frac{\partial L}{\partial \dot{\mathbf q}}\cdot\dot{\mathbf p}-\frac{\partial L}{\partial \mathbf q}\cdot  \dot{\mathbf p}=
\frac{d}{dt}\begin{bmatrix}
\mathbf P+\mathbf L\\
\Pi+\Lambda \\ 
\end{bmatrix}\cdot
\begin{bmatrix}\dot{\mathbf s}^\ast\\
\widetilde\omega\end{bmatrix}
+\begin{bmatrix}
\mathbf P+\mathbf L\\
\Pi+\Lambda \\ 
\end{bmatrix}\cdot
\begin{bmatrix} \widetilde\omega(\dot{\mathbf r}^\ast)^\perp-\omega(\dot{\mathbf s}^\ast)^\perp \\ 0\end{bmatrix}.
$$
According to Theorem~\ref{reg:matrices}, the Lagrangian function is {\it smooth} with respect to all of its variables, allowing all of the derivatives to be computed.
Invoking the {\it least action principle}, the Euler-Lagrange equation of motion is (see e.g. \cite[Theorem 7.3.3 page 187]{Marsden:1999aa}):
$$\frac{d}{dt}\frac{\partial L}{\partial \dot{\mathbf q}}\cdot \dot{\mathbf p}-\frac{\partial L}{\partial\mathbf q}\cdot \dot{\mathbf p}=0,
\quad\forall\, \dot{\mathbf p}\in \mathbf R^3.$$
In our case we get:
\begin{subequations}
\label{syst:1}
\begin{align}
\frac{d}{dt}(\mathbf P+\mathbf L)+\omega(\mathbf P+\mathbf L)^\perp&=0,\\
\frac{d}{dt}(\Pi+\Lambda )-\dot{\mathbf r}^\ast\cdot(\mathbf P+\mathbf L)^\perp&=0.
\end{align}
\end{subequations}
If at time $t=0$:
$$\begin{bmatrix}
\mathbf P+\mathbf L\\
\Pi+\Lambda \\ 
\end{bmatrix}=0,$$
this relation remains true for all $t>0$ since the Cauchy-Lipschitz Theorem ensures the uniqueness of the solution of system \eqref{syst:1}.
We obtain here the equation:
\begin{equation}
\label{equation_of_motion}
\begin{bmatrix}\dot{\mathbf r}^\ast\\ \omega\end{bmatrix}=-(\mathbb M^r(\mathbf c))^{-1}\langle\mathbb N(\mathbf c),\dot {\mathbf c}\rangle.
\end{equation}
 If we introduce, for all $t\geq 0$:
 \begin{equation}
 \label{chgt:var:1}
 \widetilde{\mathbf r}(t):=\dot{\mathbf r}^\ast_0+\int_0^t \dot{\mathbf r}^\ast(s){\rm d}s,
 \end{equation}
 then we can rewrite \eqref{equation_of_motion} as a first order ODE:
 \begin{subequations}
 \label{edo:var}
 \begin{equation}
 \label{edo:edo}
 \frac{d}{dt}
 \begin{bmatrix}
 \widetilde{\mathbf r}\\
\theta
\end{bmatrix}=
 -(\mathbb M^r(\mathbf c))^{-1}\langle\mathbb N(\mathbf c),\dot {\mathbf c}\rangle,\qquad(t>0).
\end{equation}
This expression is the one obtained in \cite{Kanso:2005aa} for articulated bodies. This expression is very convenient to study the motion of the shape-changing body since it gives the 
velocity with respect to the shape variable. Due to the change of variables \eqref{chgt:var:1}, it has to be supplemented with a so-called {\it reconstruction equation} allowing to recover $\mathbf r$ knowing $\theta$:
\begin{equation}
\label{reconst}
 \mathbf r(t)=\mathbf r_0-\int_0^tR(\theta)(\mathbb M^r(\mathbf c))^{-1}\langle\mathbb N(\mathbf c),\dot {\mathbf c}\rangle{\rm d}s.
 \end{equation}
\end{subequations}
We can also easily give the equation of motion in terms of $\mathbf r$ and $\theta$. To this purpose, we introduce the $3\times 3$ block matrix:
$$\mathcal R(\theta):=\begin{bmatrix}
R(\theta)&0\\
0&1
\end{bmatrix},$$
and since $\dot{\mathbf r}^\ast=R(\theta)^T\mathbf r$, we can rewrite \eqref{equation_of_motion} in the form:
\begin{equation}
\label{equation_of_motion:1}
\frac{d}{dt}\begin{bmatrix}{\mathbf r}\\ \theta\end{bmatrix}=-\mathcal R(\theta)(\mathbb M^r(\mathbf c))^{-1}\langle\mathbb N(\mathbf c),\dot {\mathbf c}\rangle,\qquad(t>0).
\end{equation}
\subsection{Mathematically allowable control function}
\label{subsection:finite:dim}
According to Definition~\ref{alow:cont}, a control function $t\in[0,T]\mapsto\mathbf c(t)\in\mathcal S$ to be physically allowable has to satisfy the constraints $\mathbf c(t)\in\mathcal D$, $\|\mathbf c(t)\|_{\mathcal T}=\mu$ ($\mu<1$) and  identity \eqref{C1:2} for all $t>0$.
However, when $\mathbf c$ has only a finite number of non-zero elements (the $N$ firsts, $N>1$), we observe that the expressions \eqref{def:matrix_mr_finite},  \eqref{def:matrix_N_finite} and \eqref{def:matrix_Md_finite} of the mass matrices make sense even if $\mathbf c\notin\mathcal D$ and their entries  are still polynomial functions in $\mathbf c$. Likewise, the matrix $\mathbb M^r(\mathbf c)$ is invertible for all $\mathbf c\in\mathcal S_N$ and the entries of $\mathbb M^r(\mathbf c)^{-1}$ are analytic functions in $\mathbf c$, with infinite radii of convergence. 

However, when $\mathbf c\in\mathcal D$, the mappings $\chi(\mathbf c)$ and $\phi(\mathbf c)$ may be no longer invertible and therefore the domains $\mathcal A^\ast$ and $\mathcal F^\ast$ are ill-defined (they overlap themselves). The elementary potentials cannot be defined, either. But since we can consider expressions \eqref{def:matrix_mr_finite},  \eqref{def:matrix_N_finite} and \eqref{def:matrix_Md_finite} as defining abstract matrices (not relating any longer to our physical problem) for any $\mathbf c\in\mathcal S_N$, we can also consider the ODE \eqref{equation_of_motion:1} in this case. It leads us to relax the constraint $\mathbf c\in\mathcal D$ in the finite dimensional case and to introduce for all $\mu>0$ and all integers $N>0$:
$$\mathcal E_N(\mu):=\{\mathbf c\in\mathcal S_N\,:\,\|\mathbf c\|_{\mathcal T_N}=\mu\}.$$
We can next set:
\begin{definition}[Mathematically allowable control function]
\label{allowable_finite}
A continuous piecewise $\mathcal C^1$ function $t\in[0,T]\mapsto\mathbf c(t)\in\mathcal S_N$ (for some real positive $T$ and positive integer $N$) is said to be mathematically allowable when:
\begin{itemize}
\item There exists $\mu>0$ such that $\mathbf c(t)\in \mathcal E_N(\mu)$ for all $t\geq 0$.
\item Constraint \eqref{C1:2} is satisfied for all $t\in]0,T[$ such that $\dot{\mathbf c}(t)$ is well-defined (as piecewise $\mathcal C^1$ function, $\dot{\mathbf c}$ is well-defined for all $t \in]0,T[$ but a finite number).
\end{itemize}
\end{definition}
\begin{figure}[H]     
     \centering
     \begin{tabular}{|c|c|c|c|c|}
     \hline
     \subfigure
     {\includegraphics[width=.2\textwidth]{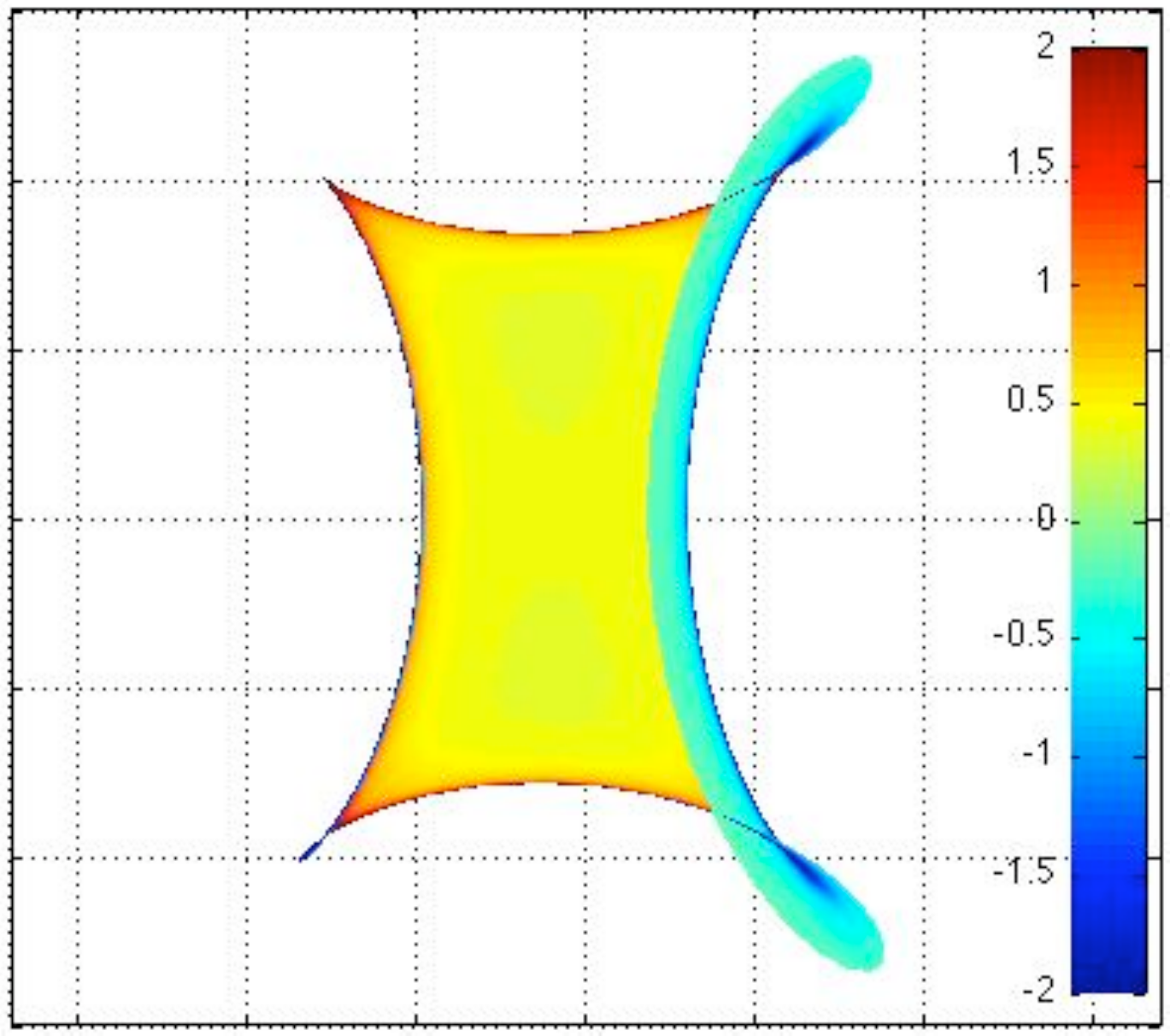}}
     &
     \subfigure
     {\includegraphics[width=.2\textwidth]{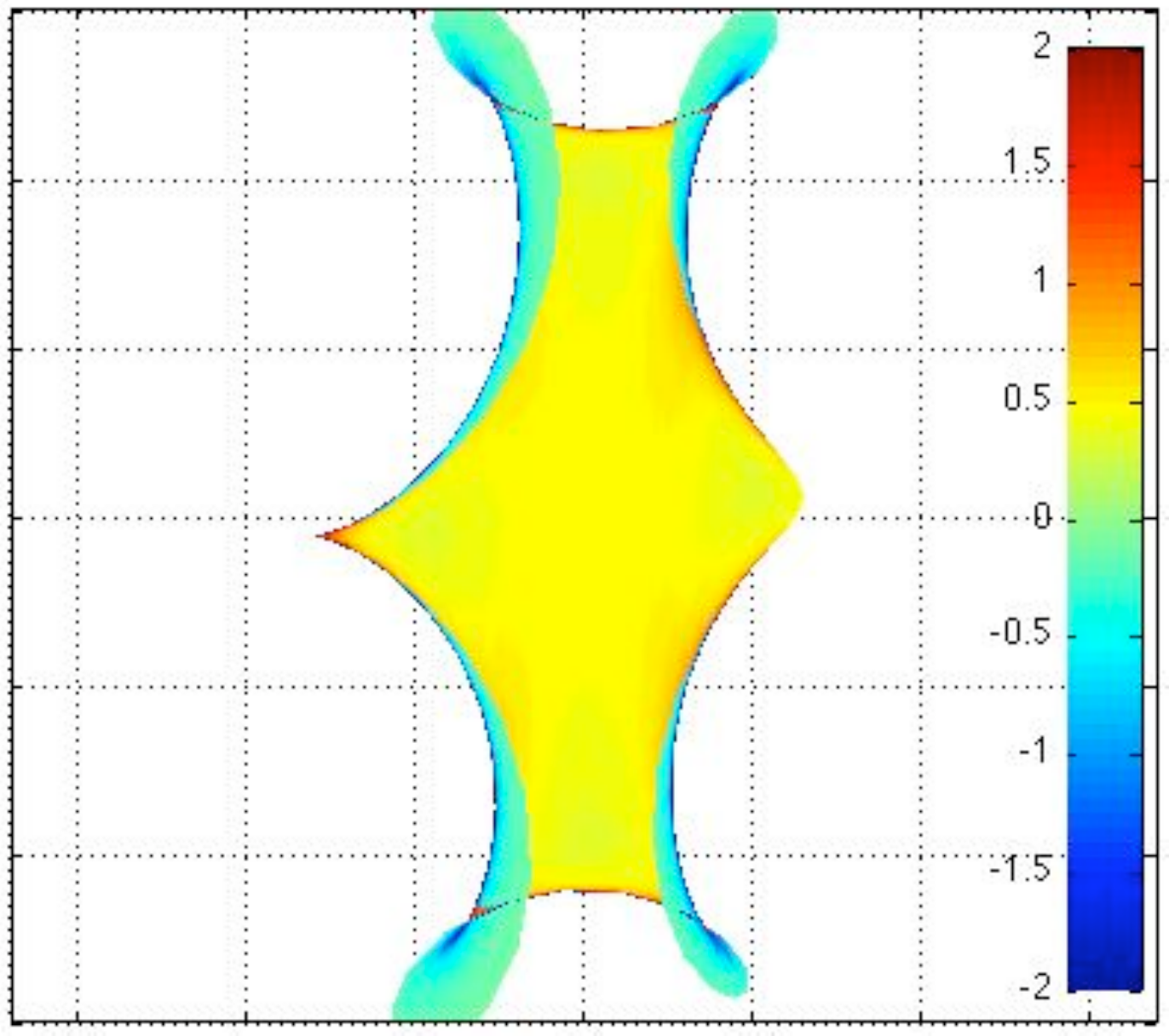}}
     &
      \subfigure
     {\includegraphics[width=.2\textwidth]{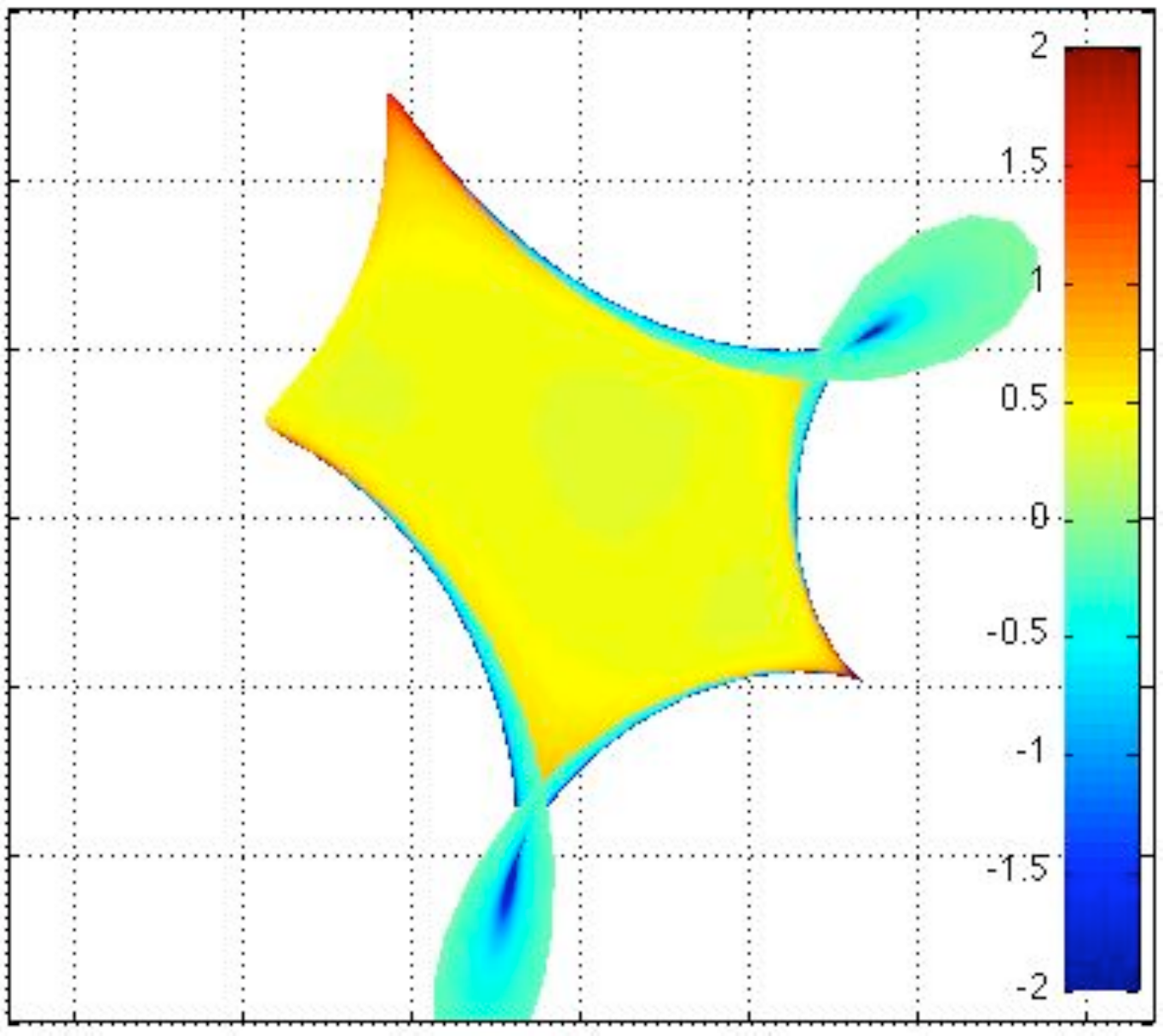}}
     &
     \subfigure
     {\includegraphics[width=.2\textwidth]{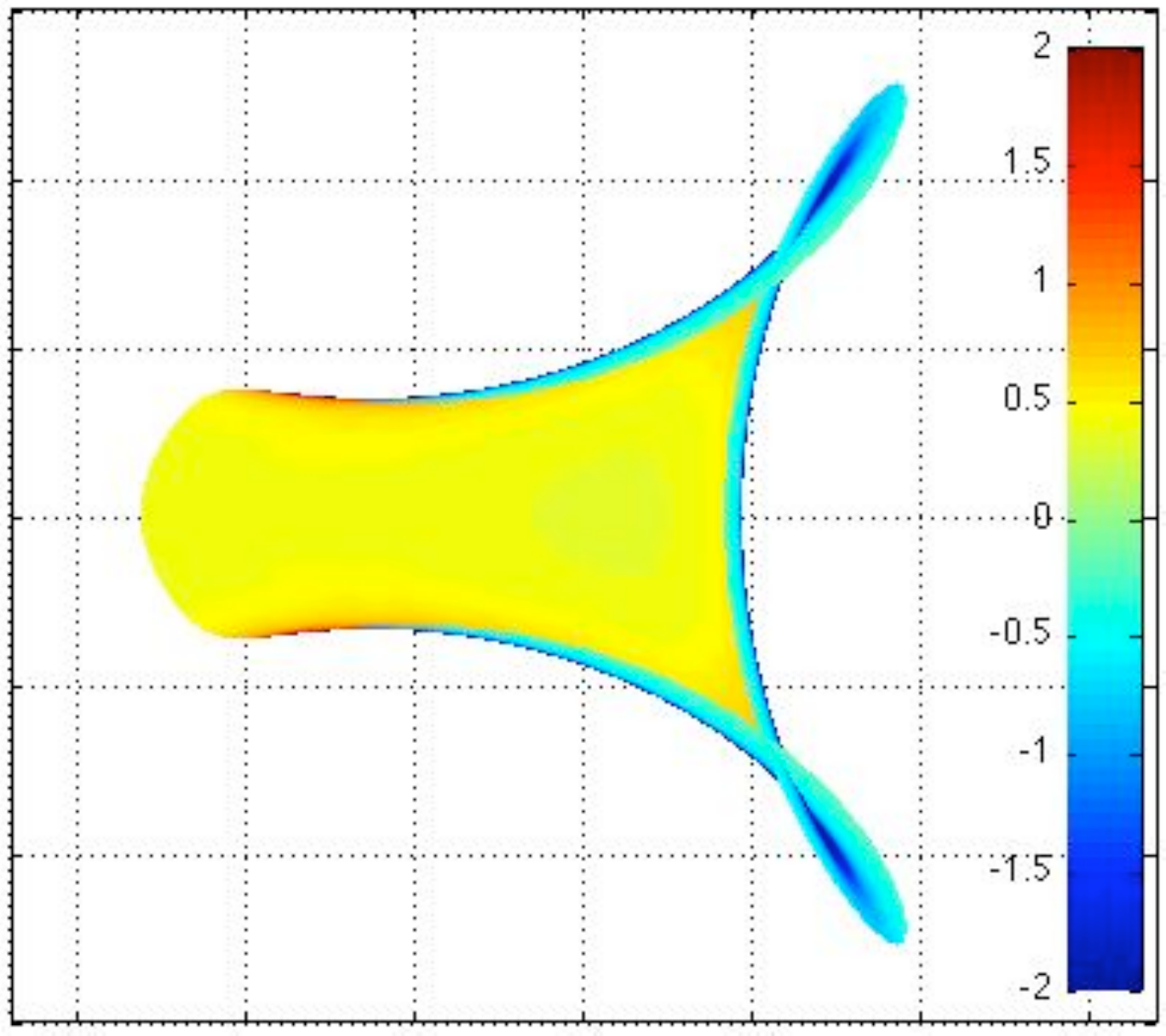}}\\
     \hline
\end{tabular}
\caption{Examples of mathematically allowable shape-changes that are not physically allowable ($\mu=0.8$, $\mathbf c\notin\mathcal D$). The domain $\mathcal S^\ast$ overlaps itself. The colors gives the value of the density inside the animal. Observe that we can have negative densities but the total mass of the animal is always positive and constant.}
\end{figure}

\subsection{Well-posedness}
From Proposition \ref{PRO_M_Lipschitz} and Theorem~\ref{reg:matrices}, we deduce:
\begin{prop}\label{PRO_well_posedness_eq_of_motion}
For any $\mu$ in $]0,1[$, for any smooth physically allowable control function $\mathbf{c}:[0,T]\rightarrow {\cal E}^\bullet(\mu)$ and for any initial condition $(\mathbf{r}_0,\theta_0)\in\mathbf R^2\times\mathbf R/2\pi$, there exists one unique smooth solution to Equations \eqref{edo:var} (or equivalently  \eqref{equation_of_motion:1}) defined on $[0,T]$.

We get the same result, replacing the physically allowable control by a mathematically allowable control function $\mathbf{c}:[0,T]\rightarrow {\cal E}_N(\mu)$ for any integer $N>1$.
\end{prop}
A mathematically allowable control function $\mathbf c$ is only assumed to be continuous and piecewise $\mathcal C^1$. It means that there exist $t_0=0<t_1\ldots<t_n=T$ such that $\mathbf c$ be $\mathcal C^1$ on each interval $]tk,t_{k+1}[$, $k=0,\ldots,n-1$. The solution given in Proposition~\ref{PRO_well_posedness_eq_of_motion} is obtained by integrating the EDO \eqref{edo:var} on each interval $]t_k,t_{k+1}[$. It is also continuous, piecewise $\mathcal C^1$ on $]0,T[$.
\subsection{Controlling with internal forces}
\label{SEC:control_with_internal_forces}

We have selected the shape of the amoeba to be the given quantity for controlling the motion of the animal. By defining the notion of physically {\it allowable} shape-changes (see Definition~\ref{alow:cont}), we took care that these deformations result from the work of internal forces only. In this Subsection, we will first compute the expression of these internal forces in terms of the given shape-changes. 
Second, we will focus on the situation where the control variable $\mathbf c$ lives in the finite  $N$-dimensional vector space ($N\geq 1$) $S_N$ (defined in the Appendix, Subsection~\ref{banach:series}). We will prove that it is immaterial whether we select either the internal forces or the shape-changes as controls, the relation linking them being one-to-one. However, we will also show that the notion of {\it allowable} internal forces is much more difficult to define than the notion of allowable shape-changes.
\subsubsection{Expression of the internal forces}
\label{expre:internal:forces}
First, the shape-changes being given, we are interested in computing the internal forces they result from. Let $t\in\mathbf [0,T]\mapsto\mathbf c(t)\in\mathcal E^\bullet(\mu)$ ($\mu\in]0,1[$) be any  smooth allowable control in the sense of Definition~\ref{alow:cont} and compute, using the Euler-Lagrange equation \eqref{equation_of_motion:1} the corresponding induced rigid motion $t\in\mathbf [0,T]\mapsto\mathbf q(t)\in\mathcal Q$. The {\it generalized} forces, denoted by $\mathbf F$ in the sequel, are usually defined in Lagrangian Mechanics by:
\begin{equation}
\label{exp:internal_forces}
\mathbf F:=\frac{d}{dt}\frac{\partial L}{\partial \dot{\mathbf c}}-\frac{\partial L}{\partial {\mathbf c}},\quad(t\geq 0),
\end{equation}
where $L$ is the Lagrangian function defined in Subsection~\ref{def:lagr}.
This equality tells us that for all time, $\mathbf F$ is an element of $\mathcal S'$ (the dual space of $\mathcal S$, defined in the Appendix, Subsection~\ref{banach:series}). 
Let us then introduce the mass matrix $\mathbb K$, defined for all physically allowable control $\mathbf c$ as an element of $\mathcal L_2(\mathcal S\times\mathcal S)$ by:
$$\langle\mathbb K(\mathbf c),\widetilde{\mathbf c_1},\widetilde{\mathbf c_2}\rangle:=\langle\mathbb M^d(\mathbf c),\widetilde{\mathbf c_1},\widetilde{\mathbf c_2}\rangle-\langle\mathbb N(\mathbf c),\widetilde{\mathbf c}_1\rangle^T(\mathbb M^r(\mathbf c))^{-1}\langle\mathbb N(\mathbf c),\widetilde{\mathbf c}_2\rangle,
\quad(\widetilde{\mathbf c}_1,\widetilde{\mathbf c}_2\in\mathcal S).$$ According to Theorem~\ref{reg:matrices} and Proposition~\ref{PRO_M_Lipschitz}, the mapping $\mathbf c\in\mathcal D\mapsto \mathbb K(\mathbf c)\in\mathcal L_2(\mathcal S\times\mathcal S)$ is analytic. 
We next define the energy-like (or modified Lagrangian) amount:
\begin{equation}
\label{forces_express}
\mathcal L:=\frac{1}{2}\langle\mathbb K(\mathbf c),\dot{\mathbf c},\dot{\mathbf c}\rangle,\quad(t\geq 0).
\end{equation}
One proves after some algebra and taking into account  \eqref{equation_of_motion:1}, that identity \eqref{exp:internal_forces} can be rewritten as:
\begin{equation}
\label{first:form}
\mathbf F=\frac{d}{dt}\frac{\partial \mathcal L}{\partial \dot{\mathbf c}}-\frac{\partial \mathcal L}{\partial {\mathbf c}},\quad(t\geq 0).
\end{equation}
This expression allows us to compute straightforwardly the internal forces from the shape-changes without computing the induced motion $\mathbf q$ of the swimming animal.
It can be slightly expanded. Since $\mathbb K(\mathbf c)\in\mathcal L_2(\mathcal S\times\mathcal S)$, we deduce that, for all $\mathbf c\in\mathcal D$:
$$\frac{\partial\mathbb K}{\partial\mathbf c}(\mathbf c)\in\mathcal L_3(\mathcal S\times\mathcal S\times\mathcal S),$$
and the Frechet derivative of $\mathbb K$ in the direction $\widetilde{\mathbf c}$ at the point $\mathbf c$ is given by:
$$\Big\langle\frac{\partial\mathbb K}{\partial\mathbf c}(\mathbf c),\widetilde{\mathbf c},\cdot,\cdot\Big\rangle\in\mathcal L_2(\mathcal S\times\mathcal S).$$
We deduce that the expanded form of \eqref{first:form} is:
\begin{equation}
\label{second:form}
\langle\mathbb K(\mathbf c),\ddot{\mathbf c},\cdot\rangle+\langle\Gamma(\mathbf c),\dot{\mathbf c},\dot{\mathbf c},\cdot\rangle 
=\langle\mathbf F,\cdot\rangle,
\end{equation}
where $\Gamma(\mathbf c)\in \mathcal L_3(\mathcal S\times\mathcal S\times\mathcal S)$ is a so-called Christoffel symbol defined by:
\begin{multline*}
\langle\Gamma(\mathbf c),\widetilde{\mathbf c}_1,\widetilde{\mathbf c}_2,\widetilde{\mathbf c}_3\rangle:
=\\
\frac{1}{2}\Big[\Big\langle\frac{\partial\mathbb K}{\partial\mathbf c}(\mathbf c),\widetilde{\mathbf c}_2,\widetilde{\mathbf c}_1,\widetilde{\mathbf c}_3\Big\rangle+\Big\langle\frac{\partial\mathbb K}{\partial\mathbf c}(\mathbf c),\widetilde{\mathbf c}_2,\widetilde{\mathbf c}_3,\widetilde{\mathbf c}_1\Big\rangle-\Big\langle\frac{\partial\mathbb K}{\partial\mathbf c}(\mathbf c),\widetilde{\mathbf c}_3,\widetilde{\mathbf c}_1,\widetilde{\mathbf c}_2\Big\rangle\Big],\end{multline*}
for all $\widetilde{\mathbf c}_1,\widetilde{\mathbf c}_2,\widetilde{\mathbf c}_3\in\mathcal S$.
In this form, $\mathbf F$ depends only on $\mathbf c$ and on its first and second derivatives with respect to time. 
The internal forces are the relevant quantities one has to consider when seeking optimal strokes. In \cite{Kanso:aa} for instance, the cost function to be minimized 
over the time interval $[0,T]$ is taken to be:
$$J:=\int_0^T\|\mathbf F(t)\|^2_{\mathcal S'}{\rm d}t.$$
\subsubsection{Equivalence between controlling with shape-changes and with forces}
We assume now that the shape variable $\mathbf c$ lives in the finite dimensional vector space, $S_N$ (for some integer $N\geq 1$; see the Appendix, Subsection~\ref{banach:series}). We use for the mass matrices the expressions \eqref{def:matrix_mr_finite},  \eqref{def:matrix_N_finite} and \eqref{def:matrix_Md_finite}, allowing them to be defined for all $\mathbf c\in S_N$. Likewise, the mass matrix $\mathbb K(\mathbf c)$ can be identified with an actual $2N\times 2N$ symmetric matrix:
$$\mathbb K(\mathbf c):=\mathbb M^d(\mathbf c)-(\mathbb N(\mathbf c))^T(\mathbb M^r(\mathbf c))^{-1}\mathbb N(\mathbf c),$$
and we claim:
\begin{lemma}
\label{inverse:K}
There exists a constant $\nu_N>0$ such that, for all $\mathbf c$, $\widetilde{\mathbf c}\in S_N$: 
\begin{equation}
\label{inequality:K}
\frac{1}{2}\widetilde{\mathbf c}^T\mathbb K(\mathbf c)\widetilde{\mathbf c}\geq \nu_N\|\widetilde{\mathbf c}\|_{T_N}^2.
\end{equation}
\end{lemma}
As usual, we refer to the Appendix, Subsection~\ref{banach:series} for the notation.
\begin{proof}
Let us set $\widetilde{\mathbf q}:=(-\mathbb M^r(\mathbf c)^{-1}\mathbb N(\mathbf c)\widetilde{\mathbf c},\widetilde{\mathbf c})^T$ and observe that:
$$\widetilde{\mathbf q}^T\mathbb M(\mathbf c)\widetilde{\mathbf q}=\widetilde{\mathbf c}^T\mathbb K(\mathbf c)\widetilde{\mathbf c},$$
where the matrix $\mathbb M(\mathbf c)$ is defined in \eqref{def:matrix:M}. We next easily get that, since $\mathbb M^r(\mathbf c)$ is positive:
\begin{align*}
\widetilde{\mathbf q}^T\mathbb M(\mathbf c)\widetilde{\mathbf q}&\geq \frac{1}{2}\big[\mathbb M^r(\mathbf c)^{-1}\mathbb N(\mathbf c)\widetilde{\mathbf c}\big]^T\mathbb M^r(\mathbf c)
\big[\mathbb M^r(\mathbf c)^{-1}\mathbb N(\mathbf c)\widetilde{\mathbf c}\big]+\frac{1}{2}\widetilde{\mathbf c}^T\mathbb M^d(\mathbf c)\widetilde{\mathbf c}\\
&\geq\frac{1}{2}\widetilde{\mathbf c}^T\mathbb M^d(\mathbf c)\widetilde{\mathbf c}\geq\frac{1}{2}\pi\rho_0
\sum_{k=1}^N\frac{1}{k+1}(\widetilde a_k^2+\widetilde b_k^2),
\end{align*}
where $\widetilde{c}_k=a_k+ib_k$ for all $k\in\{1,\ldots,N\}$. The proof is completed after setting $\nu_N=2/(\rho_0 \pi N^2)$.
\end{proof}
\begin{rem}
Observe that the conclusion of this lemma is no longer true in the general infinite dimensional case since $\nu_N\to 0$ as $N\to\infty$. This explains why we are not able to prove the equivalence between controlling by shape-changes and internal forces in the general infinite dimensional case.
\end{rem}
The dual space of $S_N$ can be identified with $\mathbf C^N$ and equation \eqref{second:form} can be merely seen as an ODE in $\mathbf C^{N}$.
We can now  state our main equivalence result:
\begin{theorem}
\label{theo:equivalence}
For any given Lipschitz continuous function $t\in[0,T]\mapsto \mathbf F(t)\in\mathbf C^{N}$ and for any Cauchy data $(\dot{\mathbf c}_0,\mathbf c_0)\in S_N\times S_N$, there exists a unique smooth maximal solution $t\in[0,T]\mapsto \mathbf c(t)\in\mathbf C^N$ solving equation \eqref{second:form} and such that $\dot{\mathbf c}(0)=\dot{\mathbf c}_0$ and $\mathbf c(0)=\mathbf c_0$. 
\end{theorem}
\begin{rem}
Here, an intricate (an up to now open) problem consists in determining conditions for the generalized forces $\mathbf F$ ensuring that the shape-changes we obtain by integrating ODE \eqref{second:form} are allowable in the sense of Definition~\ref{allowable_finite} (or even more complicated, in the sense of Definition~\ref{alow:cont}). In other words, how to specify within Lagrangian formalism that the forces $\mathbf F$ are indeed internal to the animal?
\end{rem}
\begin{proof}
As explained in Subsection~\ref{subsection:finite:dim}, all of the terms depending in $\mathbf c$ in the ODE are analytic. Further, Lemma~\ref{inverse:K} ensures that the matrix $\mathbb K(\mathbf c)$ is always invertible.  
The Cauchy-Lipschitz Theorem applies and yields the existence and uniqueness of a maximal solution defined on some interval $[0,T^\ast)$ with $T^\ast\leq T$. 
Remember that $\mathcal L=\dot{\mathbf c}^T\mathbb K(\mathbf c)\dot{\mathbf c}/2$. One easily verifies that:
$$\left(\frac{d}{dt}\frac{\partial \mathcal L}{\partial\dot{\mathbf c}}-\frac{\partial\mathcal L}{\partial\mathbf c}\right)\cdot\dot{\mathbf c}=\frac{d\mathcal L}{dt},\quad(t\in[0,T^\ast)),$$
meaning that the variation of energy $\mathcal L$ is equal to the power-like amount $\mathbf F\cdot\dot{\mathbf c}$. Integrating over $[0,t]$ for any $0<t<T^\ast$ and invoking inequality \eqref{inequality:K} we get:
\begin{equation}
\label{7:4}
\nu_N\|\dot{\mathbf c}\|^2_{T_N}\leq \mathcal L(t)\leq \mathcal L(0)+\int_0^t\|\mathbf F(s)\|_{T_N}\|\dot{\mathbf c}(s)\|_{T_N}{\rm d}s.
\end{equation}
Setting then:
$$\Upsilon(t):=\int_0^t\|\mathbf F(s)\|_{T_N}\|\dot{\mathbf c}(s)\|_{T_N}{\rm d}s,$$
we obtain after some basic algebra that, for all $0<t<T^\ast$:
$$\frac{\Upsilon'(t)}{\sqrt{\Upsilon(t)+\mathcal L(0)}}\leq \frac{\|\mathbf F(t)\|_{T_N}}{\sqrt{\nu_N}}.$$
Integrating this inequality with respect to time, we get the estimate:
$$\Upsilon(t)\leq\mathcal L(0)+\left[\sqrt{\mathcal L(0)}+\frac{1}{2\sqrt{\nu_N}}\int_0^t\|\mathbf F(s)\|_{T_N}{\rm d}s\right]^2.$$
Plugging this result into \eqref{7:4}, we just have proved the Gronwall-type inequality:
$$\|\dot{\mathbf c}\|^2_{T_N}\leq \frac{2}{\nu_N}\mathcal L(0)+\frac{1}{\nu_N}\left[\sqrt{\mathcal L(0)}+\frac{1}{2\sqrt{\nu_N}}\int_0^t\|\mathbf F(s)\|_{T_N}{\rm d}s\right]^2,$$
meaning that $\dot{\mathbf c}$ remains bounded for all $t\in[0,T^\ast)$. Classical behavior results for ODEs tell us that $T^\ast=T$ and the proof is completed.
\end{proof}

\section{Controllability results} 
\label{SEC_control_abstract}
This Section is dedicated to the study of control problems associated with Equations \eqref{edo:var} or  \eqref{equation_of_motion:1}.
\subsection{Main Theorem of controllability}
\label{mainandcomments}
We begin by giving our main controllability result. We make some comments and give the outline of the proof that will be next set out in the following Subsections.
\subsubsection{Statement of the main theorem}
\begin{theorem}\label{THE_diminf_tracking}
For every ${\mu}^\dagger$ in $]0,1[$, for every $\varepsilon>0$ and for every reference continuous curve $({\mathbf{q}}^\dagger,{\mathbf c}^\dagger):[0,T]\rightarrow \mathcal Q \times { \cal E}^\bullet({\mu}^\dagger)$, there exists a real $\mu$ in $]0,1[$ and an analytic physically allowable curve $\mathbf{c}:[0,T]\rightarrow { \cal E}^\bullet(\mu)$  such that
\begin{enumerate}
\item $\|\mathbf c(t)-\mathbf c^\dagger(t)\|_{\mathcal S}<\varepsilon$ for all $t\in]0,T[$;
\item The solution $\mathbf{q}$ of \eqref{equation_of_motion:1} starting from ${\mathbf{q}}^\dagger(0)$ satisfies $\|\mathbf q(t)-\mathbf q^\dagger(t)\|_{\mathcal Q}<\varepsilon$ for all $t\in]0,T[$.
\end{enumerate}
\end{theorem}
\begin{rem} 
Note that in the hypotheses of Theorem \ref{THE_diminf_tracking}, we do \emph{not} require that ${\mathbf{c}}^\dagger$ be physically allowable. More precisely, $\mathbf c^\dagger$ can violate constraint \eqref{const_volume:2}.
\end{rem}
The conclusion of the Theorem may seem a little bit surprising. It implies that for given shape-changes $t\in[0,T]\mapsto{\mathbf c}^\dagger(t)\in\mathcal S$, to which corresponds the motion $t\in[0,T]\mapsto {\mathbf{q}}^\dagger(t)\in\mathcal Q$ obtained by solving the Euler-Lagrange equation \eqref{equation_of_motion:1} (say for instance, moving forward), one can find shape-changes $t\in[0,T]\mapsto {\mathbf c}(t)\in\mathcal S$ arbitrarily close to $t\in[0,T]\mapsto{\mathbf c}^\dagger(t)\in\mathcal S$ (for the uniform norm on the set of maps from $[0,T]$ to $\mathcal S$),  whose corresponding motion $t\in[0,T]\mapsto \mathbf{q}(t)\in\mathcal Q$ is arbitrarily close to any given trajectory (for instance, moving backward). This phenomenon will be termed {\it Moonwalking}. The strength of the Theorem can also be illustrated by selecting as reference function $t\mapsto\mathbf c^\dagger(t)$, shape-changes that do not result in locomotion. Such an example is given in the following paragraph. 

\subsubsection{Flapping does not allow locomotion}

We establish in this paragraph a well known negative result (usually referred to as the \emph{scallop theorem}, \cite{Purcell1977}). It states that it is impossible to achieve an arbitrarily large displacement of the amoeba by flapping. 

Let any physically allowable control function $\mathbf c:[0,T]\rightarrow \mathcal E^\bullet(\mu)$ be given (for some $\mu\in]0,1[$).
\begin{prop}\label{PRO_flapping}
There exists a real number $R>0$ such that
for any  $\mathcal C^1$ function $\beta:\mathbf{R}_+ \rightarrow [0,T]$ and for any initial condition $\mathbf{q}_0\in\mathcal Q$,  the solution $\mathbf{q}_{\beta}:=(\mathbf{r}_{\beta},\theta_{\beta})^T:\mathbf{R}_+\rightarrow \mathcal Q$ of Equation \eqref{equation_of_motion:1} corresponding to the shape-changes $t\in\mathbf R_+\mapsto \mathbf c(\beta(t))\in\mathcal E^\bullet(\mu)$, with initial condition $\mathbf{q}_0$ remains in the ball of $\mathcal Q$ of center $\mathbf{q}_0$ with radius $R$.
\end{prop}
\begin{proof}
Notice first that if $ \mathbf{c}:[0,T]\rightarrow {\cal E}^\bullet(\mu)$ is physically allowable, then for any smooth function $\beta:\mathbf{R}_+\rightarrow [0,T]$, the control function $\mathbf{c} \circ \beta :t\in\mathbf R_+\mapsto \mathbf{c}(\beta(t))\in \mathcal E^\bullet(\mu)$ is also physically allowable.   
From the equation of motion (\ref{equation_of_motion:1}), one gets
$$
\frac{d \mathbf{q}_{\beta}}{dt}(t)=-{\cal R}(\theta_{\beta}(\beta(t)))\mathbb M^{r}(\mathbf c(\beta(t))^{-1} \langle \mathbb N(\mathbf c(\beta(t))), \dot{\mathbf c}(\beta(t)) \rangle \beta'(t),
$$
which yields successively:
\begin{align*}
\mathbf{q}_{\beta}(t)&=\mathbf{q}(0)-\int_{0}^{t}\!\!\!{\cal R}(\theta_{\beta}(\beta(s)))\mathbb  M^{r}(\mathbf c(\beta(s)\mathbb )^{-1} \langle \mathbb N(\mathbf c(\beta(s))), \dot{\mathbf{c}}(\beta(s))\rangle  \beta'(s) \mathrm{d}s\\
&=\mathbf{q}(0)-\int_{\beta(0)}^{\beta(t)} \!\!\!\!\!\!{\cal R}(\theta_{\beta}(\tau)) \mathbb M^{r}(\mathbf c(\tau))^{-1}\langle \mathbb N(\mathbf c(\tau)), \dot{\mathbf c}(\tau) \rangle \mathrm{d}\tau.
\end{align*}
We next easily obtain the estimate:
\begin{multline*}
 \Big| \int_{\beta(0)}^{\beta(t)} \!\! {\cal R}(\theta_{\beta}(s))\mathbb  M^{r} (\mathbf c(s))^{-1}\langle \mathbb  N(\mathbf c(s)), \dot{\mathbf c}(s) \rangle {\rm d}s \Big |\\
  \leq T \sup_{s \in [0,T]} \Big| \mathbb M^{r} (\mathbf c(s))^{-1} \langle \mathbb  N(\mathbf c(s)), \dot{\mathbf c}(s) \rangle \Big |.
\end{multline*}
The function $t\mapsto \left |\mathbb  M^{-1} (\mathbf c(s))\mathbb  N(\mathbf c(s)) \dot{\mathbf c}(s) \right |$ being continuous, it has a finite supremum $A>0$ on the compact set $[0,T]$. It entails that for every positive $t$, $\mathbf{q}(t)$ is contained in the ball centered in $\mathbf{q}(0)$ with radius $R:=T A$.
\end{proof}

\begin{rem}
Proposition \ref{PRO_flapping} does not apply any longer if the function $\beta$ is no more assumed to be bounded. For instance, if $\mathbf{c}(0)=\mathbf{c}(T)$ and $\beta$ is the identity function, it may happen that $t\in\mathbf R_+\mapsto \mathbf{r}_{\beta}(t)\in\mathbf R^2$ is not bounded. 
See Section \ref{SEC_numerical} for examples of swimming with periodic deformations. 
\end{rem}
\subsubsection{Outline of the proof of Theorem~\ref{THE_diminf_tracking}}
The proof of the Theorem is somewhat intricate and will be done using finite dimensional control techniques. In the following subsection, we state Theorem~\ref{PRO_NDtracking},  a finite dimensional version of Theorem~\ref{THE_diminf_tracking}. Theorem~\ref{PRO_NDtracking} is proved in the particular $2$-dimensional case in Subsection~\ref{subsect:N=2} (using Lie brackets computed with Maple and Maxima, softwares allowing symbolic computations) and in the general case in Subsection~\ref{proof_theorem_42} (the proof resting on the computations of Lie brackets done in Subsection~\ref{subsect:N=2}). At last, in Subsection~\ref{prooftheorem41}, we prove how the infinite dimensional problem of control can be suitably approximated by a finite dimensional problem for which Theorem~\ref{PRO_NDtracking} applies. This will conclude the proof of Theorem~\ref{THE_diminf_tracking}. 


\subsection{Finite dimensional version of Theorem~\ref{THE_diminf_tracking} }
\subsubsection{Statement of the Theorem}
Remember that we have denoted merely $\rho$ the quotient $\rho_0/\rho_f>0$.
\begin{theorem}\label{PRO_NDtracking}
For every integer $N\geq 2$, for all but maybe a finite number of pairs $(\mu,\rho)\in]0,+\infty[\times]0,+\infty[$, for every $\varepsilon>0$ and for every reference continuous curve $({\mathbf{q}}^\dagger,{\mathbf c}^\dagger):[0,T]\rightarrow \mathcal Q \times { \cal E}_N({\mu})$ ($T>0$), there exists an analytic and mathematically allowable curve $\mathbf{c}:[0,T]\rightarrow { \cal E}_N(\mu)$  such that
\begin{enumerate}
\item $\|\mathbf c(t)-\mathbf c^\dagger(t)\|_{\mathcal S}<\varepsilon$ for all $t\in]0,T[$;
\item The solution $\mathbf{q}$ of \eqref{equation_of_motion:1} starting from ${\mathbf{q}}^\dagger(0)$ satisfies $\|\mathbf q(t)-\mathbf q^\dagger(t)\|_{\mathcal Q}<\varepsilon$ for all $t\in]0,T[$.
\end{enumerate}
\end{theorem}
\begin{rem}
As in Theorem~\ref{THE_diminf_tracking}, observe that the reference control function $\mathbf c^\dagger$ can violate constraint \eqref{C1:2}.
\end{rem}
\subsubsection{Restatement of the finite dimensional control problem}
\label{Para:restatement}
In Theorem~\ref{PRO_NDtracking}, it is required that the control functions $t\in[0,T]\mapsto\mathbf c(t)\in\mathcal S_N$ be allowable in the sense of Definition~\ref{allowable_finite}. It means in particular, according to \eqref{const_volume:2} and \eqref{C1:2}, that $\langle F(\mathbf c),\dot{\mathbf c}\rangle=\langle G(\mathbf c),\dot{\mathbf c}\rangle=0$ for all $t\in]0,T[$ such that $\dot{\mathbf c}(t)$ is well defined. In order to deal with these constraints, we consider $\mathfrak X_N:=(\mathbf X^j)_{1\leq j\leq n}$ a set of $n$ vectors fields in $\mathcal S_N$ ($n$ an integer, $n\geq 1$), where $\mathbf X^j:=(X^j_k)_{k\geq 1}$ and for all $\mathbf c\in\mathcal S_N$, $X^j_k(\mathbf c):=x^j_k(\mathbf c)+i y^j_k(\mathbf c)$, $x^j_k(\mathbf c)$, $y^j_k(\mathbf c)\in\mathbf R$, $k\geq 1$ and $X^j_k=0$ if $k>N$.
\begin{definition}[Allowable set of vector fields]
\label{allowable:fields}
A set of vector fields $\mathfrak X_N:=(\mathbf X^j)_{1\leq j\leq n}$ defined in $\mathcal S_N$ is said to be allowable when:
\begin{enumerate}
\item For all $j\in\{1,\ldots,n\}$, the field $\mathbf c\in\mathcal S_N\mapsto \mathbf X^j(\mathbf c)\in\mathcal S_N$ is analytic.
\item For all $\mathbf c\in\mathcal S_N$ and all $j\in\{1,\ldots,n\}$, $\langle F(\mathbf c),\mathbf X^j({\mathbf c})\rangle=\langle G(\mathbf c),\mathbf X^j({\mathbf c})\rangle=0$.
\end{enumerate}
\end{definition}
Consider now a set $\boldsymbol\lambda:=(\lambda_j)_{1\leq j\leq n}$ of piecewise constant functions from $[0,T]\subset\mathbf R$ into $\mathbf R$ and the EDO:
\begin{equation}
\label{edo:def:c}
\dot{\mathbf c}(t)=\sum_{j=1}^n\lambda_j(t)\mathbf X^j(\mathbf c(t)),\quad(t>0).
\end{equation}
The usefulness of Definition~\ref{allowable:fields} arises with the following, easy to prove, property:
\begin{prop}
For any set of piecewise constant functions $\boldsymbol\lambda=(\lambda_j)_{1\leq j\leq n}$ as above and any initial data $\mathbf c_0\in\mathcal S_N$, there exists one unique solution $t\in[0,T]\mapsto \mathbf c(t)\in\mathcal S_N$ to EDO\eqref{edo:def:c}. Further, $\mathbf c$ is mathematically allowable (in the sense of Definition~\ref{allowable_finite} with $\mu:=\|\mathbf c_0\|_{\mathcal T_N}$).
\end{prop}
Once the set $\mathfrak X_N$ has been chosen, the new control variables turn out to be the piecewise constant functions $\boldsymbol\lambda=(\lambda_j)_{1\leq j\leq n}$. We can give an example of such allowable set of vector fields when $N=2$. Thus, let us define $\mathfrak X_2:=(\mathbf X^j)_{1\leq j\leq 4}$ by:
\begin{equation}
\label{def:champ:vecteurs}
\begin{array}{ll}
X^1_1(\mathbf c)=-2 \left \lbrack \left ( {b_1 b_2}/3 + a_1 a_2 \right ) + i \left ( b_1 a_2 -a_1 b_2 /3\right ) \right \rbrack, & X^1_2(\mathbf c)= |c_1|^2,\\[2pt]
X^2_1(\mathbf c)=-2 \left \lbrack \left (   a_1 b_2 -{b_1 a_2}/3\right ) + i \left ( b_1 b_2 +a_1 a_2/3 \right ) \right \rbrack,& X^2_2(\mathbf c)= i|c_1|^2,\\[2pt]
X^3_2(\mathbf c)=-({3}/{2}) \left \lbrack \left ( b_1 b_2 + a_1 a_2/3 \right ) + i \left (  a_1 b_2/3 -b_1 a_2 \right ) \right \rbrack, & X_1^3(\mathbf c)=|c_2|^2,\\[2pt]
X^4_2(\mathbf c)=-({3}/{2}) \left \lbrack \left ( b_1 a_2/3 -a_1 a_2 \right ) + i \left ( b_1 b_2/3 +a_1 a_2 \right ) \right \rbrack, & X_1^4(\mathbf c)=i|c_2|^2,
\end{array}
\end{equation}
where we recall that $\mathbf c:=(c_k)_{k\geq 1}$ with $c_k=a_k+ib_k$ ($a_k$, $b_k\in\mathbf R$) for all $k\geq 1$. One can easily verify that $\mathfrak X_2$ is indeed allowable in the sense of Definition~\ref{allowable:fields}.

Going back to the general case, we can rewrite the EDO \eqref{equation_of_motion:1} in the form:
\begin{equation}
\label{EQ_control_system_dim4}
\frac{d}{dt}\begin{bmatrix}\mathbf q \\ \mathbf{c} \end{bmatrix} =\begin{bmatrix}-\sum_{j=1}^n \lambda_j(t) \mathcal R(\theta)(\mathbb M^r(\mathbf c))^{-1}\langle\mathbb N(\mathbf c),\mathbf X^j(\mathbf c)\rangle\\\mathbf \sum_{j=1}^n \lambda_j(t) \mathbf X^j(\mathbf c)\end{bmatrix},
\end{equation}
supplemented with initial conditions: $(\mathbf q(0),\mathbf c(0))=(\mathbf q_0,\mathbf c_0)\in\mathcal Q\times\mathcal S_N$. This then leads us to introduce the set of analytic vector fields $\mathfrak Y_N:=(\mathbf Y^j)_{1\leq j\leq n}$, defined on $\mathcal Q\times\mathcal S_N$ by:
\begin{equation}
\label{def:field:Y_N}
\mathbf Y^j(\mathbf q,\mathbf c):=\begin{bmatrix}
-\mathcal R(\theta)(\mathbb M^r(\mathbf c))^{-1}\langle\mathbb N(\mathbf c), \mathbf X^j(\mathbf c)\rangle\\
\mathbf  X^j(\mathbf c)
\end{bmatrix},\quad\forall\,(\mathbf q,\mathbf c)\in\mathcal Q\times\mathcal S_N,
\end{equation}
and to rewrite \eqref{EQ_control_system_dim4} as:
\begin{equation}
\label{EQ_control_system_dim4:1}
 \frac{d}{dt}\begin{bmatrix}
\mathbf q\\
\mathbf c
\end{bmatrix}=\sum_{j=1}^n\lambda_j(t) \mathbf Y^j(\mathbf{q},\mathbf c).
\end{equation}
Notice that, according to the expressions of $\mathbb M^r(\mathbf c)$ and $\mathbb N(\mathbf c)$,  the vector fields $\mathbf Y^j$ depend not only on $(\theta,\mathbf c)$ but also on $\rho:=\rho_0/\rho_f$, excluding all other quantities. Further, the dependence is analytic with respect to all of these variables (including $\rho$).
Seen as a problem of control, equation \eqref{EQ_control_system_dim4:1} fits the general form of geometric control theory. So, before going further, let us now recall some important results.
\subsubsection{Tools of geometric control theory}
\label{SEC-Track}
Let $M$ be an analytic connected manifold, endowed with the Riemannian distance $\mathrm{d}_M$ and let ${\mathcal X}
$ be a set of analytic vector fields on $M$.

Let $(X_j)_{1\leq j\leq p}$ ($p\in\mathbf N$, $p\geq 1$) be a finite sequence of elements of ${\mathcal X}
$. For any given finite sequence $(t_l)_{1\leq l \leq p}$ of positive real numbers, we define $T:=\sum_{l=1}^p t_l$ and the application $$\Phi_{(t_l,X_l
)_{1\leq l \leq p}}:M\times [0,T] \rightarrow M,$$  by: 
$$\begin{cases}\Phi_{(t_l,X_l
)_{1\leq l \leq p}}(q,0)=q\\
\Phi_{(t_l,X_l
)_{1\leq l \leq p}}(q,t)=e^{t_j X_j} \left ( e^{t_{j-1} X_{j-1}}\circ \cdots \circ e^{t_1 X_1} q\right )&\text{ if }\sum_{l=1}^{j-1} t_l \leq t \leq  \sum_{l=1}^{j} t_l .
\end{cases}$$  
This definition can easily be extended to the case of non-complete vector fields, but the domain of $\Phi$ is then restricted to the product of a certain (possibly empty) open set of $M$ by the interval $[0,T]$. 
\begin{defn}
 We say that the trajectories of ${\mathcal X}
$ can \emph{track} a given continuous curve $\gamma:[0,T]\rightarrow M$ ($T\in\mathbf R_+$) if for every $\varepsilon>0$, there exists a finite sequence $S=(t_j,X_j)_{1\leq j\leq p}$ of elements of $\mathbf{R}\times {{\mathcal X}
}$ such that (i) $\gamma(0)\times [0,T]$ belongs to the domain of $\Phi_{S}$, (ii) $\Phi_S(\gamma(0),T)=c(T)$ and (iii) $ {\rm d}_M(\Phi_S(\gamma(0),t),\gamma(t))<\varepsilon$ for all $t\in[0,T]$. 
\end{defn}
The following proposition is a pretty classical consequence of the Orbit Theorem (recalled in Section \ref{SEC_Appendix_Orbit_Th}). We give a proof for the sake of completeness (see \cite{MR1972788} for further discussions). 
\begin{prop}\label{PRO_tracking}
If ${\mathcal X}
$ is a symmetric cone such that $\mathrm{Lie}_{q}{{\mathcal X}
}=T_qM$ for any $q$ in $M$, then the trajectories of ${\mathcal X}
$ can track any given continuous reference curve on $M$. 
\end{prop}

\begin{proof}
Let $\varepsilon>0$ be a positive number and let $\gamma:[0,T]\rightarrow M$ be a continuous curve on $M$. Since $[0,T]$ is compact, $\gamma$ is uniformly  continuous on $[0,T]$. Hence, there exists some $\eta>0$ such that for any $t,t'$ in $[0,T]$, $|t-t'|\leq \eta$ implies ${\rm d}_M(\gamma(t),\gamma(t'))<{\varepsilon}/{3}$. Without loss of generality, we may assume that $T=N \eta$ with $N \in \mathbf{N}$. 
For any integer $n$ in $[0,N-1]$, define $M_n$ as the open ${\varepsilon}/{2}$ neighborhood of $\gamma(n\eta)$, and consider ${{\mathcal X}
}_n$ the restriction of ${\mathcal X}
$ to $M_n$. The set $M_n$ is a connected analytic manifold, which contains both $\gamma(n\eta)$ and $\gamma((n+1)\eta)$. For every $q$ in $M_n$, the Lie algebra of ${{\mathcal X}
}_n$ at $q$ is equal to $T_qM_n$. From Proposition \ref{PRO_CompleteControl}, we deduce the existence of a finite sequence $S_n=(t_l^n,X_l
^n)_{1\leq l \leq p_n}$  such that $\Phi_{S_n}(\gamma((l-1) \eta),\eta)=\gamma(l \eta)$, and $\Phi_{S_n}(\gamma((l-1) \eta),t)$ belongs to $M_l$ for every $t$ in $[0,\eta]$, $l$ in $\{1 \ldots p_n\}$. 

Define now $S$ as the concatenation of $S_0,S_2,\ldots, S_{N-1}$. The set $S$ is a finite sequence of elements of the product of $\mathbf{R}$ and ${\mathcal X}
$.  By construction, $\{\gamma(0)\} \times [0,T]$ belongs to the domain of $\Phi_S$ and $\Phi_S(\gamma(0),T)=\gamma(T)$. For every $t$ in $[0,T]$, there exists an integer $n$ not greater than $N-1$ such that $n\eta \leq t \leq (n+1) \eta$. Hence   $\Phi_S(\gamma(0),t)=\Phi_{S_n}({\gamma(n\eta,t-n\eta)})$ belongs to $M_n$, that is
\begin{align*}
 \mathrm{d}_M(\Phi_S(\gamma(0),t),\gamma(t))&={\rm d}_M(\Phi_{S_n}{\gamma(n\eta,t-n\eta)},\gamma(t))\\
&<\mathrm{d}_M(\Phi_{S_n}{\gamma(n\eta,t-n\eta)},\gamma(n\eta))+\mathrm{d}_M(\gamma(n\eta),\gamma(t))\\
&< \frac{\varepsilon}{2}+\frac{\varepsilon}{2}=\varepsilon,
\end{align*}
and the proof is completed.
\end{proof}
The main idea in the proof of Theorem~\ref{PRO_NDtracking} will be to apply Proposition~\ref{PRO_tracking} on the analytic manifold $\mathcal Q\times\mathcal S_N$ and with analytic set of vector fields $\mathfrak Y_N$ defined in \eqref{def:field:Y_N}.
\subsubsection{Remark on the optimal control problem}
Geometric control theory gives for free the existence of optimal solutions to our control problem. In particular, Filipov Theorem (see \cite[Chapter 10]{agrachev}) ensures that:
\begin{theorem}
Let $f: \mathcal Q\times \mathcal S_N
\times \mathbf{R}^n \rightarrow  \mathbf{R}$ be a continuous function, $\mathcal K$ be a compact subset of $\mathbf{R}^n$, $(\mathbf{q}_0,\mathbf{c}_0)$  and $(\mathbf{q}_1,\mathbf{c}_1)$ be two points of $\mathcal Q\times\mathcal S_N$. If there exists a trajectory of \eqref{EQ_control_system_dim4:1} steering  $(\mathbf{q}_0,\mathbf{c}_0)$ to $(\mathbf{q}_1,\mathbf{c}_1)$, associated to a measurable bounded control $\boldsymbol\lambda$ taking value in $\mathcal K$, then there exists a measurable bounded  control $\boldsymbol\lambda^\dagger:[0,T]\rightarrow \mathbf{R}^n$ taking also value in $\mathcal K$ such that:
\begin{enumerate}
\item The corresponding solution $(\mathbf q^\dagger, \mathbf c^\dagger)$ of \eqref{EQ_control_system_dim4:1} steers $(\mathbf{q}_0,\mathbf{c}_0)$ to $(\mathbf{q}_1,\mathbf{c}_1)$;
\item The triplet $(\mathbf q^\dagger,\mathbf c^\dagger,\boldsymbol\lambda^\dagger)$ realizes the infimum of the cost $$\int_0^T f(\mathbf{q}(t),\mathbf{c}(t), \boldsymbol\lambda(t)) \mathrm{d}t,$$ among all measurable bounded controls $\boldsymbol\lambda$ taking value in $\mathcal K$ and  steering the system \eqref{EQ_control_system_dim4:1} from $(\mathbf{q}_0,\mathbf{c}_0)$ to $(\mathbf{q}_1,\mathbf{c}_1)$.
\end{enumerate}
\end{theorem}
Notice that if there exists a trajectory linking $(\mathbf{q}_0,\mathbf{c}_0)$ to $(\mathbf{q}_1,\mathbf{c}_1)$, then $\mathbf c_0$ and $\mathbf c_1$ both have to belong to the same set $\mathcal E_N(\mu)$ for some $\mu>0$ i.e., they have to satisfy $\|\mathbf c_0\|_{\mathcal T_N}=\|\mathbf c_1\|_{\mathcal T_N}=\mu$.
\subsection{The case $N=2$}
\label{subsect:N=2}
In this section, we focus on the case $N=2$ and we consider the set of analytic vector fields $\mathfrak X_2$ defined on $\mathcal E_2(\mu)$ (for all $\mu>0$) by \eqref{def:champ:vecteurs} . As explained in the Appendix, Subsection~\ref{banach:series} we identify $\mathcal E_2(\mu)$ with $E_2(\mu)$ the 3-dimensional analytic submanifold of $\mathbf R^4$.
Furthermore, the dynamic induced by the vector fields $\mathfrak X_2$ on $\mathcal E_2({\mu})$ is the same as the one induced on $E_2(\mu)$ by the vector fields $\mathcal X_2:=(X^j)_{1\leq j \leq 4}$ defined for any $\mathbf c:=(a_1,b_1,a_2,b_2)^T\in\mathbf R^4$ by: 
\begin{alignat*}{3}
X^1(\mathbf c)&:=\begin{bmatrix} -2 (b_1 b_2/3 + a_1 a_2 ) \\ -2  (b_1 a_2 -a_1 b_2/3  ) \\ a_1^2+b_1^2 \\ 0 \end{bmatrix},&\quad&
X^2(\mathbf c):=\begin{bmatrix} -2 ( a_1 b_ 2 -b_1 a_2/3 ) \\ -2  (b_1 b_2 +a_1 a_2/3  ) \\ 0 \\  a_1^2+b_1^2  \end{bmatrix},\\
X^3(\mathbf c)&:=\begin{bmatrix} a_2^2+b_2^2\\ 0 \\ -3( b_1 b_2 + a_1 a_2/3 )/2 \\-3( a_1 b_2/3  - b_1 a_2)/2\end{bmatrix},&&
X^4(\mathbf c):=\begin{bmatrix} 0\\  a_2^2+b_2^2\\ -3( -a_1 a_2 +b_1 a_2/3 )/2 \\-3( b_1 b_2/3 + a_1 a_2)/2\end{bmatrix}.
\end{alignat*}
 It is easy to verify that for every $\mathbf{c}\in\mathbf R^4$, $\mathbf c\neq 0$, the linear space spanned by $\mathcal X_2$ has cardinal two.
Since the expressions of the mass matrices $\mathbb M^r(\mathbf c)$ and $\mathbb N(\mathbf c)$ are given in Subsections~\ref{SEC_Expressions_of_the_mass_matrices}  and \ref{SEC_special_case} we can compute explicitly the expression of the vector fields in $\mathfrak Y_2:=(\mathbf Y^j)_{1\leq j\leq 4}$ defined in \eqref{def:field:Y_N} in $\mathcal Q\times\mathcal S_2$. Like $\mathfrak X_2$, $\mathfrak Y_2$ can be identified with the set of vector fields $\mathcal Y_2:=(Y^j)_{1\leq j\leq 4}$ defined in $\mathcal Q\times \mathbf R^4$.
\subsubsection{Computation of the Lie algebra}
\label{subs:comput:lie:alg}
If we do not require the control function to be analytic in Theorem~\ref{PRO_NDtracking}, according to Proposition \ref{PRO_tracking}, it suffices to check whether $\mathrm{Lie}(\mathcal Y_2)$ has dimension 6 everywhere on $\mathcal Q \times E_2{(\mu)}$ to prove the theorem when $N=2$. Since the expressions of the fields $Y^j$ ($1\leq j\leq 4$) are somewhat intricate, we first concentrate on the set  $\mathrm{Lie}(\mathcal X_2)$.
\begin{prop}\label{PRO_calcul_dim_3}
For any $\mu>0$, the family $\mathcal X_2$ is completely nonholonomic on $E_2({\mu})$, that is, for any $\mathbf{c}$ in $\mathbf R^4$, $\mathbf c\neq 0$, $\mathrm{Lie}_{\mathbf{c}}(\mathcal X_2)=T_{\mathbf{c}}{E_2({\mu})}$ where $\mu:=\|\mathbf c\|_{T_2}$. 
\end{prop}
\begin{proof}
Recall that for every non-zero $\mathbf{c}\in \mathbf R^4$, 
the four vectors  $X^1(\mathbf c)$, $X^2(\mathbf c)$, $X^3(\mathbf c)$ and $X^4(\mathbf c)$ span a 2-dimensional subspace of the 3-dimensional linear space $T_{\mathbf{c}} E_2({\mu})$, ($\mu:=\|\mathbf c\|_{T_2}$).
A direct computation (see Proposition~\ref{PRO_calculLieBracket}) gives
$$
[X^1,X^2](\mathbf c)=\frac{4}{3}(a_1^2+b_1^2)\begin{bmatrix}
-b_1 \\ a_1 \\ -3 b_2 \\ 3 a_2  \end{bmatrix}
\quad\text{and}\quad
[X^3,X^4](\mathbf c)=(a_2^2+b_2^2)\begin{bmatrix}-b_1\\ a_1 \\ -3 b_2 \\ 3 a_2 \end{bmatrix}.
$$
For any $\mathbf{c}=(a_1,b_1,a_2,b_2)$ such that $a_1^2 + b_1^2 \neq 0$, the vectors $X^1(\mathbf c)$ and $X^2(\mathbf c)$ are clearly linearly independent. Proceed by contradiction and assume that $X^1(\mathbf c)$, $X^2(\mathbf c)$ and $[X^1,X^2](\mathbf c)$ are not linearly independent. Then, there exists $\alpha
_1$ and $\alpha
_2$ (two real numbers) such that $[X^1,X^2](\mathbf c)=\alpha
_1 X^1(\mathbf c)+\alpha
_2 X^2(\mathbf c)$. Since $a_1^2+b_1^2\neq 0$, one has $\alpha
_1=-4 b_2$ and $\alpha
_2=4 a_2$. Projecting the equality $[X^1,X^2](\mathbf c)=\alpha
_1 X^1(\mathbf c)+\alpha
_2 X^2(\mathbf c)$ on the first coordinate, one gets 
$-b_1(a_1^2+b_1^2)=2 b_1 (a_2^2+b_2^2)$, that is $a_2=b_2=0$ or $b_1=0$. If $a_2=b_2=0$, then $a_1^2+b_1^2=0$ which is in contradiction with $a_1^2 + b_1^2 \neq 0$, hence $b_1=0$. From the projection on the second coordinate, one gets 
$(a_1^2+b_1^2)a_1=-2a_1 (a_2^2+b_2^2)$, from which one deduces $a_1=0$, which is incompatible with the hypothesis $a_1^2+b_1^2=0$. We have then proved that for any $\mathbf c$ such that $a_1^2+b_1^2\neq 0$, $X^1(\mathbf c)$, $X^2(\mathbf c)$ and $[X^1,X^2](\mathbf c)$ are linearly independent. 

The same argument shows that, if $a_2^2 + b_2^2 \neq 0$, then $X^3(\mathbf c)$, $X^4(\mathbf c )$ and $[X^3,X^4](\mathbf c)$ are also linearly independent.

For every non-zero $\mathbf c\in\mathbf R^4$, $\dim \mathrm{Lie}_{\mathbf c}(\mathcal X_2) \geq 3=\dim T_{\mathbf c}E_2({\mu})$. Therefore, $\mathrm{Lie}_{\mathbf c}(\mathcal X_2)=T_{\mathbf c}{ E}_2({\mu})$ and the proof is completed.
\end{proof}
Since the family $\mathcal X_2$ is completely non holonomic on every submanifold $E_2({\mu})$, the attainable set at any positive time of the control system 
$$
\dot{\mathbf c}(t)=\sum_{j=1}^4 \lambda_j(t) X^j(\mathbf c(t)),\quad\mathbf c(0)=\mathbf c_0\in \mathbf R^4, 
$$
with piecewise constant controls $(\lambda_j)_{1\leq j \leq 4}$ is equal to $E_2({\mu})$ ($\mu:=\|\mathbf c_0\|_{T_2}$). 

We next consider the control system defined as the projection of the system \eqref{EQ_control_system_dim4} on $\mathbf{R}/2\pi\times \mathcal S_2$. Indeed, denoting by $(v)_3$ the third component of  any vector $v$ of $\mathbf{R}^3$, this system  reads 
\begin{align*}
\frac{d}{dt}
\begin{bmatrix}
\theta\\
\mathbf c
\end{bmatrix}&=
\begin{bmatrix}
-\sum_{j=1}^4\lambda_j(t)\Big(\mathbb M^r(\mathbf c)^{-1} \langle \mathbb N (\mathbf c), \mathbf X^j(\mathbf c)\rangle\Big)_3 \\
\sum_{j=1}^4 \lambda_j(t) \mathbf X^j(\mathbf{c})
\end{bmatrix},\\
(\theta,\mathbf{c})(0)&=(\theta_0, \mathbf{c}_0)\in\mathbf R/2\pi\times \mathcal E_2(\mu).
\end{align*}
We define $\widehat{\mathfrak X}_2:=(\widehat{\mathbf X}^j)_{1\leq j\leq 4}$ the set of analytic vector fields $\widehat{\mathbf X}^j$ ($1\leq j\leq 4$) in $\mathbf{R}/2\pi \times \mathcal S_2$ by
$$ \widehat{\mathbf X}^j(\theta,\mathbf{c})=\begin{bmatrix}
\Big(\mathbb M^r(\mathbf c)^{-1} \langle \mathbb N(\mathbf c), \mathbf X^j(\mathbf{c})\rangle \Big)_3 \\ \mathbf X^j(\mathbf{c})\end{bmatrix},\quad(1\leq j\leq 4).$$
According to the expressions of the matrices $\mathbb M^r(\mathbf c)$ and $\mathbb N(\mathbf c)$, the vector fields $\widehat{\mathbf X}^j$ ($1\leq j\leq 4$) depend on $\mathbf c$ and also on $\rho:=\rho_0/\rho_f$, excluding all other quantities and the dependence is analytic.
As previously, the set $\widehat{\mathfrak X}_2$ can be identified with the set of vector fields $\widehat{\mathcal X}_2:=(\widehat X^j)_{1\leq j\leq 4}$ defined in $\mathbf R/2\pi\times\mathbf R^4$.

\begin{prop}\label{PRO_CalculAlgebreLieDim4} For all but maybe a finite number of pairs $(\mu,\rho)\in]0,+\infty[\times]0,+\infty[$ and for any $(\theta,\mathbf c)\in\mathbf{R}/2\pi\times E_2(\mu)$, the Lie algebra $\mathrm{Lie}_{(\theta,\mathbf{c})}(\widehat{\mathcal X}_2)$ is equal to the whole tangent space $T_{(\theta,\mathbf c)}( \mathbf{R}/2\pi\times E_2({\mu}))$.
\end{prop}
This proof involves symbolic computations performed with both software programs: Maple and Maxima. The relevant  computations can be downloaded on the web page \url{http://www.iecn.u-nancy.fr/~munnier/page_amoeba/control_index.html}.
\begin{proof}
The proof comprises three steps:\\
{\bf Step 1:} Let be $\mu>0$ and $\rho>0$ and choose any point $\mathbf c^\ast\in E_2(\mu)$. Since the vector fields composing $\widehat{\mathcal X}_2$ do not depend on $\theta$, we deduce that the Lie algebra $\mathrm{Lie}_{(\theta, \mathbf c^\ast)}(\widehat{\mathcal X}_2)$ has the same dimension $d$, ($0\leq d \leq 4$) for any $\theta\in\mathbf R/2\pi$. According to Proposition~\ref{PRO_calcul_dim_3},  the orbit of $\mathcal X_2$ through $\mathbf c^\ast$ is equal to $E_2({\mu})$ and hence the orbit of $\widehat{\mathcal X}_2$ through any point $(\theta_1,\mathbf c)\in \mathbf R/2\pi\times E_2(\mu)$ contains at least one point $(\theta_2,\mathbf c^\ast)$ for some $\theta_2\in\mathbf R/2\pi$. As a consequence of the Orbit Theorem (Theorem~\ref{THE_Orbit}), $\mathrm{Lie}_{(\theta_1, \mathbf c)}(\widehat{\mathcal X}_2)$
is also of dimension $d$ and hence, the dimension of  $\mathrm{Lie}_{(\theta, \mathbf c)}(\widehat{\mathcal X}_2)$ is constant on $\mathbf R/2\pi\times E_2(\mu)$.

We deduce that for any fixed pair $(\mu,\rho)\in]0,+\infty[\times]0,+\infty[$, it is enough to prove that there exists at least one point $(\theta^\ast,\mathbf c^\ast)\in \mathbf R/2\pi\times E_2(\mu)$ for which $\mathrm{Lie}_{(\theta^\ast, \mathbf c^\ast)}(\widehat{\mathcal X}_2)$ has dimension 4 to prove that $\mathrm{Lie}_{(\theta,\mathbf{c})}(\widehat{\mathcal X}_2)= T_{(\theta,\mathbf c)}( \mathbf{R}/2\pi\times E_2({\mu}))$ for all $(\theta,\mathbf c)\in \mathbf{R}/2\pi\times E_2({\mu})$.

\noindent{\bf Step 2:} Explicit computations of the iterated Lie brackets of the fields $\widehat{X}^j$ are straightforward, yet quite intricate. In order to simplify the resulting expressions, we define the fields $\widehat{X}^j_{\bullet}=\rho_f^{-3}\det(\mathbb M^r) \widehat{X}^j$ and $\widehat{\mathcal X}_2^\bullet:=(\widehat{X}^j_{\bullet})_{1\leq j\leq 4}$. Since $\rho_f^{-3}\det (\mathbb M^r)$ does not vanish, the linear spaces $\mathrm{Lie}_{(\theta,\mathbf c)}(\widehat{\mathcal X}^2)$ and $\mathrm{Lie}_{(\theta,\mathbf c)}(\widehat{\mathcal X}^2_\bullet)$ have the same dimension  (see Proposition~\ref{PRO_LieBracket_Multiple}). Further, like the fields $\widehat X^j$, the fields $\widehat X^j_\bullet$ depend only (and analytically) on $\mathbf c$ and $\rho:=\rho_0/\rho_f$. Despite this simplification, the explicit results of the computation are definitely still too long to be printed. Since the fields $\widehat{X}^j_\bullet$ do not depend on $\theta$, these expressions also only depend on $\mathbf{c}$ and $\rho$. We chose a specific $\mathbf c$ having the form $\mathbf c_+:=(x,0,x,0)$ and $\mathbf c_-:=(-x,0,-x,0)$ ($x\in\mathbf R$). When $x$ describes $\mathbf R_+$, $\mu:=\|\mathbf c_+\|_{T_2}=\|\mathbf c_-\|_{T_2}=\sqrt{3}x$ describes also $\mathbf R_+$.
We use Maxima and Maple to obtain the huge expressions of ${\widehat{X}^1}_{\bullet}(\theta,{\mathbf c}_+)$, ${\widehat{X}^2}_\bullet(\theta,{\mathbf c}_+)$, $[{\widehat{X}^1}_\bullet,{\widehat{X}^2}_\bullet](\theta,{\mathbf c}_+)$ and $[{\widehat{X}^1}_\bullet,[\widehat{X}^1_\bullet,{\widehat{X}^2}_\bullet]](\theta,{\mathbf c}_+)$ and ${\widehat{X}^1}_{\bullet}(\theta,{\mathbf c}_-)$, ${\widehat{X}^2}_\bullet(\theta,{\mathbf c}_-)$, $[{\widehat{X}^1}_\bullet,{\widehat{X}^2}_\bullet](\theta,{\mathbf c}_-)$ and $[{\widehat{X}^1}_\bullet,[\widehat{X}^1_\bullet,{\widehat{X}^2}_\bullet]](\theta,{\mathbf c}_-)$ in terms of $x$ and $\rho$. With each family of 4 column vectors, we define a $5\times 4$ matrix and extract a $4\times 4$ submatrix. Computing next the determinants of these square matrices, we obtain two polynomials $P^+(\rho,x)$ and $P^-(\rho,x)$ of degree 20 in the variable $\rho$ and 65 in the variables $x$. We first consider them as polynomial in the variable $x$ with coefficients in $\mathbf R(\rho)$, the field of rational fractions in the variable $\rho$ with coefficients in $\mathbf R$.  Seeking their gcd, we obtain a monomial. If we denote by $\mathcal Z_\rho$ the finite set consisting of all of the zeros of either the numerator or the denominator of all of the rational fractions of $\mathbf R(\rho)$ arising in the Euclidean algorithm leading to the expression of the gcd, we deduce that for all $\rho\notin\mathcal Z_\rho$, $P^+(\rho,x)$ and $P^-(\rho,x)$ (seen as polynomials in the variable $x$) have only $0$ as common root and next that for any $x\neq 0$, at least one out of the two families of vectors fields spans a 4-dimensional vector space. This proves the Proposition for all of the pairs $(\mu,\rho)\in]0,+\infty[\times(]0,+\infty[\setminus \mathcal Z_\rho)$.
\begin{figure}[H]
\centerline{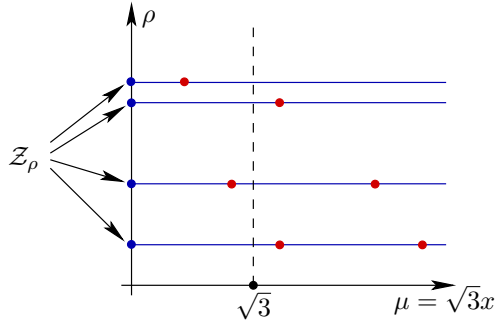}
\caption{In red, the possibly {\it bad} pairs $(\mu,\rho)$ of $]0,+\infty[\times]0,+\infty[$}
\end{figure}
\noindent{\bf Step 3:} Let us now specialize $x=1$ ($\mu=\sqrt{3}$) and see both $P^+(\rho,1)$ and $P^-(\rho,1)$ as polynomials in $\rho$ with coefficients in $\mathbf R$. Their gcd (also computed with Maple and Maxima) is equal to 1, meaning that they have no common root. This yields the conclusion of the Proposition for all of the pairs $(\sqrt{3},\rho)$, $\rho>0$. Let us next pick some $\rho^\dagger\in\mathcal Z_\rho$ and consider $(P^+(\rho^\dagger,x)-P^-(\rho^\dagger,x))^2+P^+(\rho^\dagger,x)^2$ as a polynomial in $x$. If it vanishes in an infinite number of points, then it would be identically zero, which would contradict the result obtained for $x=1$. We deduce that for any $\rho\in\mathcal Z_\rho$, the polynomials in $x$, $P^+(\rho,x)$ and $P^-(\rho,x)$ have at most a finite number of commune zeros and the proof is completed.
\end{proof}
\begin{rem}
The Authors conjecture that the result of Proposition~\ref{PRO_CalculAlgebreLieDim4} actually holds true for all of the pairs $(\mu,\rho)\in]0,+\infty[\times]0,+\infty[$. 
\end{rem}
We denote $\widehat{\mathfrak B}$ the finite (maybe empty) set of all the pairs $(\mu,\rho)\in ]0,+\infty[\times]0,+\infty[$ for which the conclusion of Proposition~\ref{PRO_CalculAlgebreLieDim4} does not hold true and we claim:
\begin{prop}\label{PRO_calcul_Lie_dim_6}
For all but maybe a finite number of pairs $(\mu,\rho)\in]0,+\infty[\times]0,+\infty[$ and for any $(\mathbf q,\mathbf c)\in\mathcal Q\times E_2(\mu)$, the Lie algebra $\mathrm{Lie}_{(\mathbf q,\mathbf{c})}({\mathcal Y}_2)$ is equal to the whole tangent space $T_{(\mathbf q,\mathbf c)}(\mathcal Q\times E_2({\mu}))$.
\end{prop}
\begin{proof}
The proof is very similar to that of Proposition~\ref{PRO_CalculAlgebreLieDim4}, so we only give the outline:

\noindent{\bf Step 1:} Since the expressions of the vector fields in $\mathcal Y_2$ do not depend on $\mathbf r$, the same arguments as those of step 1 in the proof of Proposition~\ref{PRO_CalculAlgebreLieDim4}, allow us to prove that for all $(\rho,\mu)\in(]0,+\infty[\times]0,+\infty[)\setminus\widehat{\mathfrak B}$, the dimension of $\mathrm{Lie}_{(\mathbf q,\mathbf{c})}({\mathcal Y}_2)$ is constant on $\mathcal Q\times E_2({\mu})$. We deduce that for any fixed pair $(\mu,\rho)\in(]0,+\infty[\times]0,+\infty[)\setminus\widehat{\mathfrak B}$, it is enough to prove that there exists at least one point $(\mathbf q^\ast,\mathbf c^\ast)\in \mathcal Q\times E(\mu)$ for which $\mathrm{Lie}_{(\mathbf q^\ast, \mathbf c^\ast)}({\mathcal Y}_2)$ has dimension 6 to prove that $\mathrm{Lie}_{(\mathbf q,\mathbf{c})}({\mathcal Y}_2)= T_{(\mathbf q,\mathbf c)}( \mathcal Q\times E_2({\mu}))$ for all $(\mathbf q,\mathbf c)\in \mathcal Q\times E_2({\mu})$.

\noindent{\bf Step 2} and {\bf Step 3:} We define $Y^j_\bullet:=\rho_f^{-3}\det \mathbb M^r Y^j$ for all $1\leq j\leq 4$ and denote 
$\mathcal Y^\bullet_2:=(Y^j_\bullet)_{1\leq j\leq 4}$. Using the same notation as in the proof of Proposition~\ref{PRO_CalculAlgebreLieDim4}, we consider the points $(\mathbf q_0,\mathbf c_+)$ and  $(\mathbf q_0,\mathbf c_-)$ with $\mathbf q_0:=(0,0,0)^T\in\mathcal Q$. Using Maxima and Maple, we next prove that for all but maybe a finite number of pairs $(\mu,\rho)\in(]0,+\infty[\times]0,+\infty[)\setminus\widehat{\mathfrak B}$, at least one out of the two families: 
\begin{multline*}\big\{Y^1_{\bullet}(\mathbf q_0,\mathbf c_+), Y^2_{\bullet}(\mathbf q_0,\mathbf c_+), [Y^1_{\bullet},Y^2_{\bullet}](\mathbf q_0,\mathbf c_+),\\
[Y^1_{\bullet},[Y^1_{\bullet}, Y^2_{\bullet}]](\mathbf q_0,\mathbf c_+), [Y^2_{\bullet},[Y^1_{\bullet},Y^2_{\bullet}]](\mathbf q_0,\mathbf c_+), 
[[Y^1_{\bullet},Y^2_{\bullet}],[Y^1_{\bullet},Y^2_{\bullet}]](\mathbf q_0,\mathbf c_+)\big\}
\end{multline*}
and
\begin{multline*}\big\{Y^1_{\bullet}(\mathbf q_0,\mathbf c_-), Y^2_{\bullet}(\mathbf q_0,\mathbf c_-), [Y^1_{\bullet},Y^2_{\bullet}](\mathbf q_0,\mathbf c_-), \\
[Y^1_{\bullet},[Y^1_{\bullet}, Y^2_{\bullet}]](\mathbf q_0,\mathbf c_-), [Y^2_{\bullet},[Y^1_{\bullet},Y^2_{\bullet}]](\mathbf q_0,\mathbf c_-), 
[[Y^1_{\bullet},Y^2_{\bullet}],[Y^1_{\bullet},Y^2_{\bullet}]](\mathbf q_0,\mathbf c_-)\big\},
\end{multline*}
spans a $6$-dimensional vector space. 
 \end{proof}
We define $\mathfrak B$ as the (possibly empty) subset of $]0,+\infty[\times]0,+\infty[$ (containing $\widehat{\mathfrak B}$) that is excluded in the statement of Proposition \ref{PRO_calcul_Lie_dim_6}. 
\subsection{Proof of Theorem~\ref{PRO_NDtracking}}
\label{proof_theorem_42}
We perform now the proof of Theorem~\ref{PRO_NDtracking} in the general case.
\subsubsection{A distribution on $\mathcal E_N(\mu)$}
Fix $N$ an integer ($N\geq 2$) and let us define $\mathfrak X_N$ a set of $n:=2N(N-1)$ vector fields on $\mathcal S_N$. 
The vector fields composing $\mathfrak X_N$ are denoted $\mathbf X^{(k_0,k_1,1)}:=(X^{(k_0,k_1,1)}_k)_{k\geq 1}$ and  $\mathbf X^{(k_0,k_1,2)}:=(X^{(k_0,k_1,2)}_k)_{k\geq 1}$ with $1\leq k_0\neq k_1\leq N$. For any couple $(k_0,k_1)$ of distinct integers less than $N$, all of the components of the fields $\mathbf X^{(k_0,k_1,1)}$ and $\mathbf X^{(k_0,k_1,2)}$ are constant equal to zero, but (maybe) the four following entries:  
\begin{align*}
X^{(k_0,k_1,1)}_{k_0}=& -\left ({k_1}^{2}  a_{k_0} a_{k_1} + k_1 a_{k_0} a_{k_1}-{k_0}^{2} b_{k_0} b_{k_1} - k_0 b_{k_0} b_{k_1} \right ) \\
&- i \left (k_1^2 a_{k_1} b_{k_0} + k_1 a_{k_1} b_{k_0} -k_0^2 a_{k_0} b_{k_1}-k_0 a_{k_0} b_{k_1}\right ),\\
X^{(k_0,k_1,1)}_{k_1}=&\, k_0 (k_1 +1)(a_{k_0}^2+b_{k_0}^2),\\
X^{(k_0,k_1,2)}_{k_0}=&\left (-k_1^2 a_{k_0} b_{k_1} - k_1 a_{k_0} b_{k_1} +k_0^2 a_{k_1} b_{k_0} +k_0 a_{k_1} b_{k_0}  \right ) \\
&+ i \left ( -k_1^2 b_{k_0} b_{k_1}- k_1 b_{k_0} b_{k_1} -k_0^2 a_{k_0} a_{k_1}-k_0 a_{k_0} a_{k_1}\right ),\\
X^{(k_0,k_1,2)}_{k_1}= &\,i k_0 (k_1 +1) (a_{k_0}^2+b_{k_0}^2).
\end{align*}
One can verify that $\mathfrak X_N$ is allowable in the sense of Definition~\ref{allowable:fields}.
\begin{rem}
\label{embeded:XN}
The fields  $\mathbf X^{(1,2,1)}$, $\mathbf X^{(1,2,2)}$, $\mathbf X^{(2,1,1)}$ and $\mathbf X^{(2,1,2)}$ extend to $\mathcal S_N$ ($N\geq 2$) the fields $\mathbf X^1$, $\mathbf X^2$, $\mathbf X^3$ and $\mathbf X^4$ defined in $\mathcal S_2$ by relations \eqref{def:champ:vecteurs}.

More generally, for all integer $N'$ such that $2\leq N'\leq N$, we have $\mathfrak X_{N'}\subset\mathfrak X_N|_{\mathcal S_{N'}}$.
\end{rem}
As in Paragraph~\ref{Para:restatement}, we introduce $\mathfrak Y_N$, the set composed of the vector fields $\mathbf Y^{k_0,k_1,j}$ defined for all $(\mathbf q, \mathbf c)$ in $\mathcal Q\times \mathcal E_N (\mu)$ by:
$$ \mathbf Y^{k_0,k_1,j}(\mathbf{q},\mathbf{c}):=\begin{bmatrix}
-\mathcal R(\theta)(\mathbb M^r(\mathbf c))^{-1}\langle\mathbb N(\mathbf c), \mathbf X^{(k_0,k_1,j)}(\mathbf c)\rangle\\
 \mathbf X^{(k_0,k_1,j)}(\mathbf c)
\end{bmatrix},
$$
for $j=1,2$ and $1\leq k_0\neq k_1\leq N$ and we rewrite \eqref{EQ_control_system_dim4:1}:
\begin{equation}
\label{EQ_control_system_dimN}
 \frac{d}{dt}\begin{bmatrix}
\mathbf q\\
\mathbf c
\end{bmatrix}=\sum_{1 \leq k_0 \neq k_1 \leq N,\atop j=1,2}^n\lambda_{(k_0,k_1,j)}(t) \mathbf Y^{(k_0,k_1,j)}(\mathbf{q},\mathbf c),
\end{equation}
where $\lambda_{(k_0,k_1,j)}:[0,T]\to\mathbf R$ are piecewise constant functions.
\begin{rem}
\label{embeded:YN}
Like in Remark~\ref{embeded:XN}, for all integers $N'$ such that $1\leq N'\leq N$, we have $\mathfrak Y_{N'}\subset \mathfrak Y_N|_{\mathcal S_{N'}}$.
\end{rem}
As already mentioned in Subsection~\ref{subsect:N=2}, via the identification between $\mathcal S_N$ and $S_N:=\mathbf R^{2N}$, we can identify $\mathfrak X_N$ with $\mathcal X_N$, whose elements are denoted $X^{(k_0,k_1,j)}$ and $X^{(k_0,k_1,j)}$, $j=1,2$, $1\leq k_0\neq k_1\leq N$. 
The system \eqref{EQ_control_system_dimN} turns out to be a finite dimensional control system on the analytic $2N+2$-dimensional submanifold $\mathcal Q\times E_N(\mu)$ in $\mathbf R^{2N+3}$ (as usually, $\mu:=\|\mathbf c_0\|_{T_N}$, where $(\mathbf q_0,\mathbf c_0)\in\mathcal Q\times S_N$ is the initial condition of System \eqref{EQ_control_system_dimN}).
\subsubsection{Computation of the Lie algebras}
The following Proposition is a generalization of Proposition~\ref{PRO_calcul_dim_3} to any dimension $N\geq 2$.
\begin{prop}
For all $N\geq 2$ and for any $\mu>0$, the family $\mathcal X_N$ is completely nonholonomic on $E_N({\mu})$, that is, for any $\mathbf{c}$ in $\mathbf R^{2N}$, $\mathbf c\neq 0$, $\mathrm{Lie}_{\mathbf{c}}(\mathcal X_N)=T_{\mathbf{c}}{E_N({\mu})}$ where $\mu:=\|\mathbf c\|_{T_N}$.
\end{prop}
\begin{proof}
Let $\mathbf c$ be as in the Proposition. Since $\mathbf{c}$ is not zero, there exists $k_0$ such that $a_{k_0}^2+b_{k_0}^2\neq 0$. The vector
 space $\mathrm{span}\{X^{(k_0,k,j)}(\mathbf c)\,:\,k\neq k_0,\, j=1,2\}$ has dimension $2N-2$. To prove that
the Lie algebra $\mathrm{Lie}_{\mathbf{c}}(\mathcal X_N)$ has dimension $2N-1$, it is
 enough to check that $\langle F_N(\mathbf{c}), [X^{(k_0,k_1,1)},X^{(k_0,k_1,2)}] \rangle \neq 0$ for at
 least one $k_1$ (the mapping $F_N$ being defined in the Appendix, Subsection~\ref{banach:series}). A simple computation shows that  $\langle F_N(\mathbf{c}),
 [X^{(k_0,k_1,1)},X^{(k_0,k_1,2)}] \rangle =-2(b_{k_0}^2+a_{k_0}^2) k_1(k_2+1)(b_{k_1}^2 k_1+a_{k_1}^2 k_1+b_{k_0}^2 k_0+a_{k_0}^2 k_0)< 0$ for any choice of $k_1 \neq k_0$, and the proof is then completed.
\end{proof}
As in Paragraph~\ref{subs:comput:lie:alg}, we next define on $\mathbf R/2\pi\times\mathcal S_N$ the set of vector fields $\widehat{\mathfrak X}_N$ which can be identified with $\widehat{\mathcal X}_N$ defined on $\mathbf R/2\pi\times S_N$. According to our usual notation, we denote $\widehat X^{(k_0,k_1,j)}$ ($j=1,2$, $1\leq k_0\neq k_1\leq N$) the $n$ elements of $\mathcal X_N$. We can restate Proposition~\ref{PRO_CalculAlgebreLieDim4} in the general $N$-dimensional case:
\begin{prop}\label{PRO_CalculAlgebreLieDimN} For all integer $N\geq 2$, for all of the pairs $(\mu,\rho)\in(]0,+\infty[\times]0,+\infty[)\setminus\widehat{\mathfrak B}$ and for any $(\theta,\mathbf c)\in\mathbf{R}/2\pi\times E_N(\mu)$, the Lie algebra $\mathrm{Lie}_{(\theta,\mathbf{c})}(\widehat{\mathcal X}_N)$ is equal to the whole tangent space $T_{(\theta,\mathbf c)}( \mathbf{R}/2\pi\times E_N({\mu}))$. 
\end{prop}
\begin{proof}
When $N=2$, the result is given by Proposition~\ref{PRO_CalculAlgebreLieDim4} so let us assume that $N\geq 3$. As in the proof of Proposition~\ref{PRO_CalculAlgebreLieDim4}, we first establish that for all $N\geq 3$ and all $(\mu,\rho)\in]0,+\infty[\times]0,+\infty[$,
the dimension of  $\mathrm{Lie}_{(\theta, \mathbf c)}(\widehat{\mathcal X}_N)$ is constant on $\mathbf R/2\pi\times E_N(\mu)$. We next 
define $\widehat{\mathcal X}^\star_N:=\{\widehat X^{(1,k_1,j)},\,1\leq j\leq 2,\,3\leq k_1\leq N\}\subset \widehat{\mathcal X}_N$.
Observe that for all $N\geq 3$ and for all $(\mu,\rho)\in]0,+\infty[\times]0,+\infty[$, $\mathcal E_2(\mu)\subset \mathcal E_N(\mu)$  and hence $E_2(\mu)$ can be seen as an immersed submanifold of $E_N(\mu)$. We can easily verify that for all $(\theta,\mathbf c)\in \mathbf R/2\pi\times E_2(\mu)$ ($E_2(\mu)$ seen as a subset of $E_N(\mu)$), ${\rm dim}\big({\rm span}_{(\theta,\mathbf c)}(\widehat{\mathcal X}_N^\star)\big)=2(N-2)$ and ${\rm Lie}_{(\theta,\mathbf c)}(\widehat{\mathcal X}_2)\cap {\rm span}_{(\theta,\mathbf c)}(\widehat{\mathcal X}_N^\star)=\{0\}$, whence we deduce that:
$${\rm Lie}_{(\theta,\mathbf c)}(\widehat{\mathcal X}_2)+{\rm span}_{(\theta,\mathbf c)}(\widehat{\mathcal X}_N^\star)={\rm Lie}_{(\theta,\mathbf c)}(\widehat{\mathcal X}_2)\oplus{\rm span}_{(\theta,\mathbf c)}(\widehat{\mathcal X}_N^\star)\subset {\rm Lie}_{(\theta,\mathbf c)}(\mathcal X_N).$$
But in Proposition~\ref{PRO_CalculAlgebreLieDim4}, we have proved that for all of the pairs $(\mu,\rho)\in(]0,+\infty[\times]0,+\infty[)\setminus\widehat{\mathfrak B}$ and for any $(\theta,\mathbf c)\in\mathbf{R}/2\pi\times E_2(\mu)$, ${\rm dim}\big(\mathrm{Lie}_{(\theta,\mathbf{c})}(\widehat{\mathcal X}_2)\big)=4$. We deduce that ${\rm dim}\big(\mathrm{Lie}_{(\theta,\mathbf{c})}(\widehat{\mathcal X}_N)\big)\geq 2N,$ which is the dimension of $\mathbf R/2\pi\times E_N(\mu)$, and the proof is completed.
\end{proof}
We can also generalize Proposition~\ref{PRO_calcul_Lie_dim_6} as follows:
\begin{prop}
\label{prop:lie_algebra:dimN}
For all $N\geq 2$, for all pairs $(\mu,\rho)\in(]0,+\infty[\times]0,+\infty[)\setminus\mathfrak B$ and for any $(\mathbf q,\mathbf c)\in\mathcal Q\times E_N(\mu)$, the Lie algebra $\mathrm{Lie}_{(\mathbf q,\mathbf{c})}({\mathcal Y}_N)$ is equal to the whole tangent space $T_{(\mathbf q,\mathbf c)}(\mathcal Q\times E_N({\mu}))$.
\end{prop}
\begin{proof}
The proof is similar to the proof of Proposition~\ref{PRO_CalculAlgebreLieDimN}, reusing the results of Proposition~\ref{PRO_calcul_Lie_dim_6}.
\end{proof}
\subsubsection{Regularization}
According to Proposition~\ref{PRO_tracking} together with Proposition~\ref{prop:lie_algebra:dimN}, we have proved so far, that for any integer $N\geq 2$, for any pair $(\mu,\rho)\in(]0,+\infty[\times]0,+\infty[)\setminus\mathfrak B$, for every $\varepsilon>0$ and for every reference continuous curve $({\mathbf{q}}^\dagger,{\mathbf c}^\dagger):[0,T]\rightarrow \mathcal Q \times {\mathcal E}_N(\mu)$, there exists $2N(N-1)$ piecewise constant functions $(\lambda_{k_0,k_1,j})_{1\leq k_0 \neq k_1 \leq N,\atop j=1,2}:[0,T]\rightarrow \mathbf{R},$ such that  the solution $(\mathbf{q},\mathbf{c})$ of \eqref{EQ_control_system_dimN} starting from $({\mathbf{q}}^\dagger(0),{\mathbf c}^\dagger(0))$ reaches $({\mathbf{q}}^\dagger(T),{\mathbf c}^\dagger(T))$ at time $T$ and remains $\varepsilon$-close to the reference curve $({\mathbf{q}}^\dagger,{\mathbf c}^\dagger)$ for all time between $0$ and $T$.

The proof of Theorem~\ref{PRO_NDtracking} follows, because the analytic real functions are dense for the $L^1([0,T])$ norm in the set of measurable bounded functions. One can therefore approximate the piecewise constant control functions $(\lambda_j)_{k_0,k_1,j}$ ($1\leq k_0\neq k_1\leq N,\,j=1,2$) by a suitable family of analytic functions.
\subsection{Proof of Theorem~\ref{THE_diminf_tracking}}
\label{prooftheorem41}
\subsubsection{Finite dimensional approximation}
\begin{prop}\label{PRO_approx_dim_infinie}
Let $\varepsilon>0$, $\rho:=\rho_0/\rho_f>0$ and $\mu^\dagger\in]0,1[$ be given and $\mathbf{c}^\dagger:[0,T]\rightarrow {\cal E}^\bullet(\mu^\dagger)$ be a continuous curve (not necessarily physically allowable). Then, there exist ${\mu}\in]0,1[$ such that $(\mu,\rho)\in(]0,+\infty[\times]0,+\infty[)\setminus\mathfrak B$, an integer $N\geq 2$ and a continuous curve ${\mathbf{c}}:[0,T]\rightarrow {\cal E}_N({\mu})$ such that
$\|{\mathbf{c}}(t)-{\mathbf{c}}^\dagger(t)  \|_{\mathcal S} <\varepsilon$ for every $t$ in $[0,T]$. 
\end{prop}

\begin{proof}
For every integer $N\geq 2$ and every $\varepsilon'>0$ define $\Theta_N:=\{t \in [0,T]\,:\,\|{\mathbf{c}}^\dagger(t) -\Pi_N{\mathbf{c}}^\dagger(t) \|_{\mathcal S}<\varepsilon'\}$ (the projector $\Pi_N$ from $\mathcal S$ onto $\mathcal S_N$ is defined in the Appendix, Subsection~\ref{banach:series}). 
 Because ${\mathbf c}^\dagger$ is continuous, the set $\Theta_N$ is open in $[0,T]$ for all $N\geq 2$ and since for any $t\in [0,T]$, $\Pi_N{\mathbf{c}}^\dagger(t)\to {\mathbf{c}}^\dagger(t)$ as $N\to\infty$, we deduce that $[0,T]\subset \cup_{N\geq 1}\Theta_N$. The interval $[0,T]$ being compact and the sequence $(\Theta_N)_{N\geq 1}$ non-decreasing, $[0,T]\subset\Theta_N$ for some $N$ (depending on $\varepsilon'$). However, we cannot yet choose $\Pi_N{\mathbf{c}}^\dagger$ as a {\it good} finite dimensional approximation of ${\mathbf c}^\dagger$ because $\|\Pi_N{\mathbf{c}}^\dagger(t)\|_{\mathcal T}$ is certainly not constant. Since $\|\mathbf c\|_{\mathcal T}\leq \|\mathbf c\|_{\mathcal S}$ for all $\mathbf c\in\mathcal S$, we have:
\begin{equation}
\label{vert}
\mu^\dagger-\varepsilon'<\|\Pi_N\mathbf c^\dagger(t)\|_{\mathcal T}\leq \mu^\dagger,\quad\forall\,t\in[0,T].
\end{equation}
We deduce that $\|\Pi_N\mathbf c^\dagger(t)\|_{\mathcal T}\neq 0$ for all $t\in[0,T]$ if $\varepsilon'$ is chosen small enough. For any ${\mu}$ in $]0,1[$ and for for any $t$ in $[0,T]$ we have next the estimates:
\begin{align*}\Big \|{\mu} \frac{\Pi_N(\mathbf{c}^\dagger(t))}{\|\Pi_N(\mathbf{c}^\dagger(t))\|_{\mathcal T}} - \mathbf{c}^\dagger(t) \Big \|_{\mathcal S} \leq&  \Big \| {\mu} \frac{\Pi_N(\mathbf{c}^\dagger(t))}{\|\Pi_N(\mathbf{c}^\dagger(t))\|_{\mathcal T}} - \Pi_N(\mathbf{c}^\dagger(t)) \Big \|_{\mathcal S} \nonumber \\&+ \| \Pi_N(\mathbf{c}^\dagger(t)) - \mathbf{c}^\dagger(t) \|_{\mathcal S}\nonumber \\ \leq&
\Big | \frac{{\mu}}{\|\Pi_N(\mathbf{c}^\dagger(t))\|_{\mathcal T}} -1 \Big| \|\Pi_N(\mathbf{c}^\dagger(t))\|_{\mathcal S} + \varepsilon'\nonumber \\
\leq&C\Big | \frac{{\mu}}{\|\Pi_N(\mathbf{c}^\dagger(t))\|_{\mathcal T}} -1 \Big| + \varepsilon',
\end{align*}
where $C:=\max_{t\in[0,T]}\|\Pi_N(\mathbf{c}^\dagger(t))\|_{\mathcal S}$.
According to \eqref{vert}, we deduce that, for all $t\in[0,T]$:
\begin{equation}
\label{req:1}
\Big \|{\mu} \frac{\Pi_N(\mathbf{c}^\dagger(t))}{\|\Pi_N(\mathbf{c}^\dagger(t))\|_{\mathcal T}} - \mathbf{c}^\dagger(t) \Big \|_{\mathcal S} \leq C\max\Big\{
\big|\frac{\mu-\mu^\dagger}{\mu^\dagger}\big|,\big|\frac{\mu-\mu^\dagger+\varepsilon'}{\mu^\dagger-\varepsilon'}\big|\Big\}+\varepsilon'.
\end{equation}
For any $\varepsilon>0$, one can always choose $\varepsilon'$ small enough and $\mu$ satisfying the requirements of the Proposition, such that the right hand side 
of  \eqref{req:1} be smaller that $\varepsilon$. We conclude the proof by choosing $\mathbf c:=\mu\Pi_N(\mathbf{c}^\dagger)/\|\Pi_N(\mathbf{c}^\dagger)\|_{\mathcal T}$.
\end{proof}
\subsubsection{Conclusion of the proof}
Let ${\mu}^\dagger\in ]0,1[$, $\varepsilon>0$ and a reference continuous curve $({\mathbf{q}}^\dagger,{\mathbf c}^\dagger):[0,T]\rightarrow \mathcal Q \times { \cal E}^\bullet({\mu}^\dagger)$ be given. Apply next Proposition~\ref{PRO_approx_dim_infinie} with $\varepsilon/2$ to obtain an integer $N\geq 2$, $\mu\in]0,1[$ and a continuous curve $\mathbf c^\ddagger:[0,T]\mapsto\mathcal E_N(\mu)$ such that $\|\mathbf c^\dagger(t)-\mathbf c^\ddagger(t)\|_{\mathcal S}<\varepsilon/2$. Finally, apply Theorem~\ref{PRO_NDtracking} with $\varepsilon/2$ and reference curve $(\mathbf q^\dagger,\mathbf c^\ddagger):[0,T]\mapsto\mathcal Q\times\mathcal E_N(\mu)$, to get the conclusion of Theorem~\ref{THE_diminf_tracking}. Observe that since we are able to find $\mathbf c:[0,T]\mapsto\mathcal E_N(\mu)\subset\mathcal E(\mu)$ such that $\|\mathbf c^\dagger(t)-\mathbf c(t)\|_{\mathcal S}<\varepsilon$ for all $\varepsilon>0$ and since $\mathbf c^\dagger(t)\in\mathcal D$ for all $t\in[0,T]$, we can always assume that $\mathbf c$ is valued in $\mathcal E^\bullet(\mu)$.
\section{Numerical results} 
\label{SEC_numerical}
By integrating equations \eqref{edo:edo} and \eqref{reconst},  we can easily compute the trajectory of the swimming animal. 
Indeed, all of the mass matrices $\mathbb M^r(\mathbf c)$ and $\mathbb N(\mathbf c)$ arising in the ODE have been made explicit in Subsection~\ref{SEC_Expressions_mass_matrices}. 

This Section is accompanied by  a web page containing further animations and numerical experiments and is located at:
\url{http://www.iecn.u-nancy.fr/~munnier/page_amoeba/control_index.html}. We will always choose $\rho_f=1$ (the density of the fluid) and the density of the amoeba will be next computed accordingly based on formulae \eqref{exp:rho0} (neutrally buoyant case). The values of $\mu$ depend on the dimension $N$ of $S_N$ (the Banach space of the control variable $\mathbf c$). They are chosen in such a way that $\mathcal E_N^\bullet(\mu)=\mathcal E_N(\mu)$ (i.e. the ball of center $0$ and radius $\mu$ of $\mathcal T_N$ be included in $\mathcal D$).
All of the animations and figures have been realized with MATLAB. 
\subsection{Swimming using Lie brackets}
Although recourse to Lie brackets can be useful to derive theoretical controllability results, they yield in general a quite inefficient swimming strategy. This can be illustrated by the following example. We consider here the configuration described in Subsection~\ref{subsect:N=2} i.e. we specialize $N=2$, $\mu=1/2$, $\mathbf r_0=(0,0)^T$, $\theta_0=0$ and $\mathbf c_0=(1/2,0)^T$. The shape-changes and the trajectory are given by integrating the EDO \eqref{EQ_control_system_dim4}. We approximate the displacement induced by the Lie bracket $[X^1,X^2]$ by integrating the EDO with first $\lambda_1=1$ and $\lambda_j=0$ for $j=2,3,4$ over a small time interval (of length 0.1) then we set $\lambda_2=1$ and $\lambda_j=0$ for $j=1,3,4$ over a time interval of the same length, next $\lambda_1=-1$, $\lambda_j=0$ for $j=2,3,4$ and finally $\lambda_2=-1$ and $\lambda_j=0$ for $j=1,2,3$. We repeat this process fifty times to obtain the trajectory of the center of mass (i.e. the parameterized line  $\mathbf r(t):=(r_1(t),r_2(t))$ with $t\in[0,20]$) of the amoeba pictured in Figure~\ref{fig:1} while in Figure~\ref{fig:2} we display the shape variables $c_1$ and $c_2$ with respect to time over the time interval $[0,20]$. A movie related to this simulation is given on the web page \url{http://www.iecn.u-nancy.fr/~munnier/page_amoeba/control_index.html}, showing even more clearly how inefficient this swimming strategy is.
\begin{figure}[H]
\centerline{\includegraphics[width=.6\textwidth]{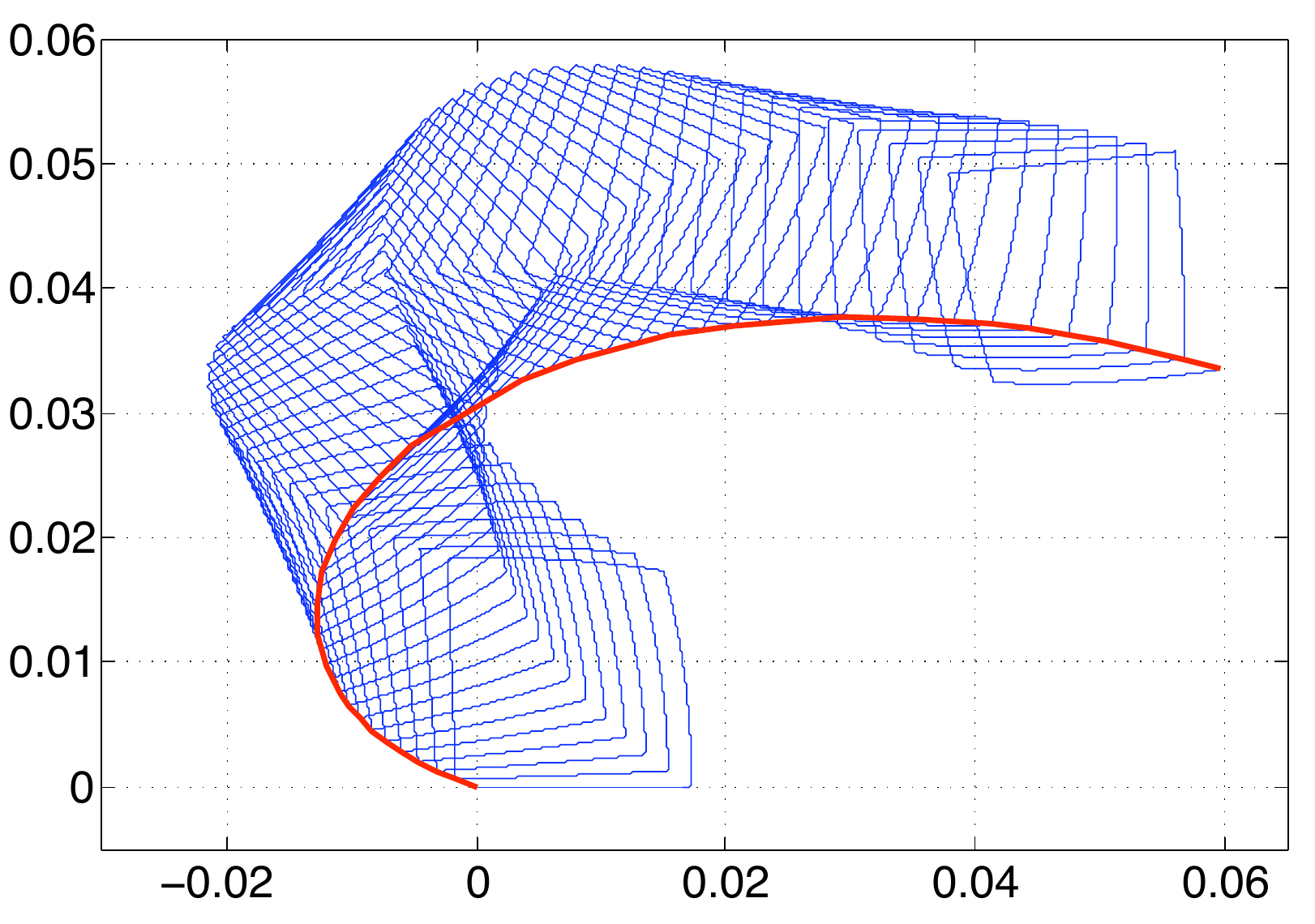}}
\caption{\label{fig:1} Trajectory of the center of mass of the amoeba when the shape-changes are obtained by {\it approximated} Lie brackets. The bold (red) line can be considered as the net displacement of the swimming animal.}
\end{figure}

\begin{figure}[H]
\centerline{\includegraphics[width=.8\textwidth]{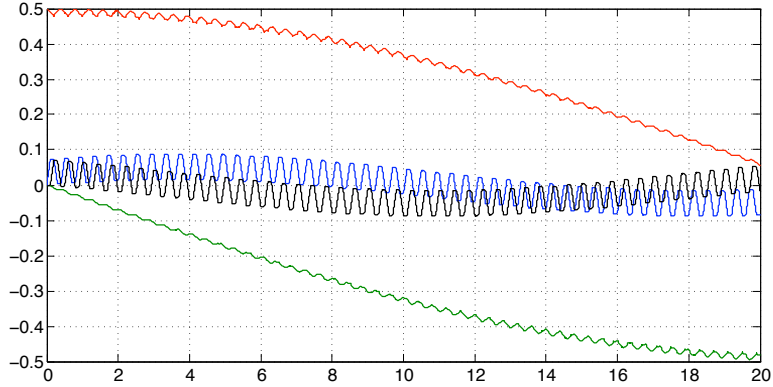}}
\caption{\label{fig:2}Shape variables $c_1(t)=a_1(t)+ib_1(t)$ and $c_2(t)=a_2(t)+ib_2(t)$ with respect to time. At time $t=20$, the lines correspond respectively from top to bottom to $a_1(t)$ (red), $b_2(t)$ (black), $a_2(t)$ (blue) and $b_1(t)$ (green).}
\end{figure}
\subsection{Examples of more efficient swimming strategies}
Throughout this subsection, we specify $N=6$ (i.e. the shape variable reads $\mathbf c=(c_1,\ldots,c_6)^T\in\mathbf C^6$ with $c_k:=a_k+ib_k$ for all $1\leq k\leq 6$). The first difficulty we are faced with in seeking swimming strategies is that the control variable $t\mapsto\mathbf c(t)$ has to be {\it allowable} in the sense of Definition~\ref{alow:cont}. 
To reflect this constraint, we introduce the new functions $t\mapsto\alpha_j(t)$ ($1\leq j\leq 5$) and $t\mapsto h_k(t)$ ($1\leq k\leq 3$) that we choose to be our new control functions and we next define: 
\begin{align*}
R_1(t)&:= \mu\cos(\alpha_1(t)),\\
R_2(t)&:= \frac{\mu}{\sqrt{2}}\sin(\alpha_1(t))\cos(\alpha_2(t)),\\
R_3(t)&:= \frac{\mu}{\sqrt{3}}\sin(\alpha_1(t))\sin(\alpha_2(t))\cos(\alpha_3(t)),\\
R_4(t)&:= \frac{\mu}{2}\sin(\alpha_1(t))\sin(\alpha_2(t))\sin(\alpha_3(t))\cos(\alpha_4(t)),\\
R_5(t)&:= \frac{\mu}{\sqrt{5}}\sin(\alpha_1(t))\sin(\alpha_2(t))\sin(\alpha_3(t))\sin(\alpha_4(t))\cos(\alpha_5(t)),\\
R_6(t)&:= \frac{\mu}{\sqrt{6}}\sin(\alpha_1(t))\sin(\alpha_2(t))\sin(\alpha_3(t))\sin(\alpha_4(t))\sin(\alpha_5(t)),\\
\end{align*}
and
\begin{alignat*}{3}
\theta_1(t) &:= -\frac{1}{3}\int_0^t h_1(s) R_2^2(s){\rm d}s,&\qquad&
\theta_2(t)&:= \frac{1}{2}\int_0^t h_1(s)R_1^2(s){\rm d}s,\\
\theta_3(t) &:= -\frac{1}{5}\int_0^t h_2(s)R_4^2(s){\rm d}s,&&
\theta_4(t) &:= \frac{1}{4}\int_0^t h_2(s)R_3^2(s){\rm d}s,\\
\theta_5(t)&:= -\frac{1}{7}\int_0^t h_3(s)R_6^2(s){\rm d}s,&&
\theta_6(t) &:= \frac{1}{6}\int_0^t h_3(s)R_5^2(s){\rm d}s.
\end{alignat*}
If we set now:
$$
a_k(t):=R_k(t)\cos(\theta_k(t))\quad\text{and}\quad
b_k(t):=R_k(t)\sin(\theta_k(t)),$$
then it can be readily verfied that the function $t\mapsto\mathbf c(t)$ is indeed allowable. 
\subsubsection{Example of a straight-forward motion}
We set $\mu=0.5$, $\alpha_1(t)=t$, $\alpha_j(t)=0$ ($j=2,3,4,5$) and $h_k(t)=0$ ($k=1,2,3$) for all $t\geq 0$ to obtain a net 
straight-forward motion for the amoeba consisting of periodic strokes (of $2\pi$ time period). With these data, the functions $t\mapsto c_k(t)$   ($k=3,4,5,6$) are constant equal to zero for all $t\geq 0$.
Screenshots of the amoeba over a stroke are given in Figure~\ref{strokes} while in Figure~\ref{fig:5} are drawn the $x$-coordinate of the center of mass of the animal and the $x$-coordinate of its velocity with respect to time.
\begin{figure}[H]     
     \centering
     \begin{tabular}{|c|c|c|}
     \hline
     \subfigure
     [\label{fig_time_zero}$t=0$]
     {\includegraphics[width=.3\textwidth]{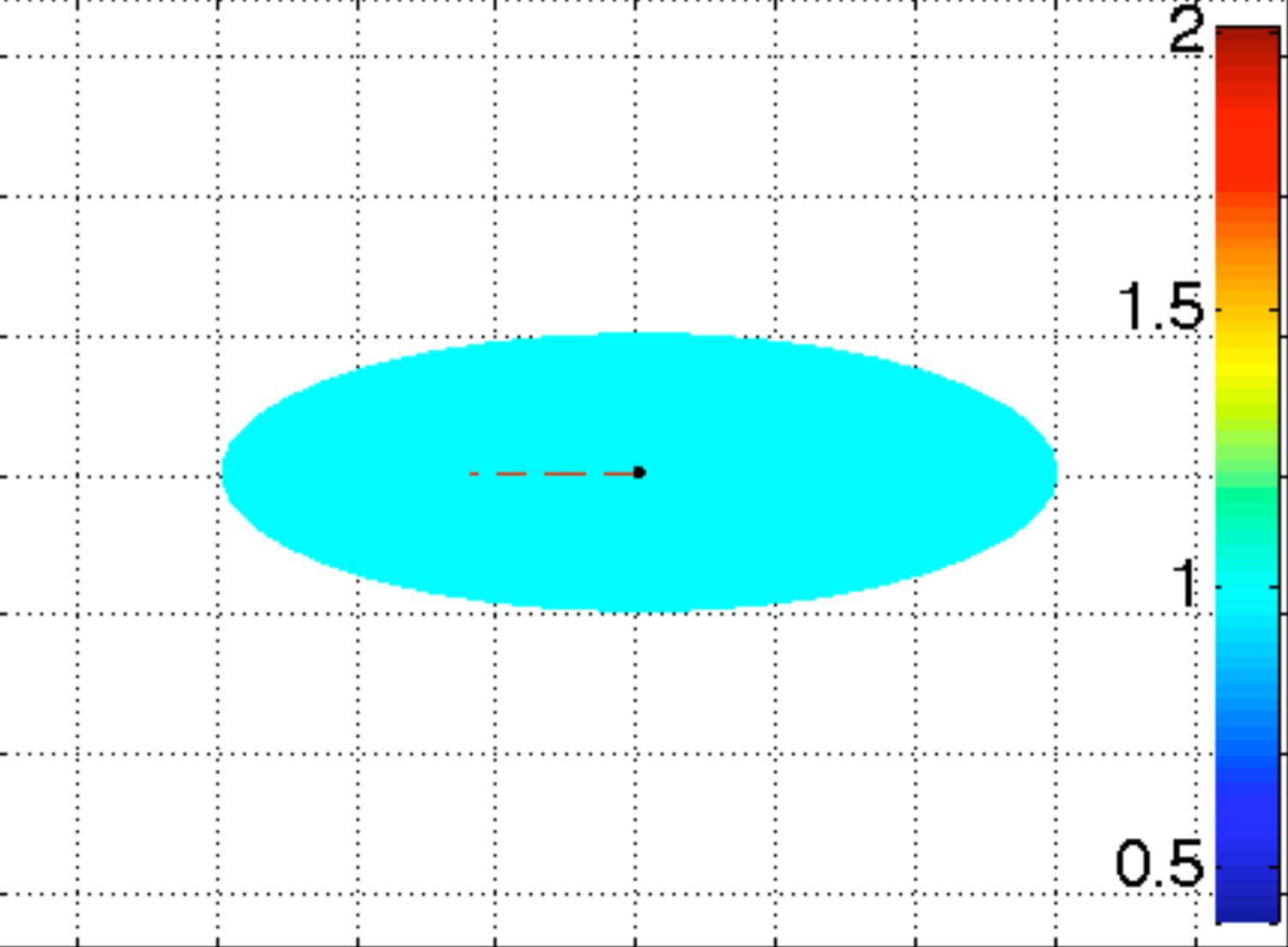}}
     &
     \subfigure
     [$t=\pi/4$]
     {\includegraphics[width=.3\textwidth]{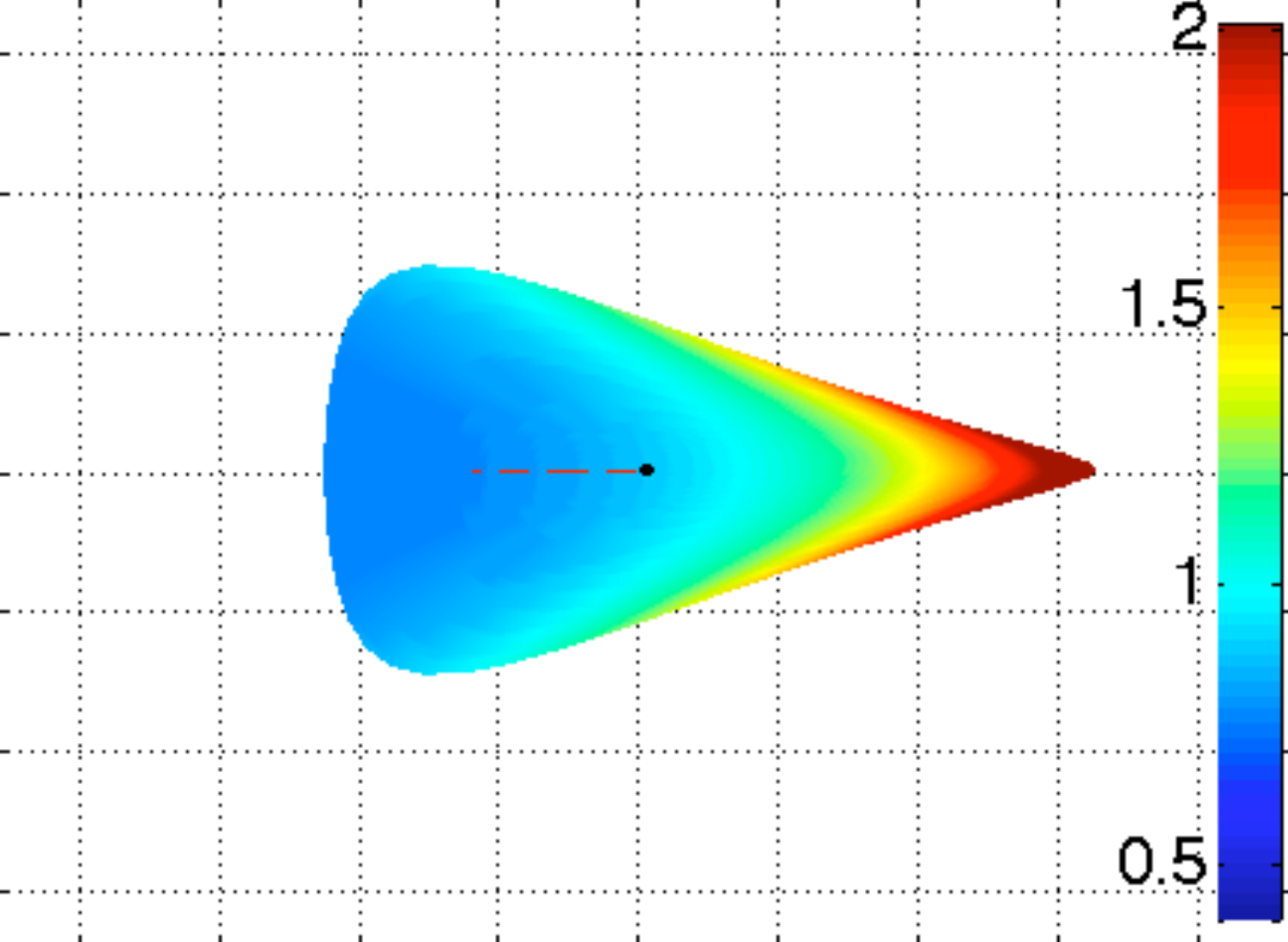}}
     &
     \subfigure
     [$t=\pi/2$]
     {\includegraphics[width=.3\textwidth]{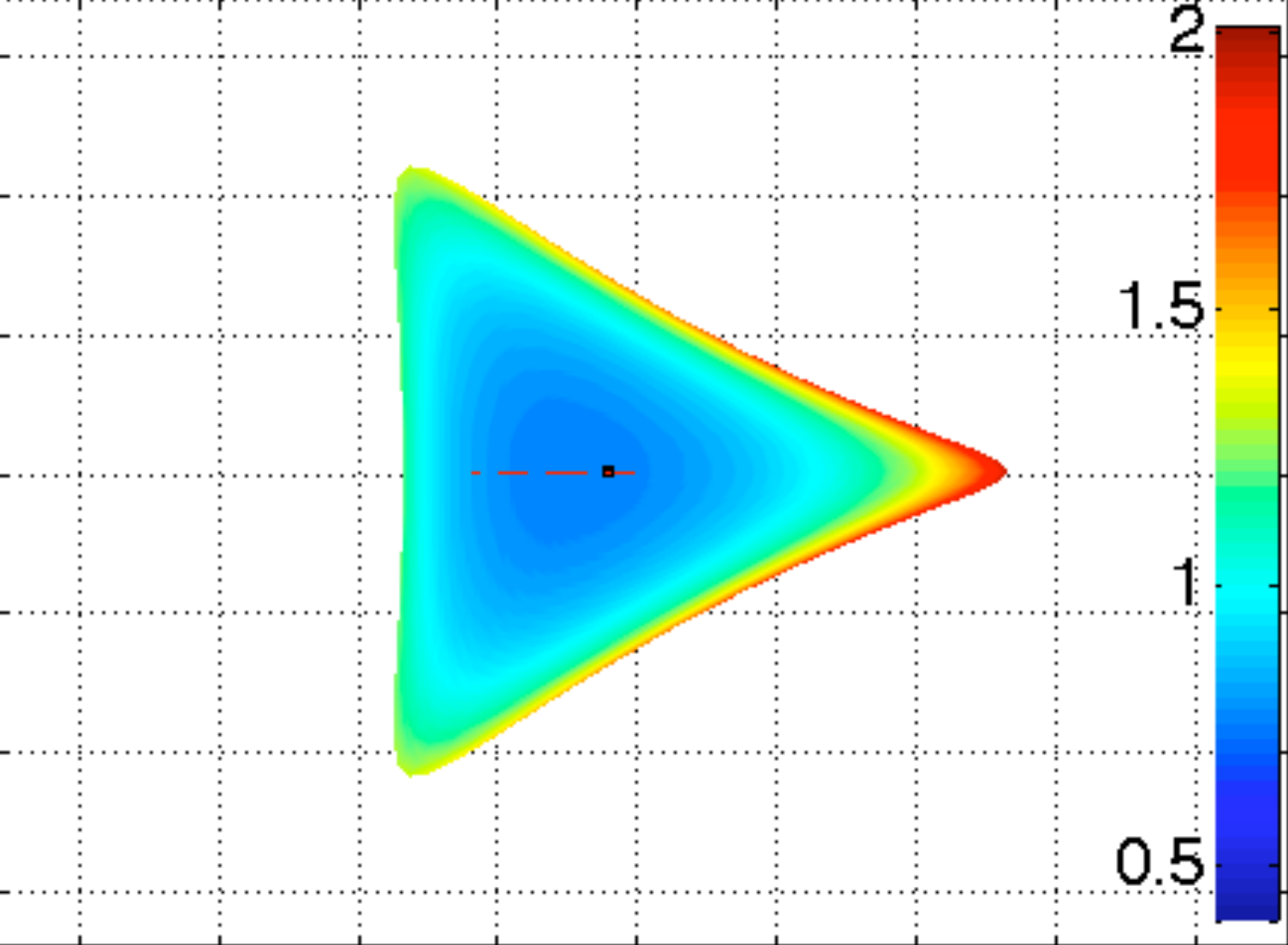}}\\
      \hline
     \subfigure
     [$t=3\pi/2$]
     {\includegraphics[width=.3\textwidth]{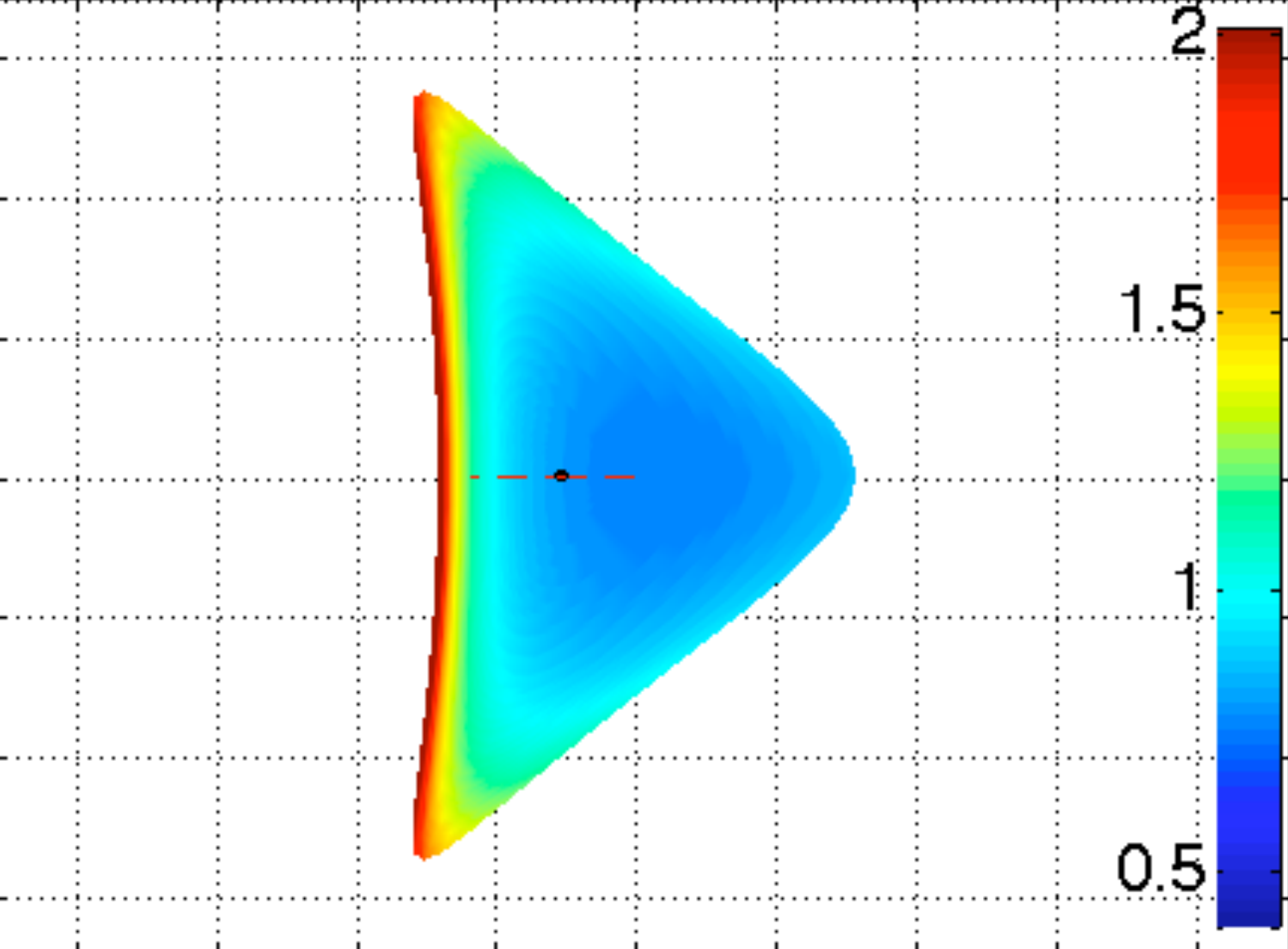}}
     &
     \subfigure
     [$t=\pi$]
     {\includegraphics[width=.3\textwidth]{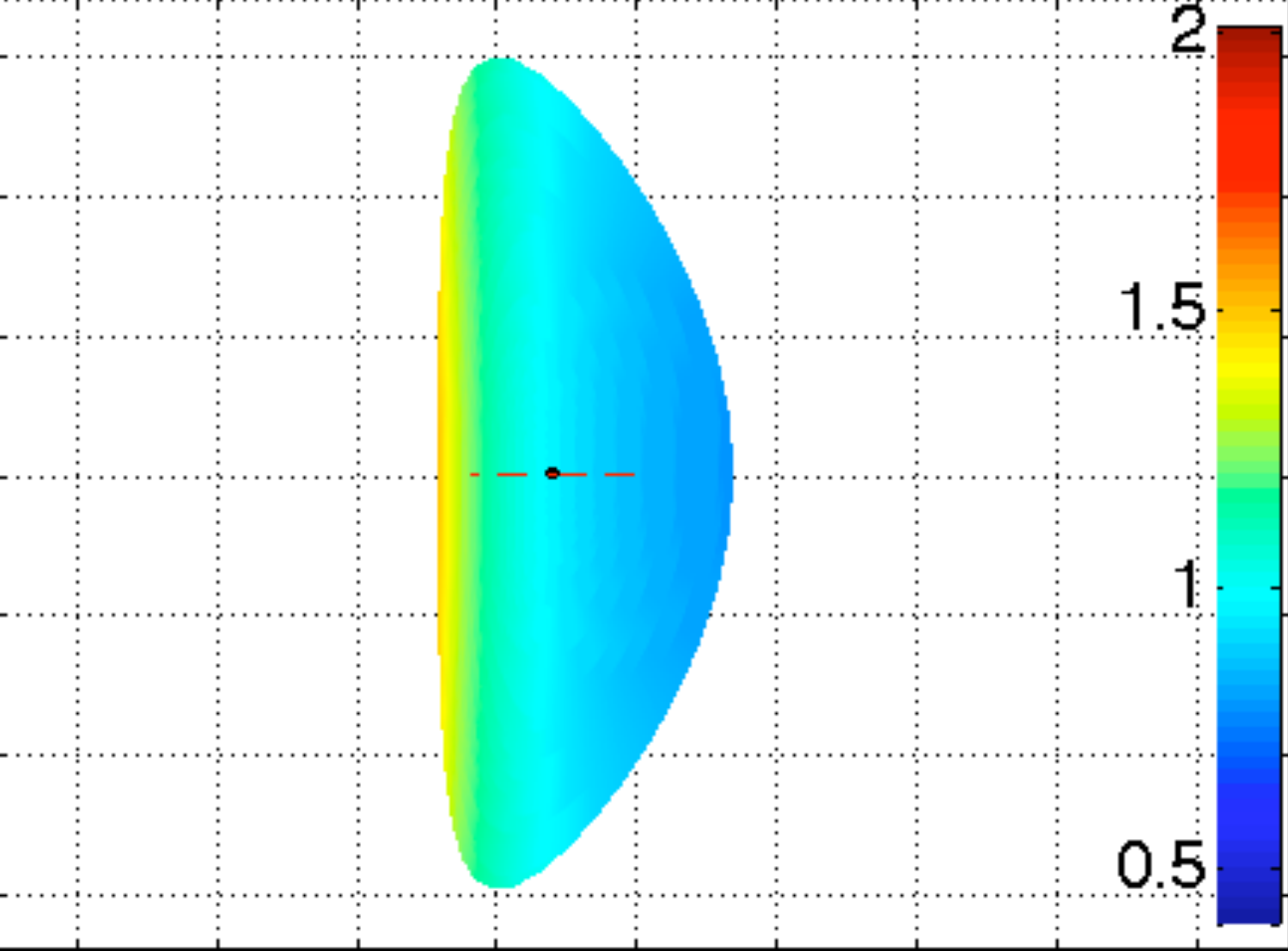}}
     &
     \subfigure
     [$t=5\pi/4$]
     {\includegraphics[width=.3\textwidth]{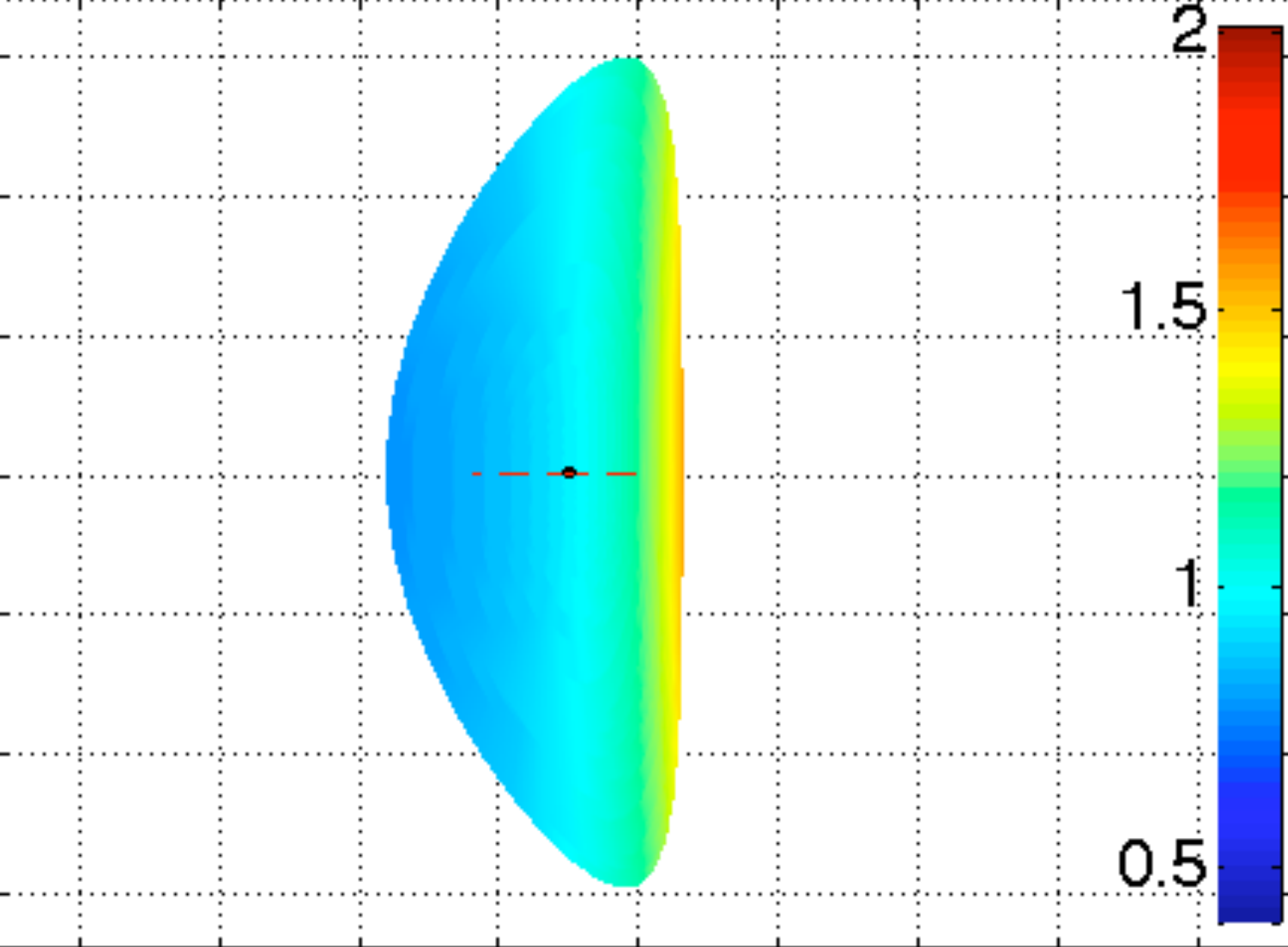}}\\
     \hline
      \subfigure
     [$t=3\pi/2$]
     {\includegraphics[width=.3\textwidth]{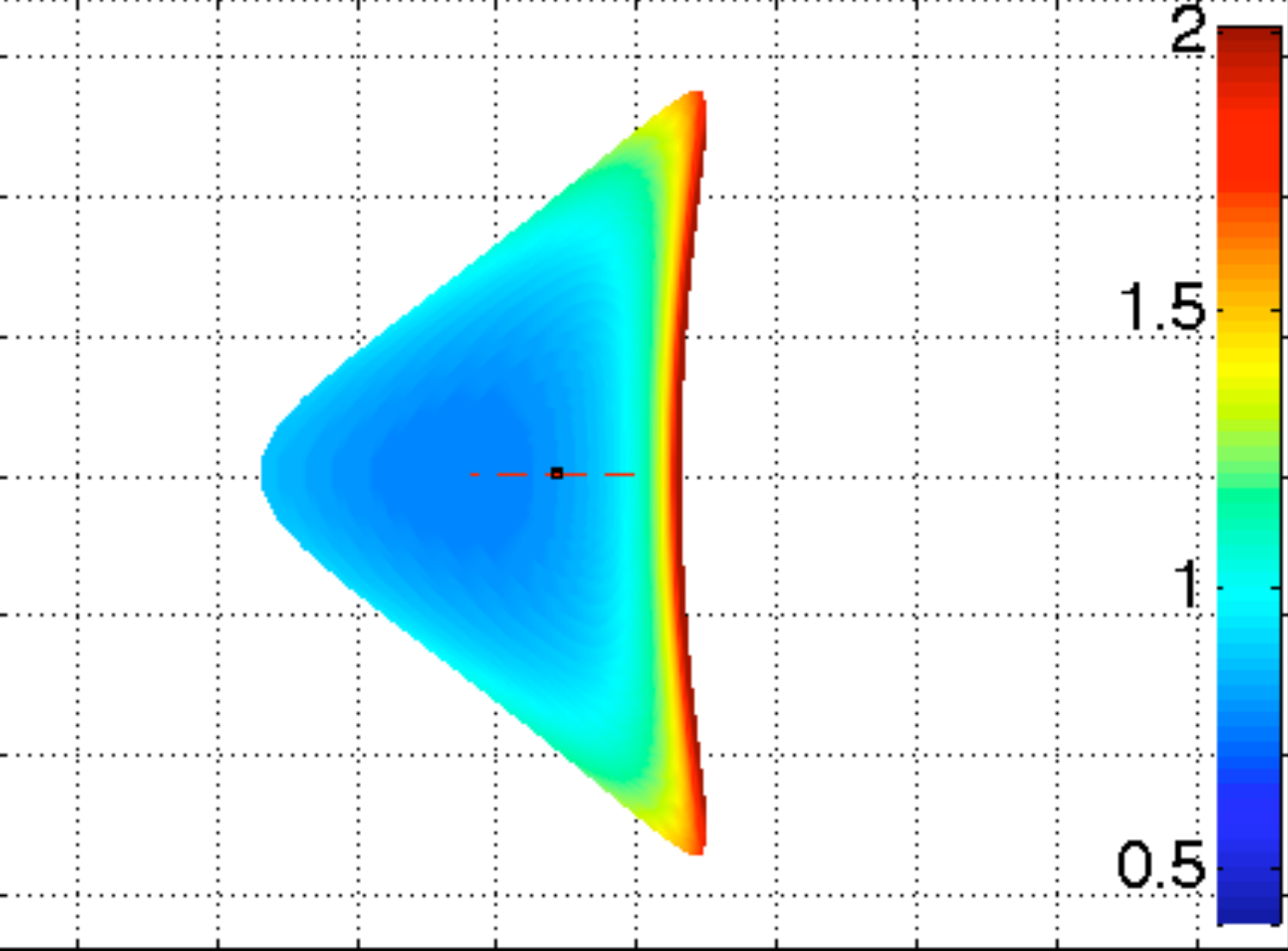}}
     &
     \subfigure
     [$t=7\pi/4$]
     {\includegraphics[width=.3\textwidth]{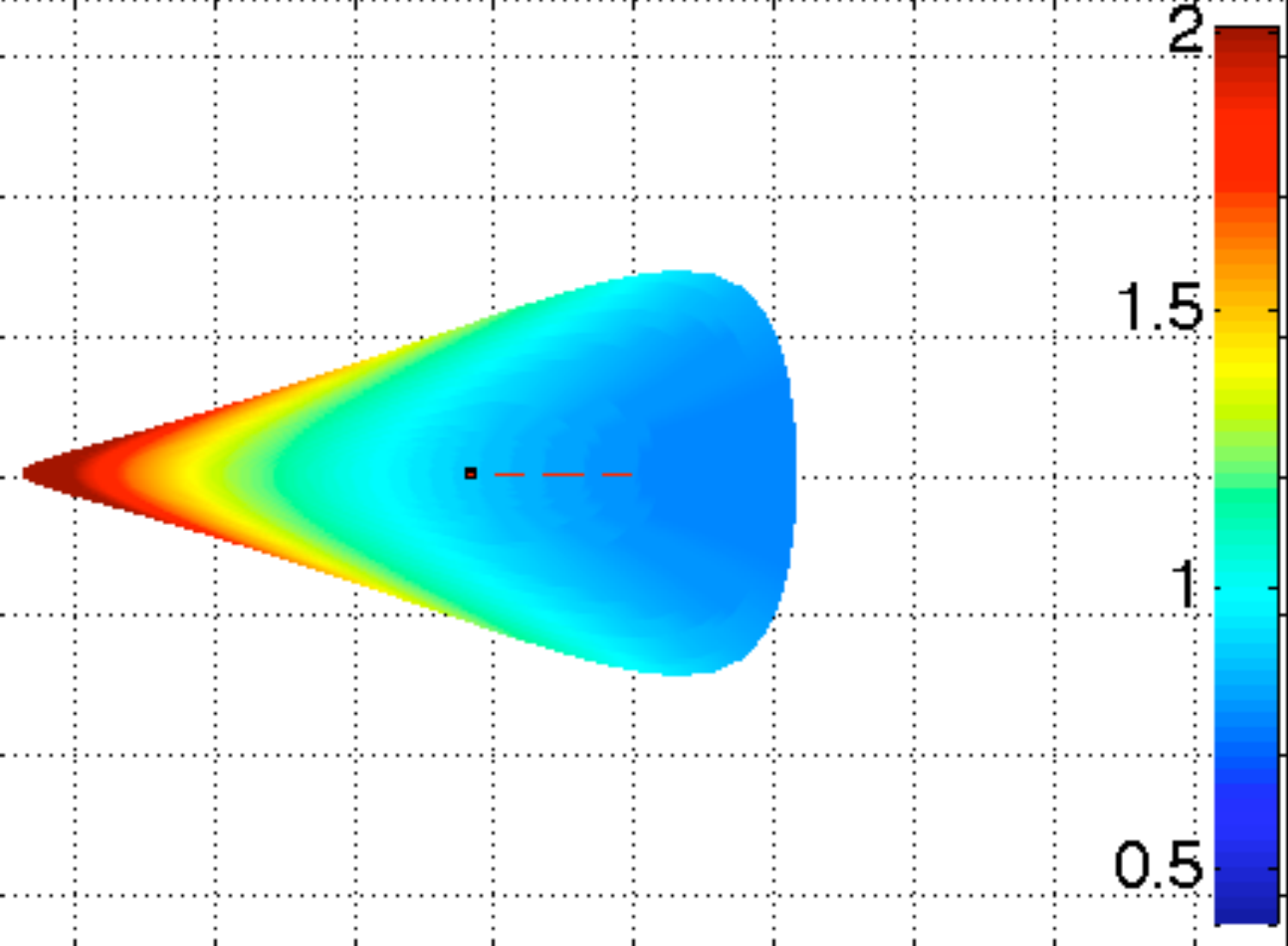}}
     &
     \subfigure
     [$t=2\pi$]
     {\includegraphics[width=.3\textwidth]{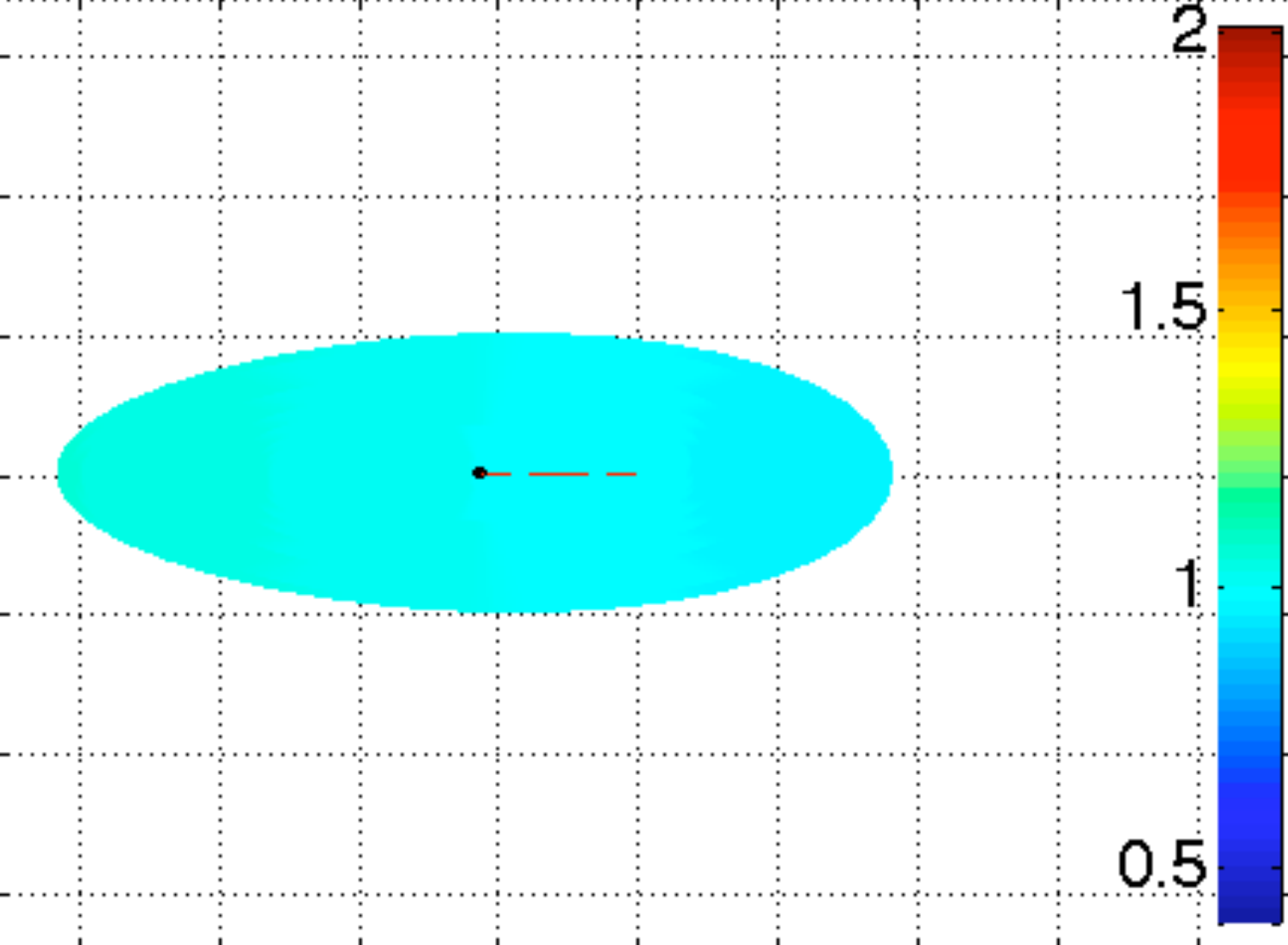}}\\
     \hline
     \end{tabular}
     \caption{\label{strokes}Screenshots of the motion of the amoeba over a stroke. The colours give the value of the internal density. The animal is neutrally buoyant, so at rest its density is 1 (the density of the fluid).}
\end{figure}
\begin{figure}[H]
\centerline{\includegraphics[width=.6\textwidth]{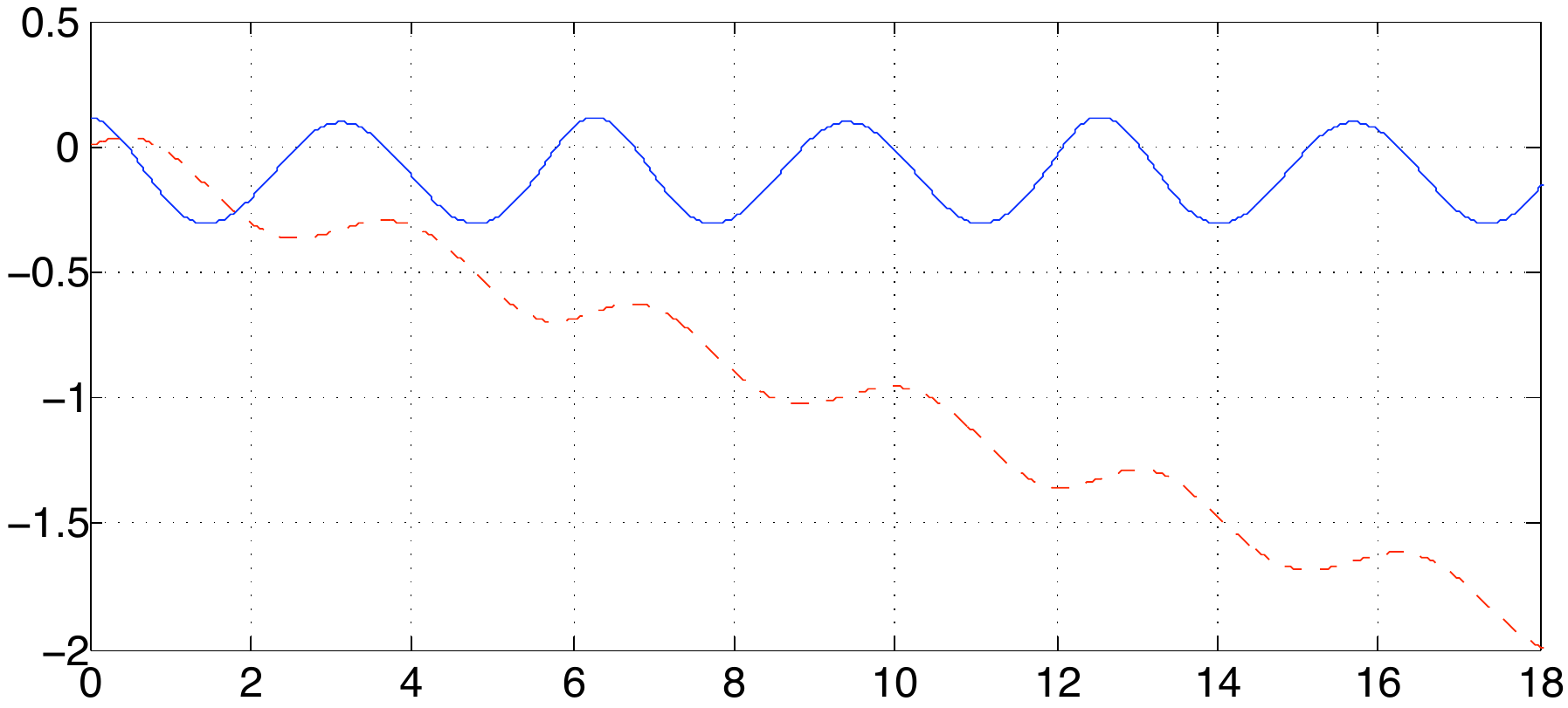}}
\caption{\label{fig:5}$x$-coordinate of the center of mass of the amoeba (dashed line) and $x$-coordinate of its velocity (solid line).}
\end{figure}
\noindent Following the method described in Subsection~\ref{expre:internal:forces}, the shape-changes being given, we can compute the expression of the internal forces. Like the shape-changes, the internal forces are also $2\pi$-periodic as it can be seen in Figure~\ref{fig:6}.
\begin{figure}[H]     
     \centering
     \begin{tabular}{|c|}
     \hline
     \subfigure
     [Controls $a_1(t)$ (solid line) and $a_2(t)$ (dashed line) with respect to time over a stroke. Both controls $b_1(t)$ and $b_2(t)$ are equal to zero.]
     {\includegraphics[width=.6\textwidth]{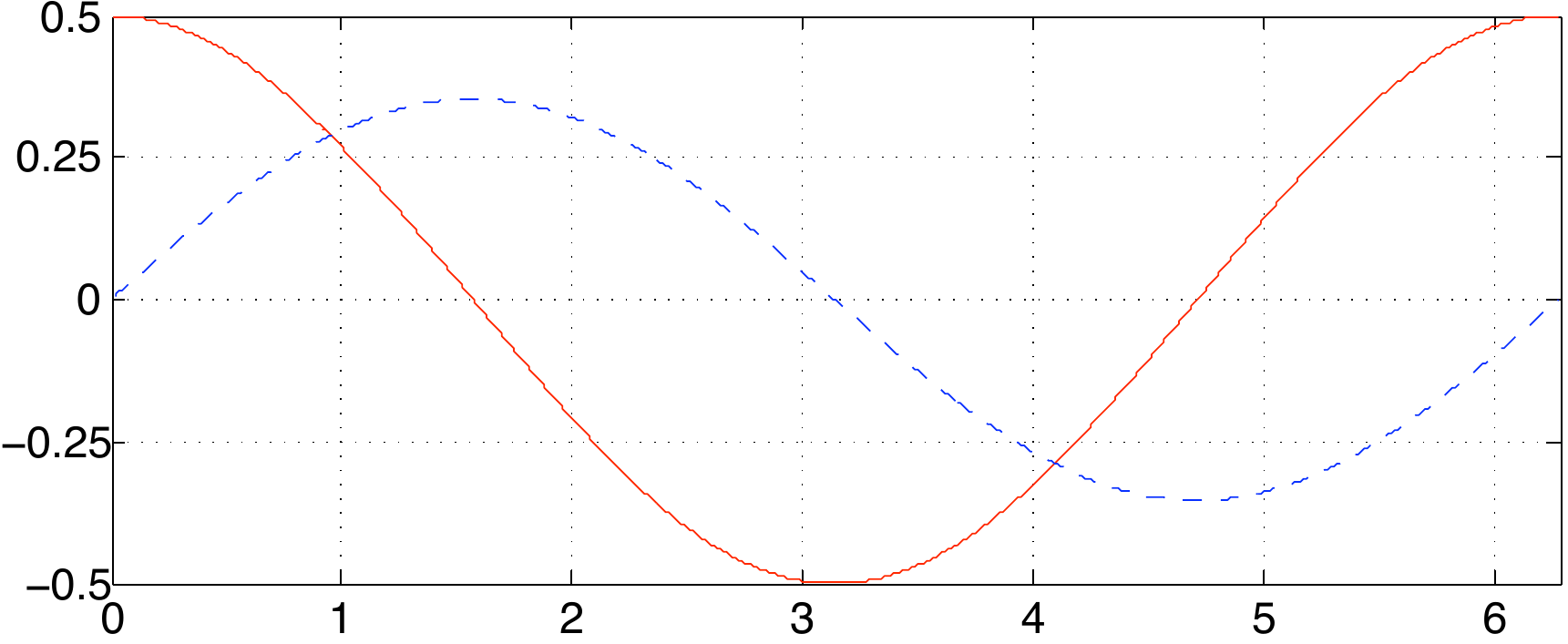}}\\
     \hline
     \subfigure
     [Coordinates $F_1(t)$ (solid line), $F_3(t)$ (dashed line) and $F_4(t)$ (dashed-dot line)  of the generalized internal force over a stroke. The coordinate $F_2(t)$ is equal to 0.]
     {\includegraphics[width=.6\textwidth]{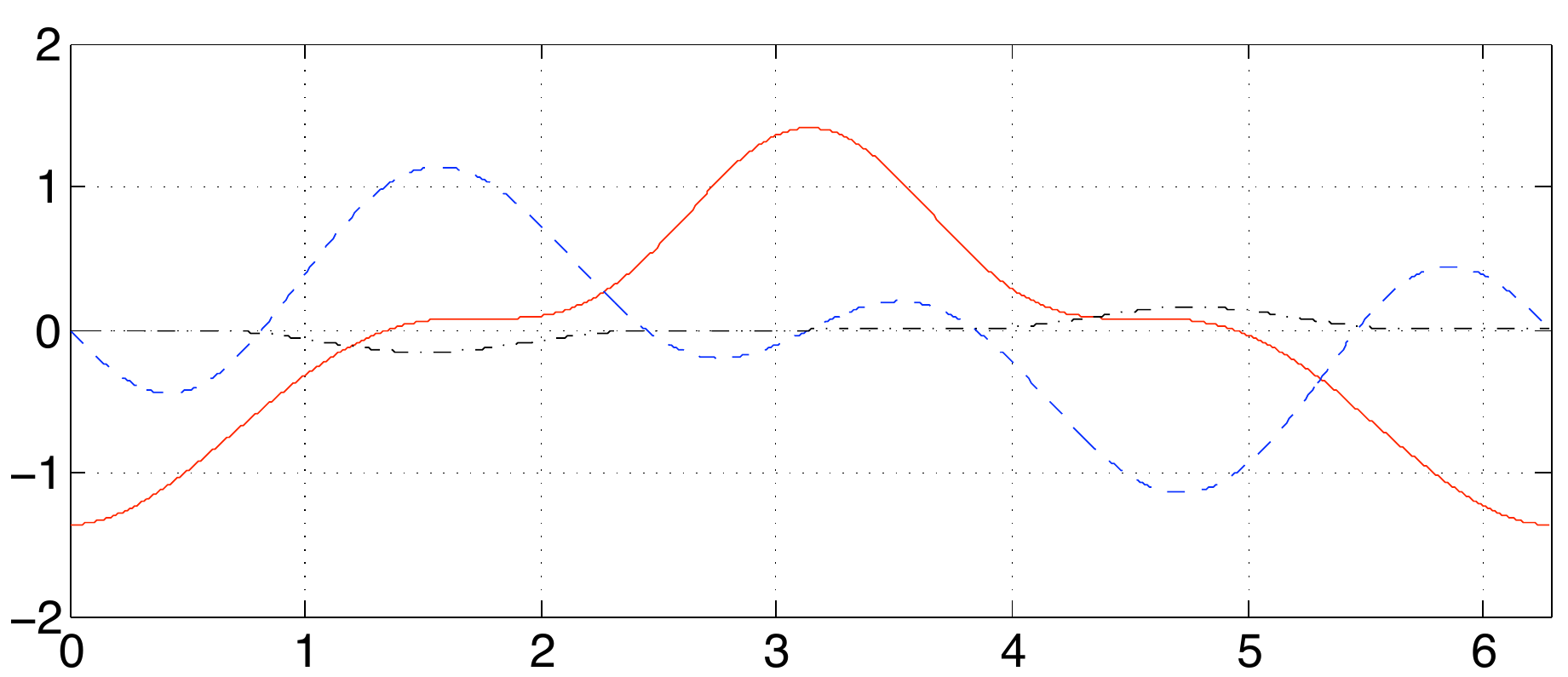}}\\
      \hline
\end{tabular}
\caption{\label{fig:6} Shape changes and internal forces.}
\end{figure}
\subsubsection{Example of circular motion}
We set now again $\mu=0.5$, $\alpha_1(t)=t$, $\alpha_j(t)=0$ ($j=2,3,4,5$) and $h_k(t)=0$ ($k=2,3$) for all $t\geq 0$ but we specify the control variable $h_1$ to be a non-zero constant. With these settings, we observe that the amoeba swims along a circular path.
\begin{figure}[H]
\centerline{\includegraphics[width=.6\textwidth]{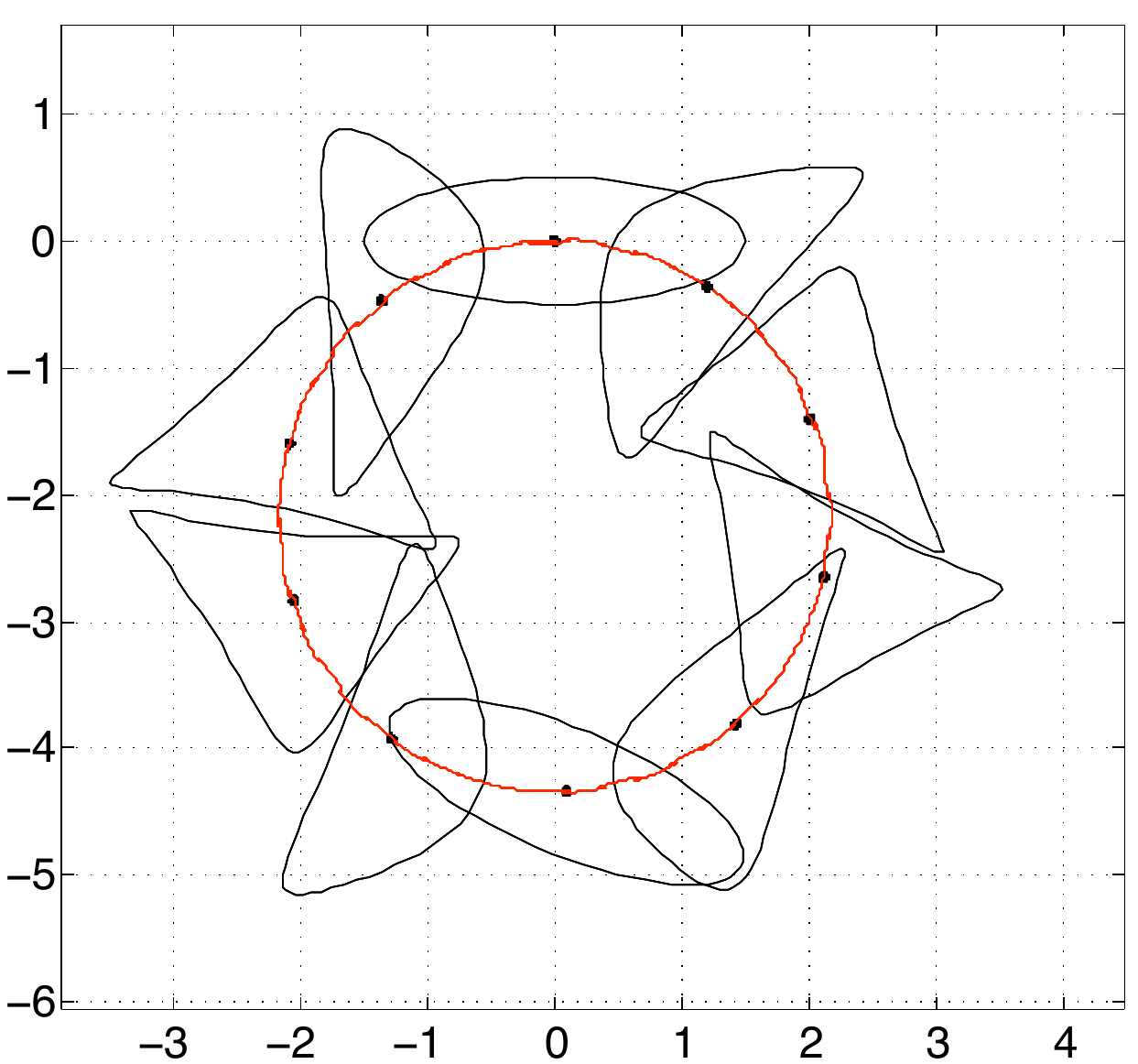}}
\caption{Successive positions and shapes of the amoeba in its course when $h_1=1$. The animal follows a circular trajectory completed over a time interval of length approximately $24\pi$.}
\end{figure}
Again, with these data, only $c_1$ and $c_2$ are non-zero functions. The radii of the circles change along with the values of the constant $h_1$ as illustrated in Figure~\ref{fig:7}. 
\begin{figure}[H]
\centerline{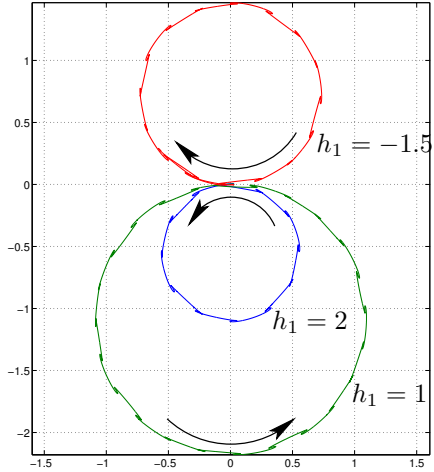}
\caption{\label{fig:7}Trajectories of the center of mass of the amoeba. For $h_1=-1.5$ in red over the time interval $[0,18\pi]$, for $h_1=1$ in green over $[0,24\pi]$ and for $h_1=2$ in blue over $[0,12\pi]$.}
\end{figure}
When $h_1=1$, the graphs of the controls are given in Figure~\ref{controls_graph} and the graphs of the internal forces in Figure~\ref{internal_forces_graph}.
\begin{figure}[H]     
     \centering
     \begin{tabular}{|c|}
     \hline
     \subfigure
     {\includegraphics[width=.7\textwidth]{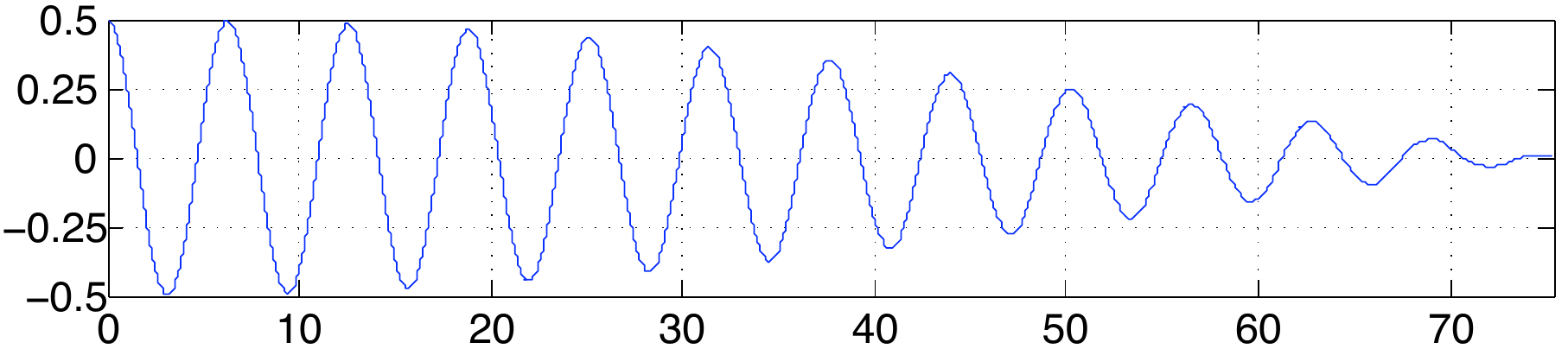}}\\
     \hline
     \subfigure
     {\includegraphics[width=.7\textwidth]{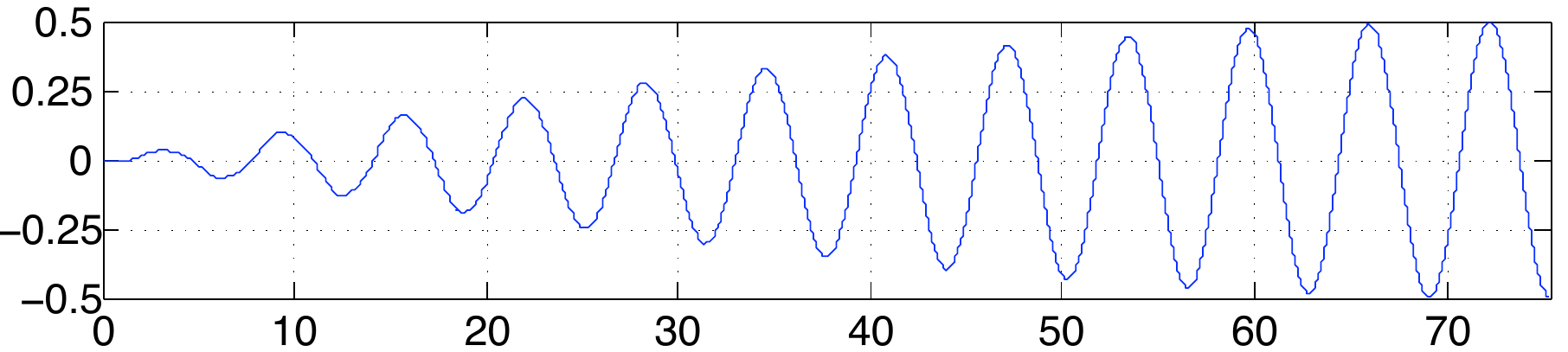}}\\
      \hline
      \subfigure
     {\includegraphics[width=.7\textwidth]{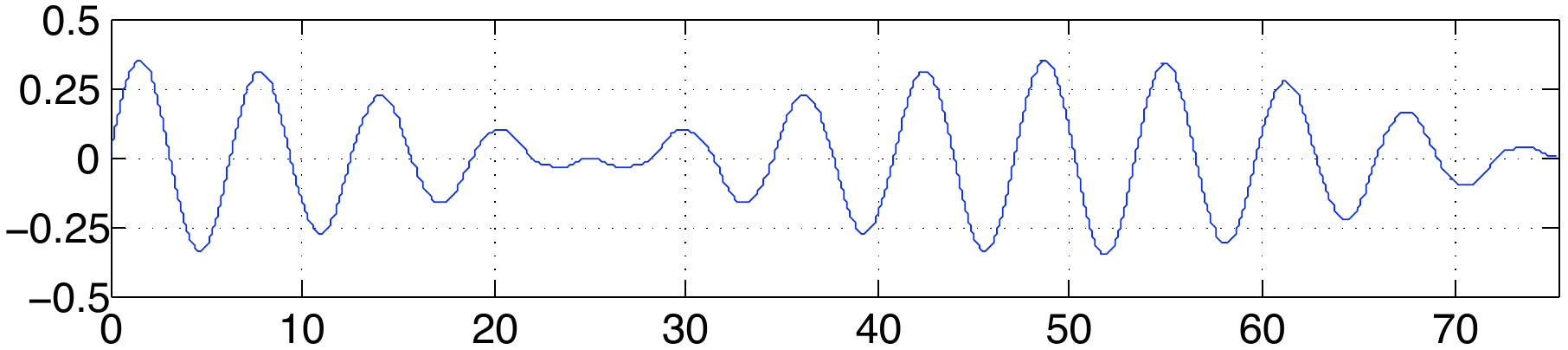}}\\
      \hline
      \subfigure
     {\includegraphics[width=.7\textwidth]{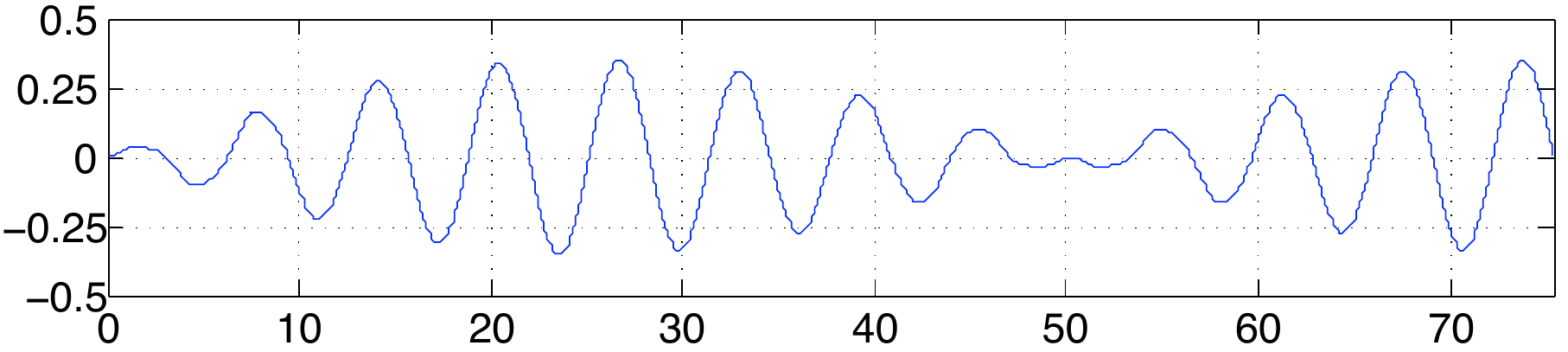}}\\
      \hline
      \end{tabular}
     \caption{\label{controls_graph}Values of the controls $a_1(t)$, $b_1(t)$, $a_2(t)$ and $b_2(t)$ over the time interval $[0,24\pi]$ for the amoeba following the circular trajectory with $h_1=1$.}
\end{figure}
\begin{figure}[H]     
     \centering
     \begin{tabular}{|c|}
     \hline
     \subfigure
     {\includegraphics[width=.7\textwidth]{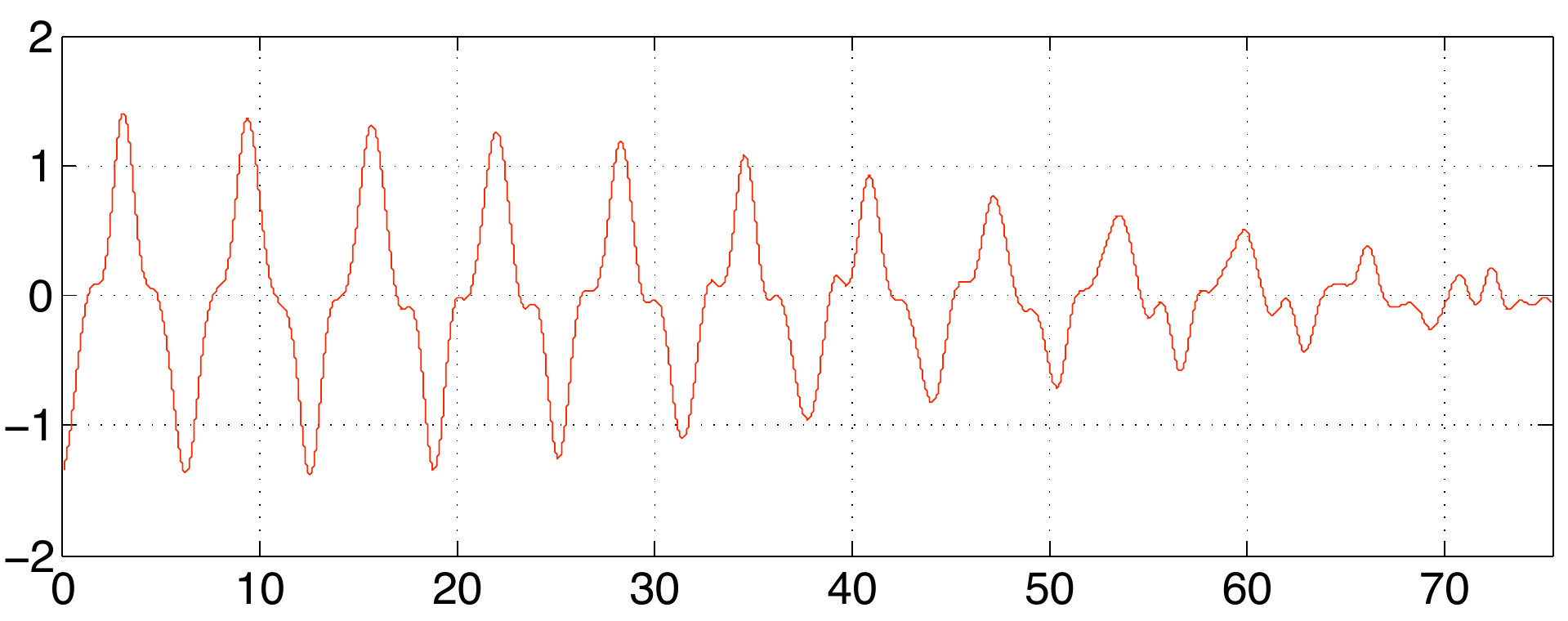}}\\
     \hline
     \subfigure
     {\includegraphics[width=.7\textwidth]{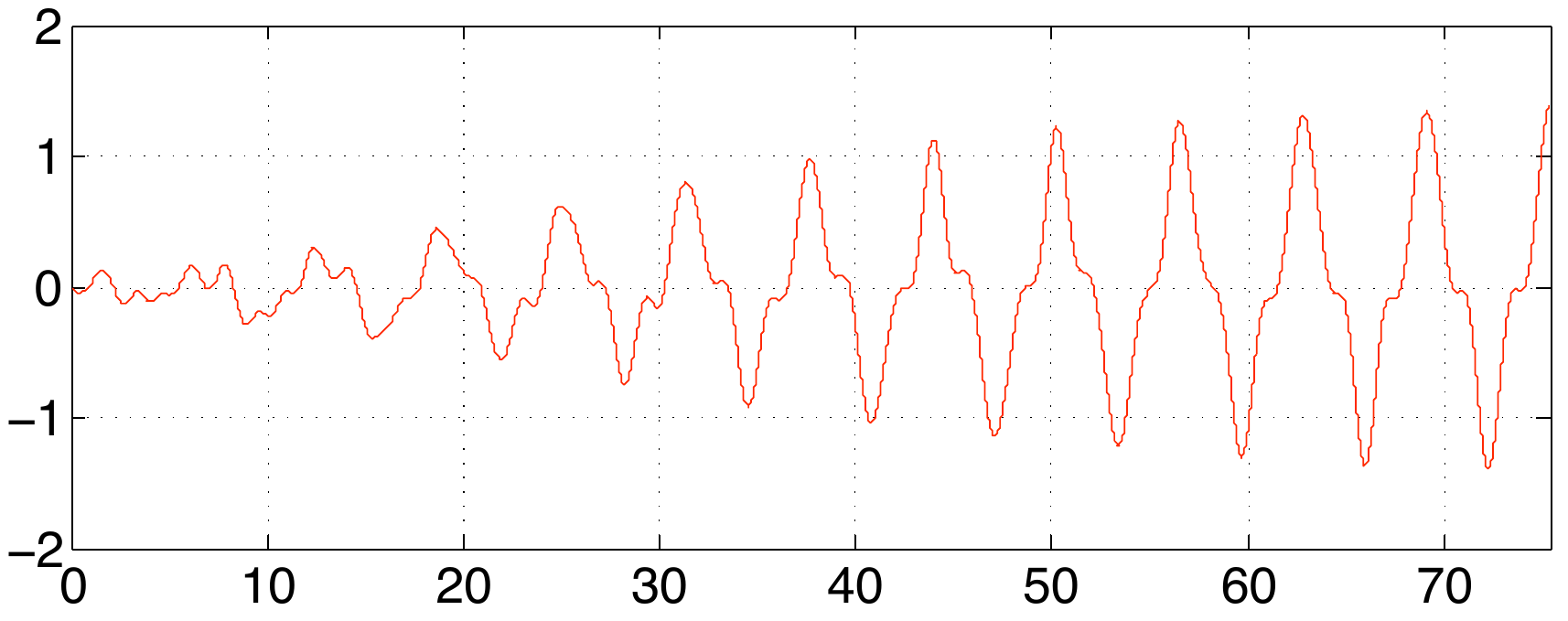}}\\
      \hline
      \subfigure
     {\includegraphics[width=.7\textwidth]{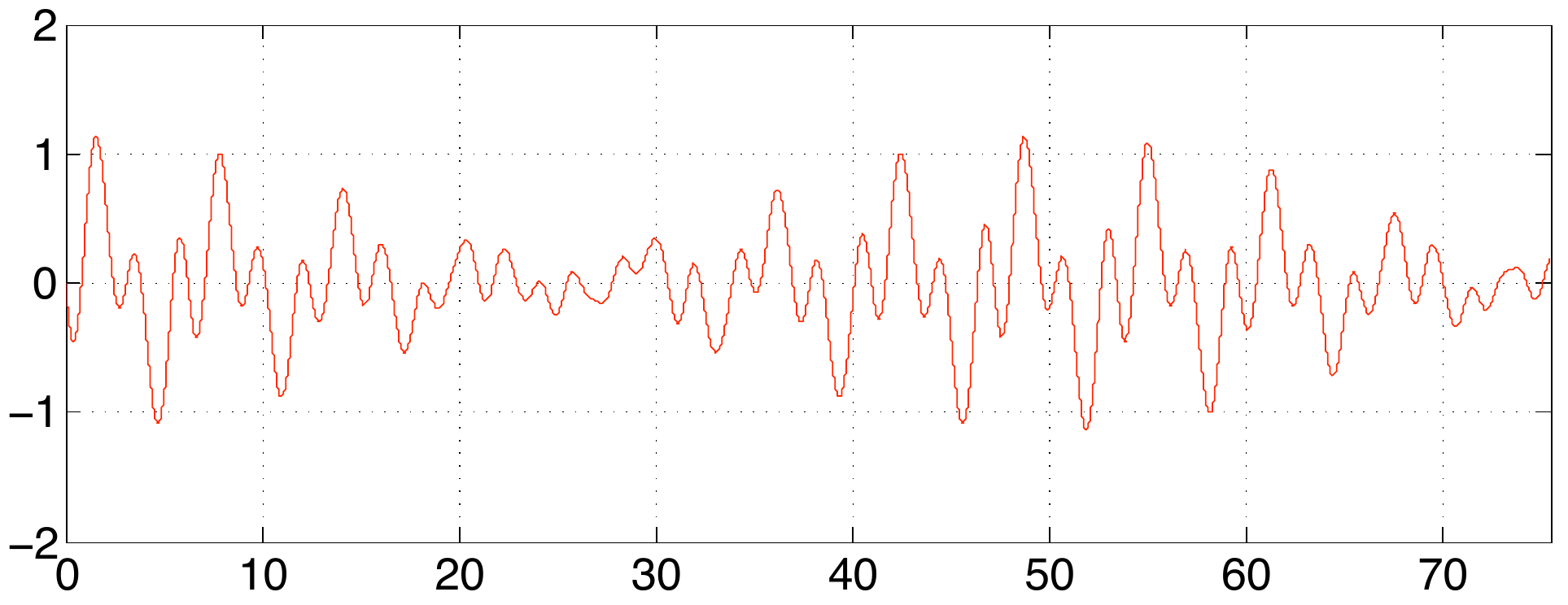}}\\
      \hline
      \subfigure
     {\includegraphics[width=.7\textwidth]{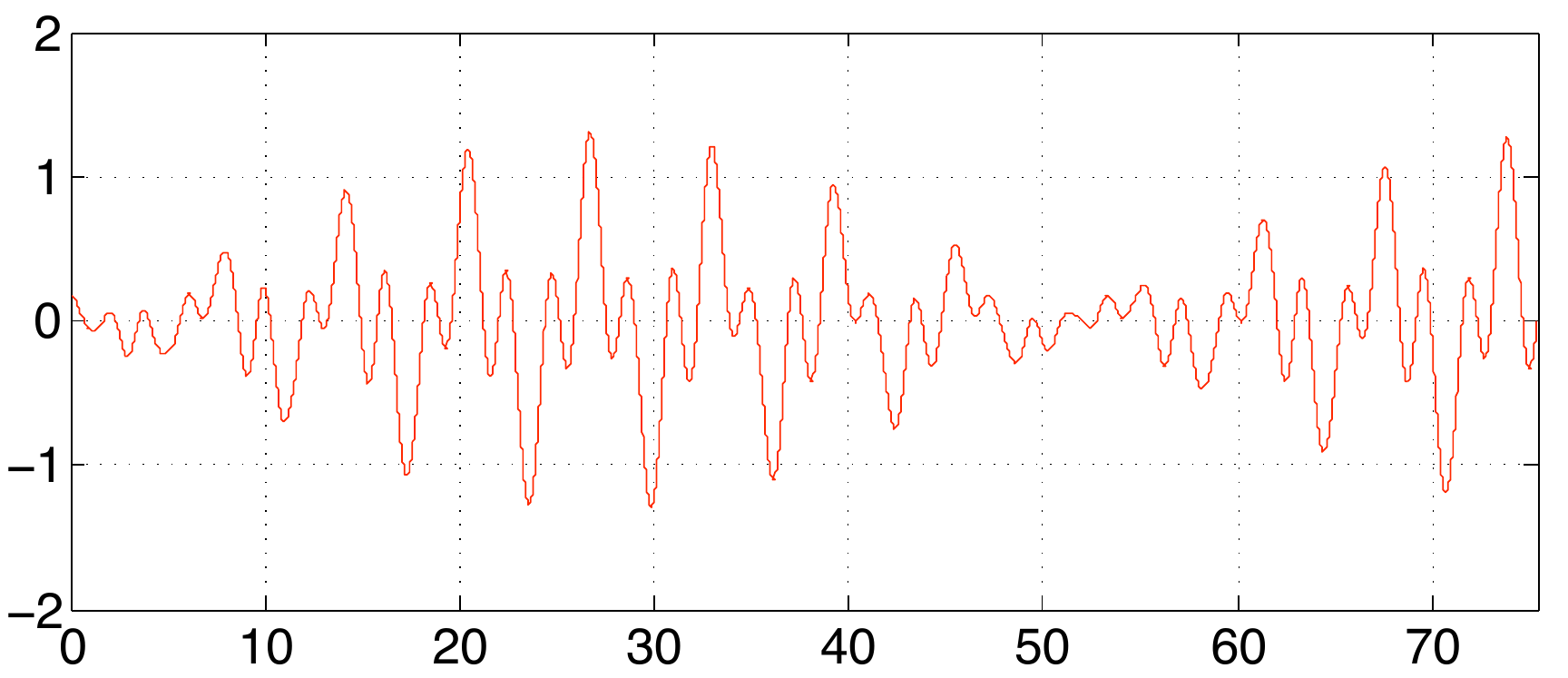}}\\
      \hline
      \end{tabular}
     \caption{\label{internal_forces_graph} Components $F_1(t)$, $F_2(t)$, $F_3(t)$ and $F_4(t)$ of the generalized force over the time interval $[0,24\pi]$ for the amoeba following the circular trajectory with $h_1=1$.}
\end{figure}
\subsubsection{Motion planning}
Based on the two preceding examples, we observe that the amoeba can follow any {\it smooth} trajectory with only the two first control variables $c_1$ and $c_2$ being non-zero. Indeed, the function $\alpha_1$ governs the frequencies of the strokes (and hence the velocity of the animal) and $h_1$ can be seen as the steering function; it allows the amoeba to turn left or right. On the web page, \url{http://www.iecn.u-nancy.fr/~munnier/page_amoeba/control_index.html}, such examples of motion planning are given. 
\subsubsection{Further examples}
The preceding example in fact illustrates a more general feature of locomotion: motion planning is actually possible with any pair of consecutive controls variables $(c_k,c_{k+1})$ $(k\geq 1$). For example, if we are willing to use only the pair $(c_3,c_4)$ as controls, then we can set $\alpha_1(t)=\alpha_2(t)=\pi/2$, $\alpha_4(t)=\alpha_5(t)=0$, $h_1(t)=h_3(t)=0$ for all $t\geq 0$ and 
the frequencies of the strokes are driven by $\alpha_3$ while $h_2$ turns out to be the steering variable. Such a strategy is illustrated in Figure~\ref{fig:8}.
\begin{figure}[H]
\centerline{\includegraphics[width=.6\textwidth]{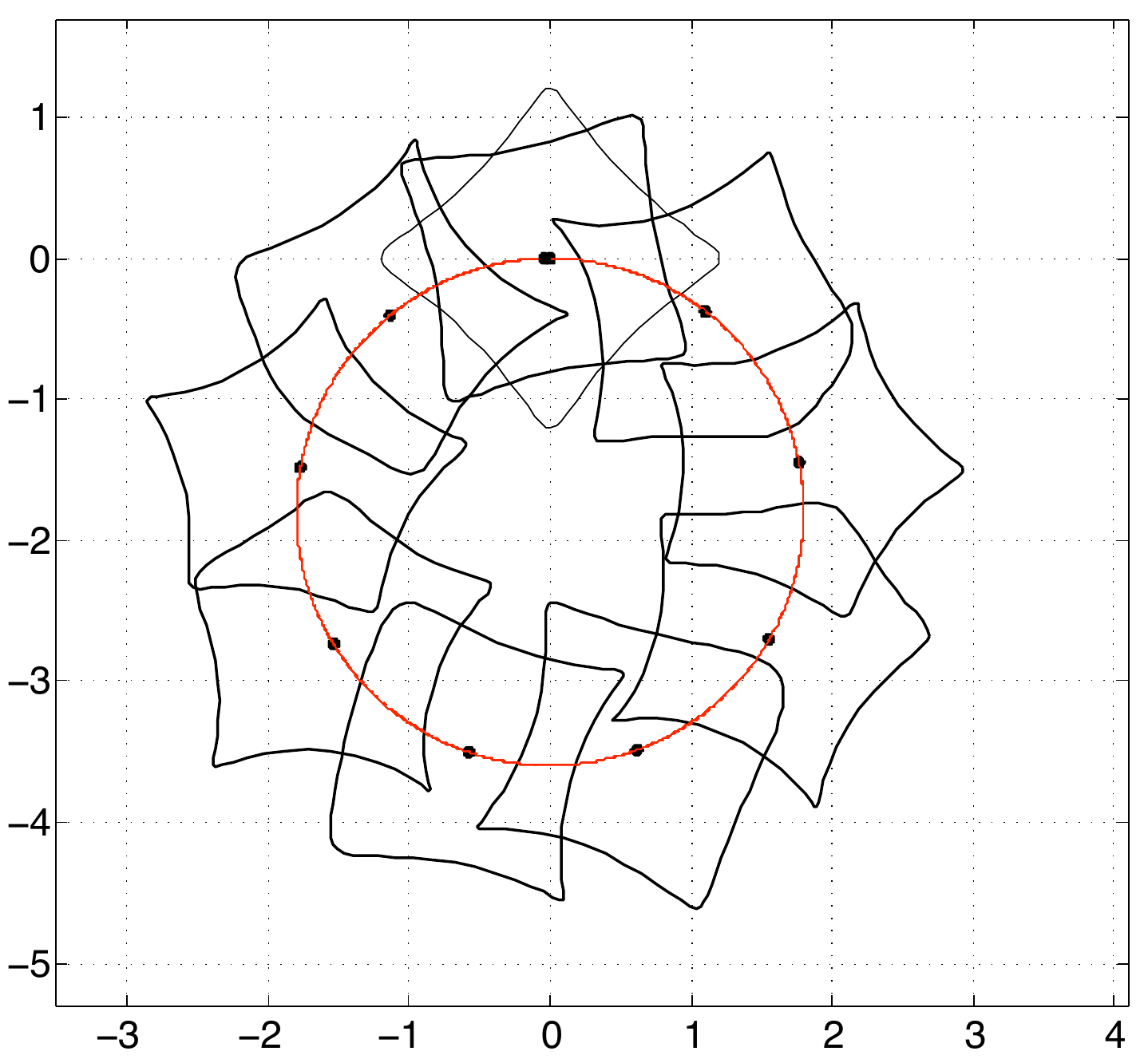}}
\caption{\label{fig:8}Successive positions and shapes of the amoeba for  $\alpha_1(t)=\alpha_2(t)=\pi/2$, $\alpha_4(t)=\alpha_5(t)=0$, $h_1(t)=h_3(t)=0$ for all $t\geq 0$ and $\alpha_3(t)=t$, $h_2(t)=1.2$.}
\end{figure}
At last, we can use the pair $(c_5,c_6)$. We set then $\alpha_k(t)=\pi/2$ ($k=1,2,3,4$), $h_1(t)=h_2(t)=0$ for all $t\geq 0$ and the variable $\alpha_5$ controls the frequency of the strokes while $h_3$ becomes the steering variable. These settings are those used in Figure~\ref{fig:9}.
\begin{figure}[H]
\centerline{\includegraphics[width=.9\textwidth]{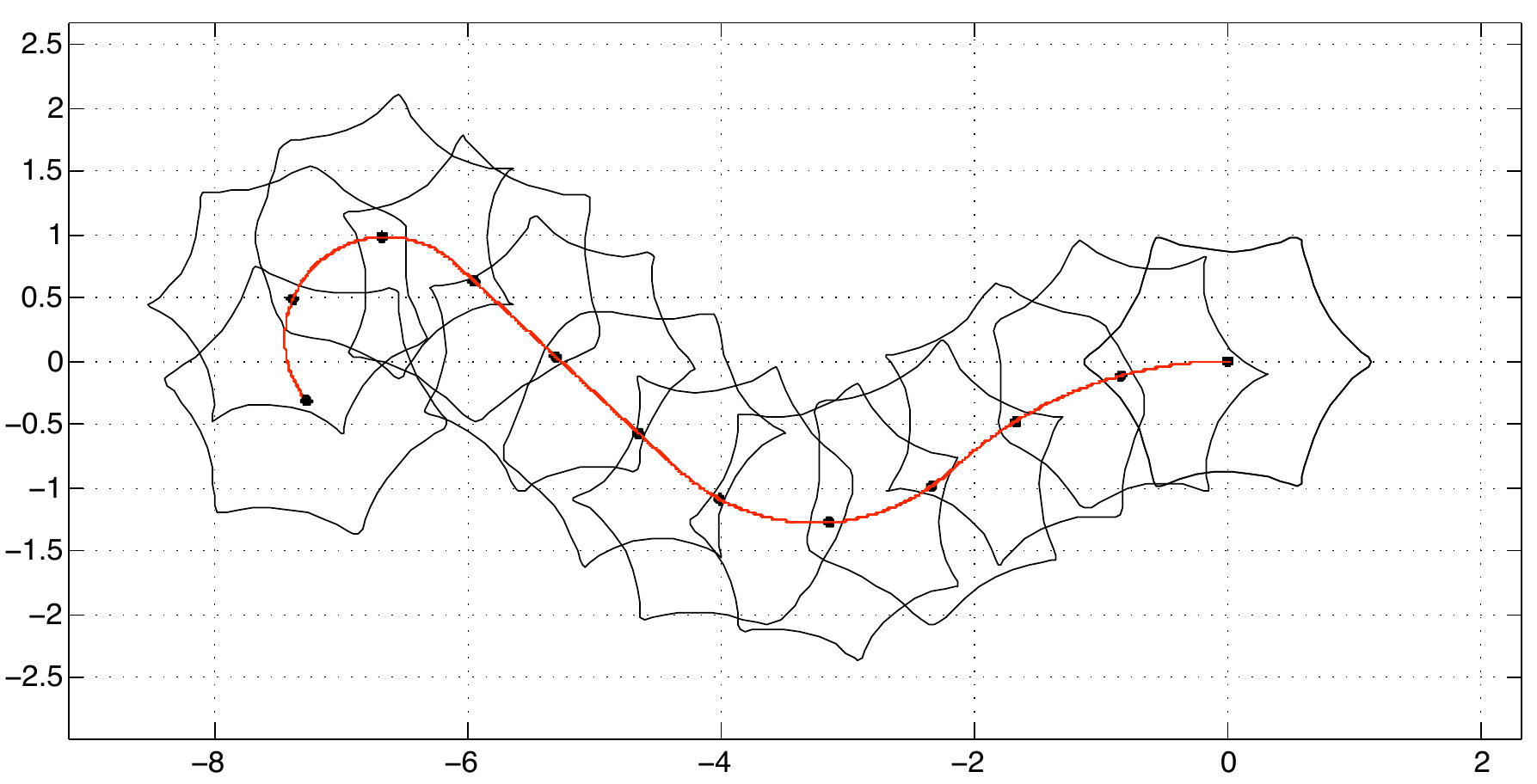}}
\caption{\label{fig:9}Successive positions and shapes of the amoeba for  $\alpha_k(t)=\pi/2$ ($k=1,2,3,4$), $h_1(t)=h_2(t)=0$, $\alpha_5(t)=t$ and $h_3(t)=1$ for $t\in[0,100\pi]$, $h_3(t)=-2$ for $t\in]100\pi,200\pi]$, $h_3(t)=0$ for $t\in]200\pi,300\pi]$ and $h_3(t)=4$ for $t\in]300\pi,400\pi]$}
\end{figure}

Once more, we refer to the web page \url{http://www.iecn.u-nancy.fr/~munnier/page_amoeba/control_index.html} for the animations and further examples.
\subsection{Moonwalking}
\label{moonwalking}
Our main result of controllability, Theorem~\ref{THE_diminf_tracking}, states that the amoeba is not only able to follow approximately any given trajectory but also that 
this task can be achieved while undergoing (also approximately) any prescribed shape-changes.  Let us illustrate this surprising result with the following example: In Figure~\ref{fig:22} are displayed screenshots of the amoeba swimming to the left in the first row and toward the opposite direction in the second row although the shape-changes seem to be similar in both cases. 
\begin{figure}[H]     
     \centering
     \begin{tabular}{|@{\hspace{-0.5mm}}c@{\hspace{0.2mm}}|@{\hspace{-0.5mm}}c@{\hspace{0.2mm}}|@{\hspace{-0.5mm}}c@{\hspace{0.2mm}}|@{\hspace{-0.5mm}}c@{\hspace{0.2mm}}|@{\hspace{-0.5mm}}c@{\hspace{0.2mm}}|@{\hspace{-0.5mm}}c@{\hspace{0.2mm}}|@{\hspace{-0.5mm}}c@{\hspace{0.2mm}}|@{\hspace{-0.5mm}}c@{\hspace{0.2mm}}|}
     \hline
     \subfigure
     {\includegraphics[width=.12\textwidth]{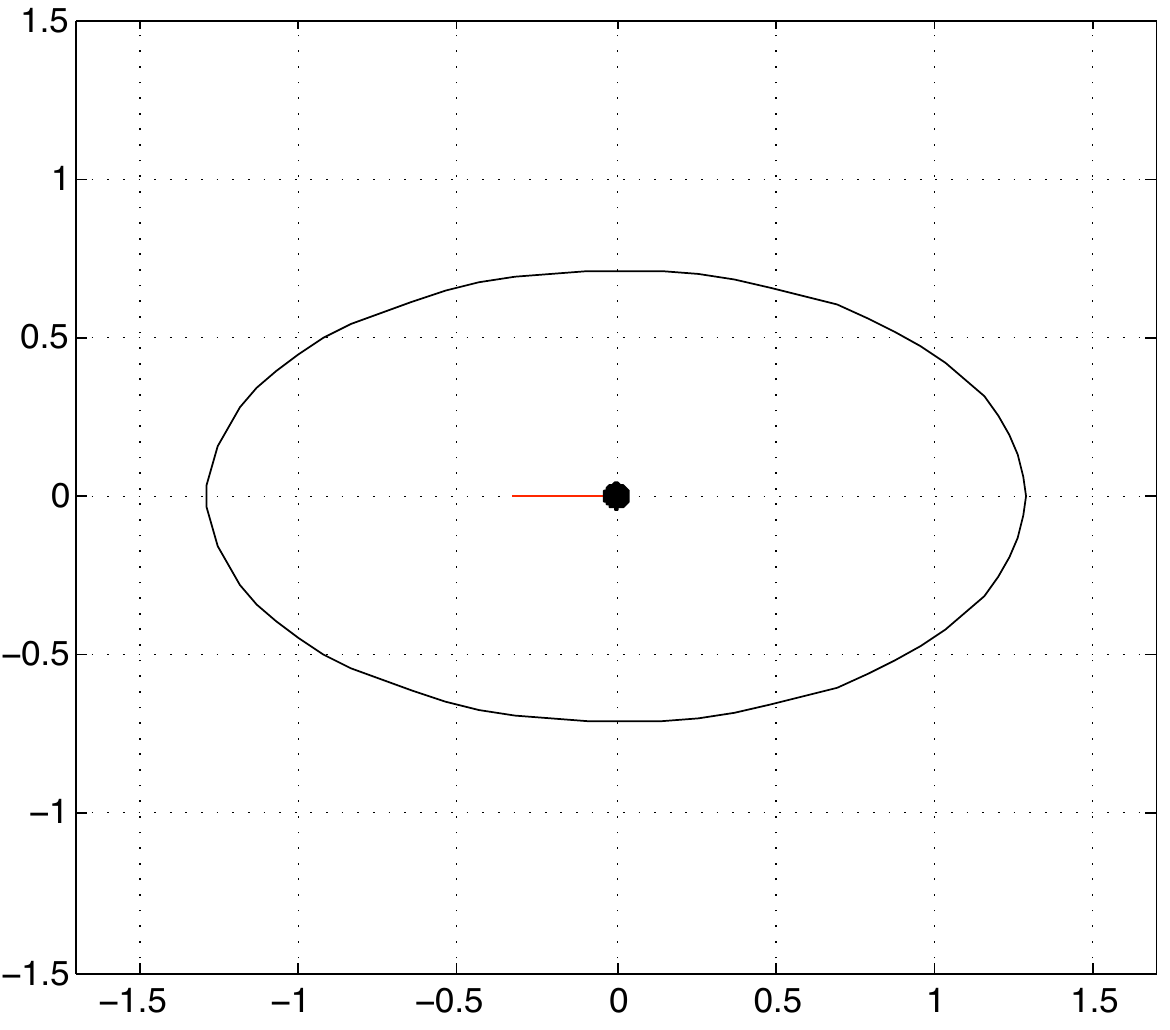}}
     &
     \subfigure
     {\includegraphics[width=.12\textwidth]{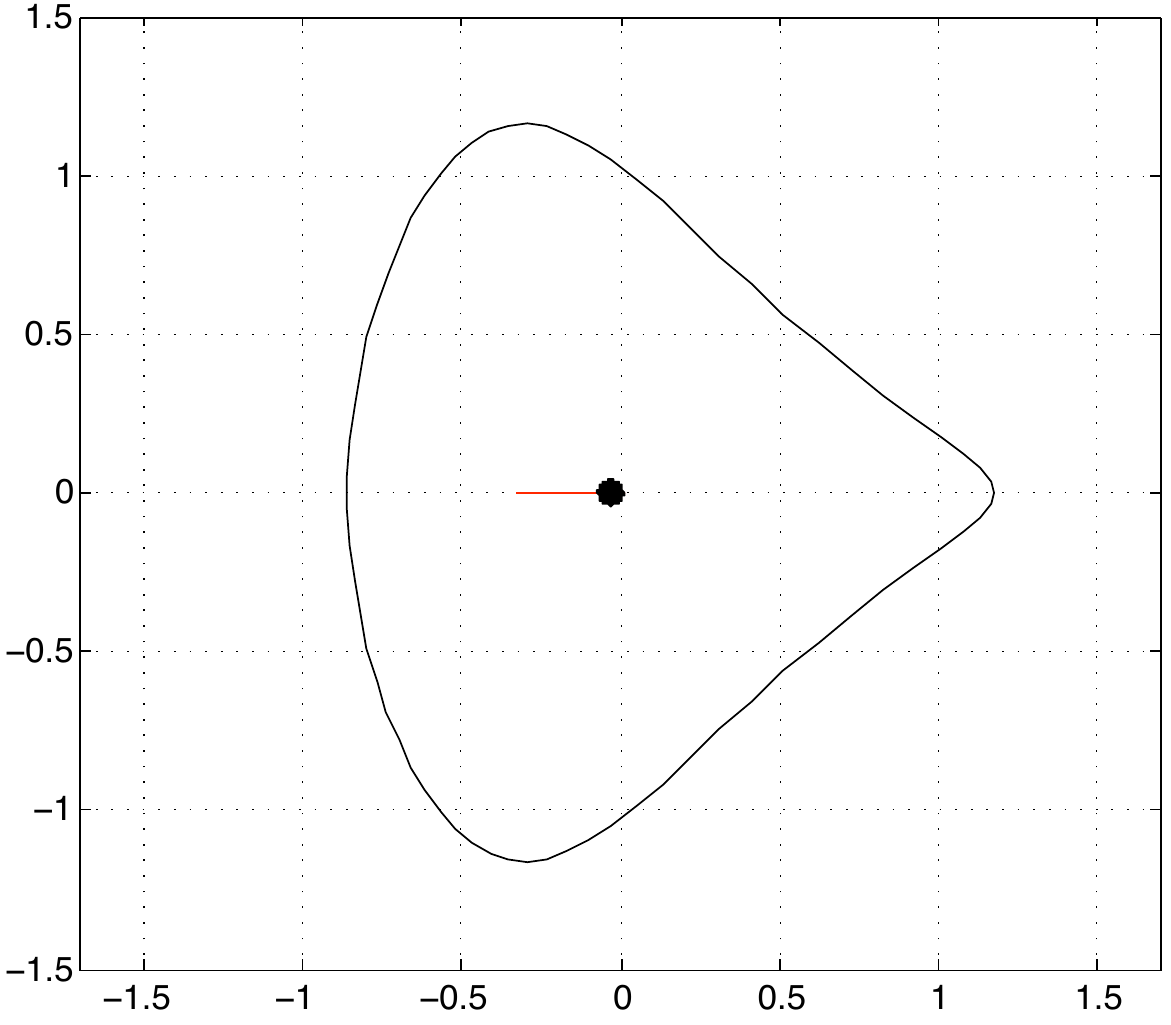}}
     &
     \subfigure
     {\includegraphics[width=.12\textwidth]{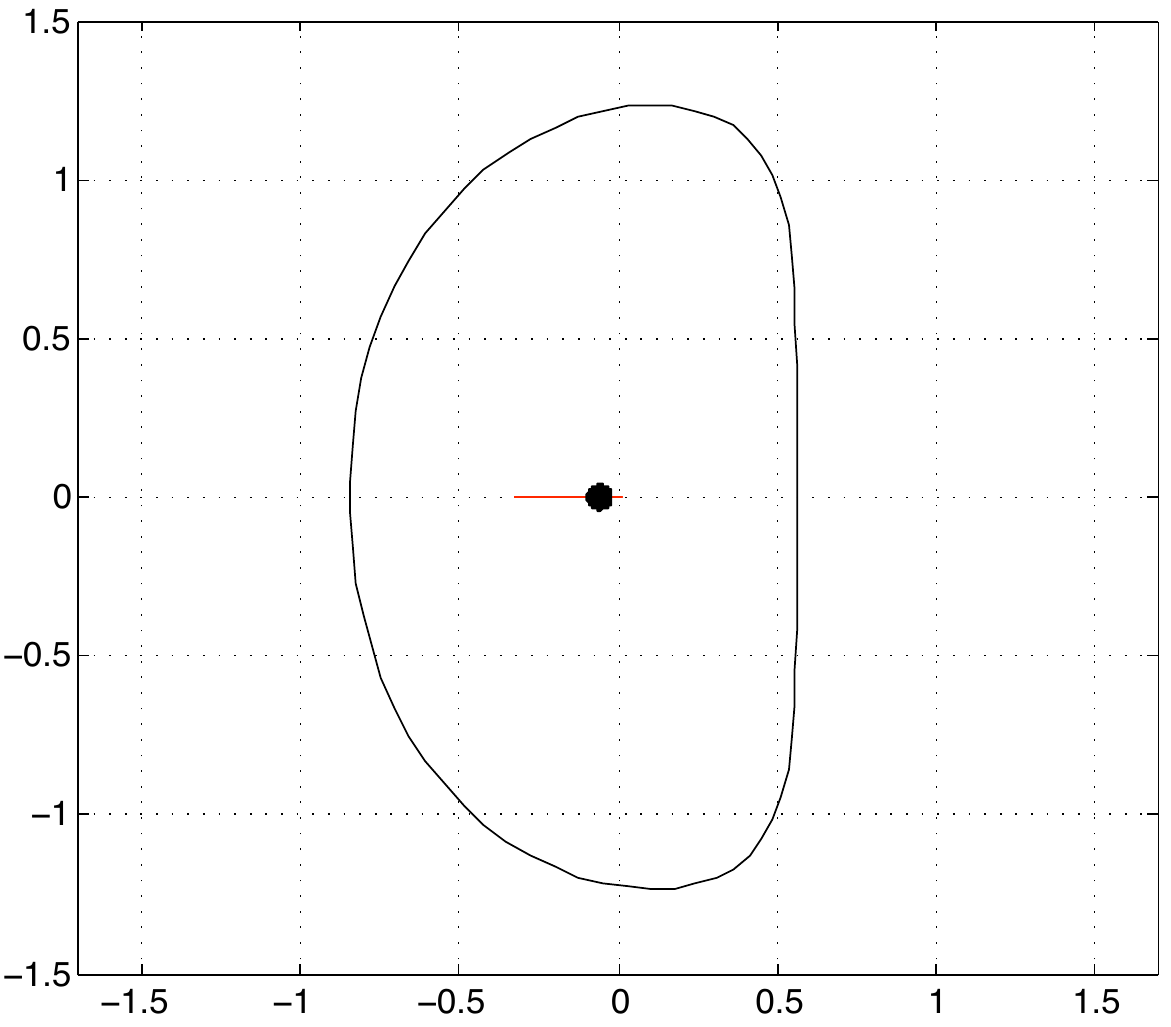}}
     &
     \subfigure
     {\includegraphics[width=.12\textwidth]{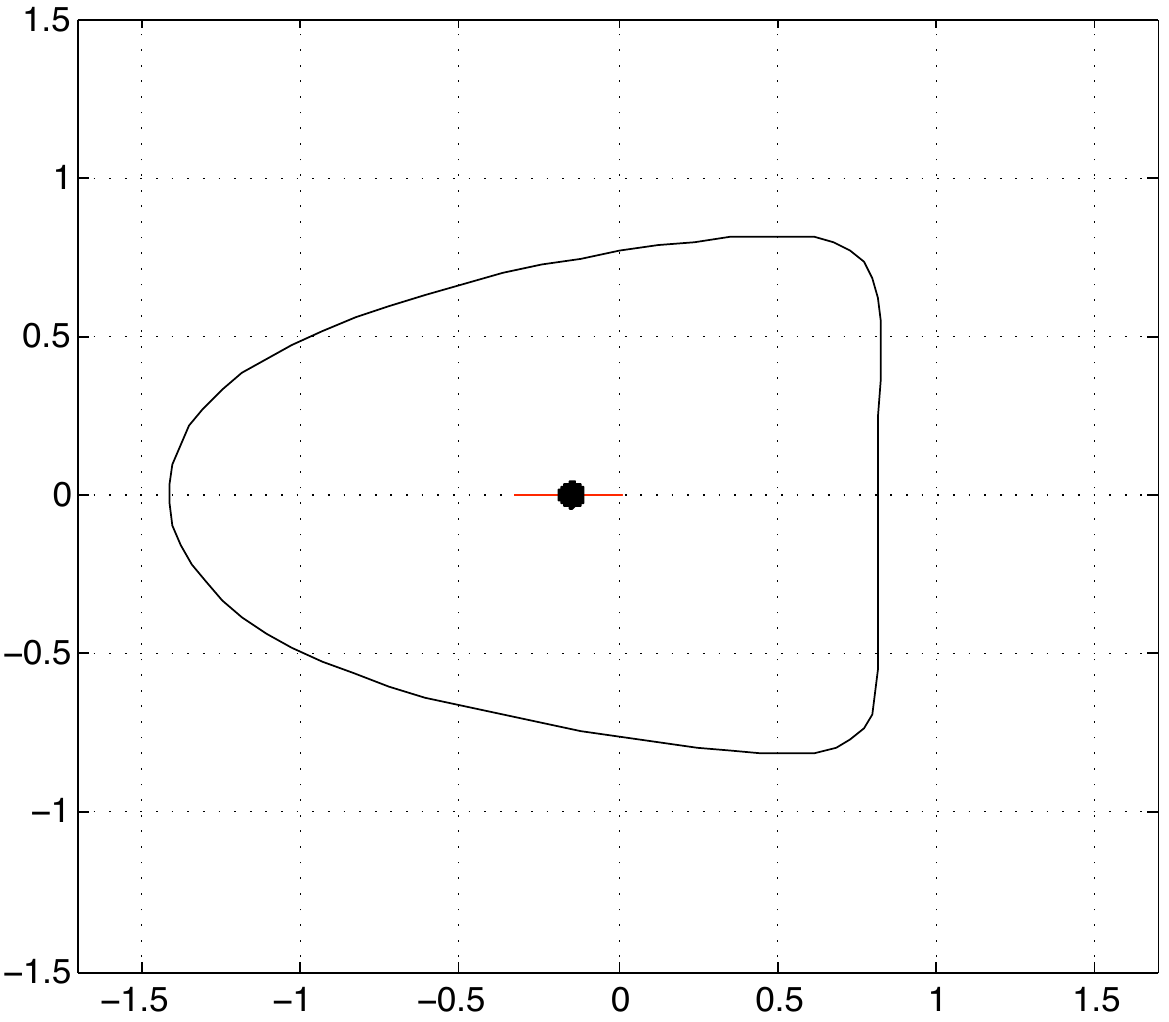}}
     &
     \subfigure
     {\includegraphics[width=.12\textwidth]{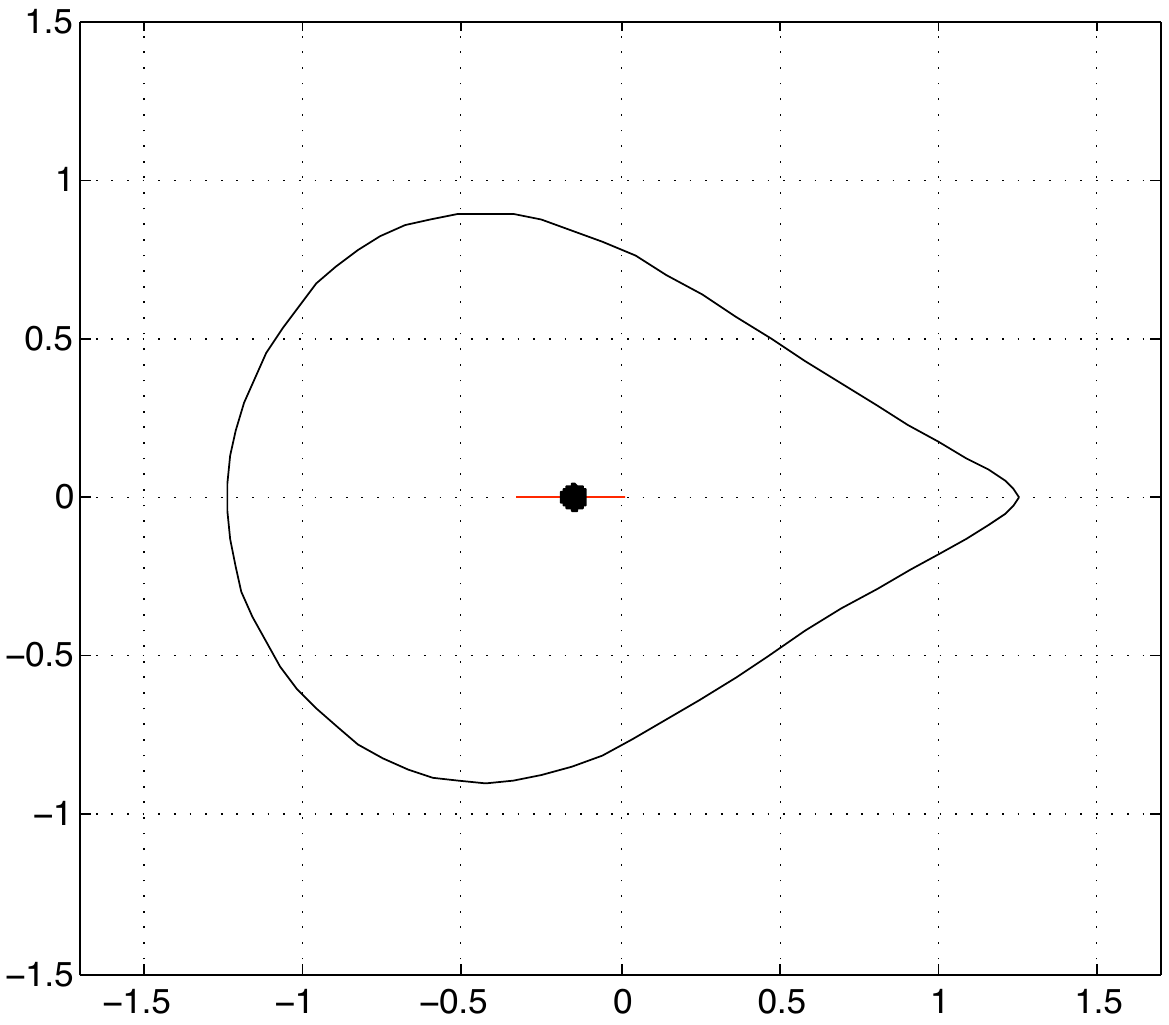}}
     &
     \subfigure
     {\includegraphics[width=.12\textwidth]{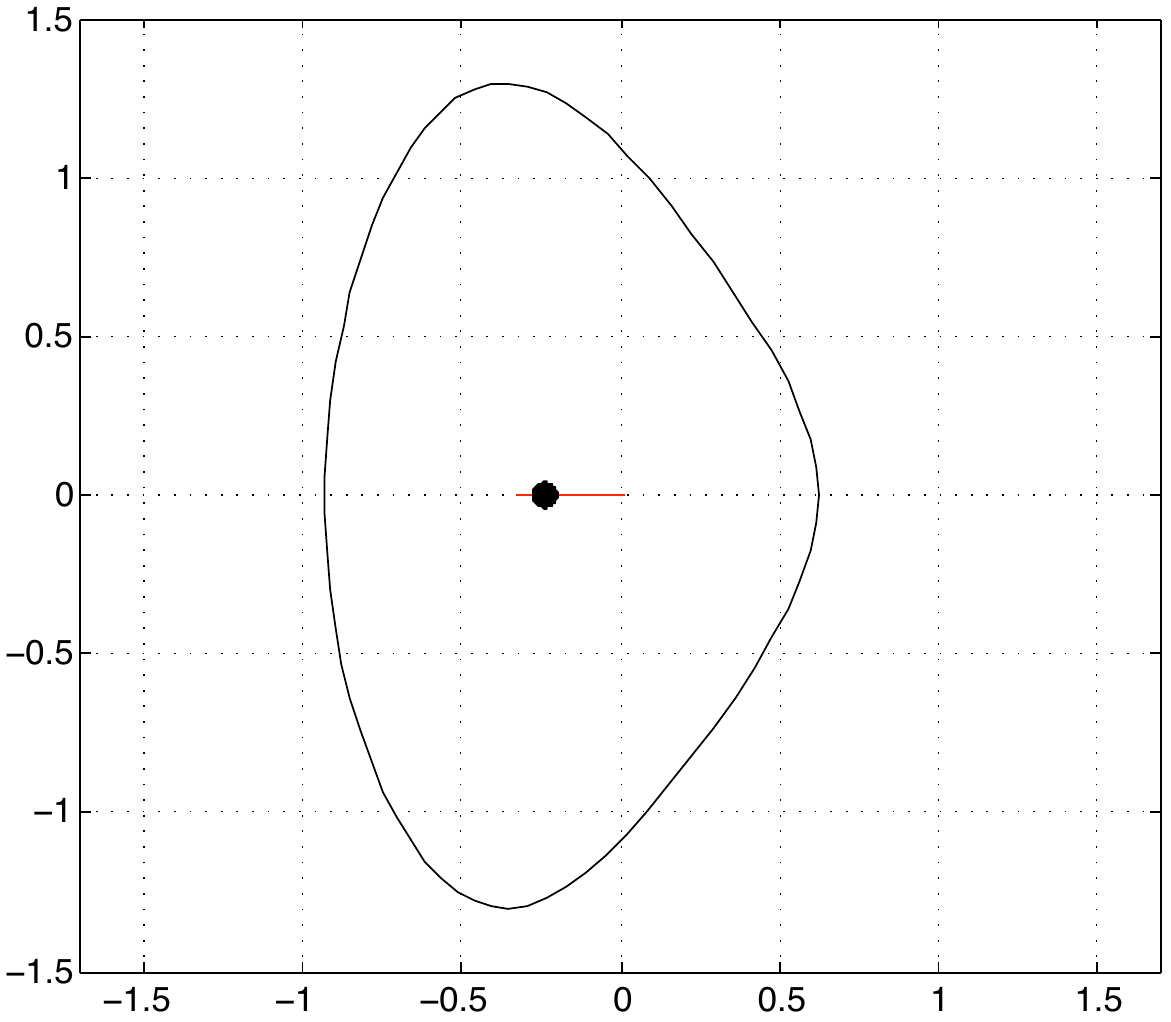}}
     &
     \subfigure
     {\includegraphics[width=.12\textwidth]{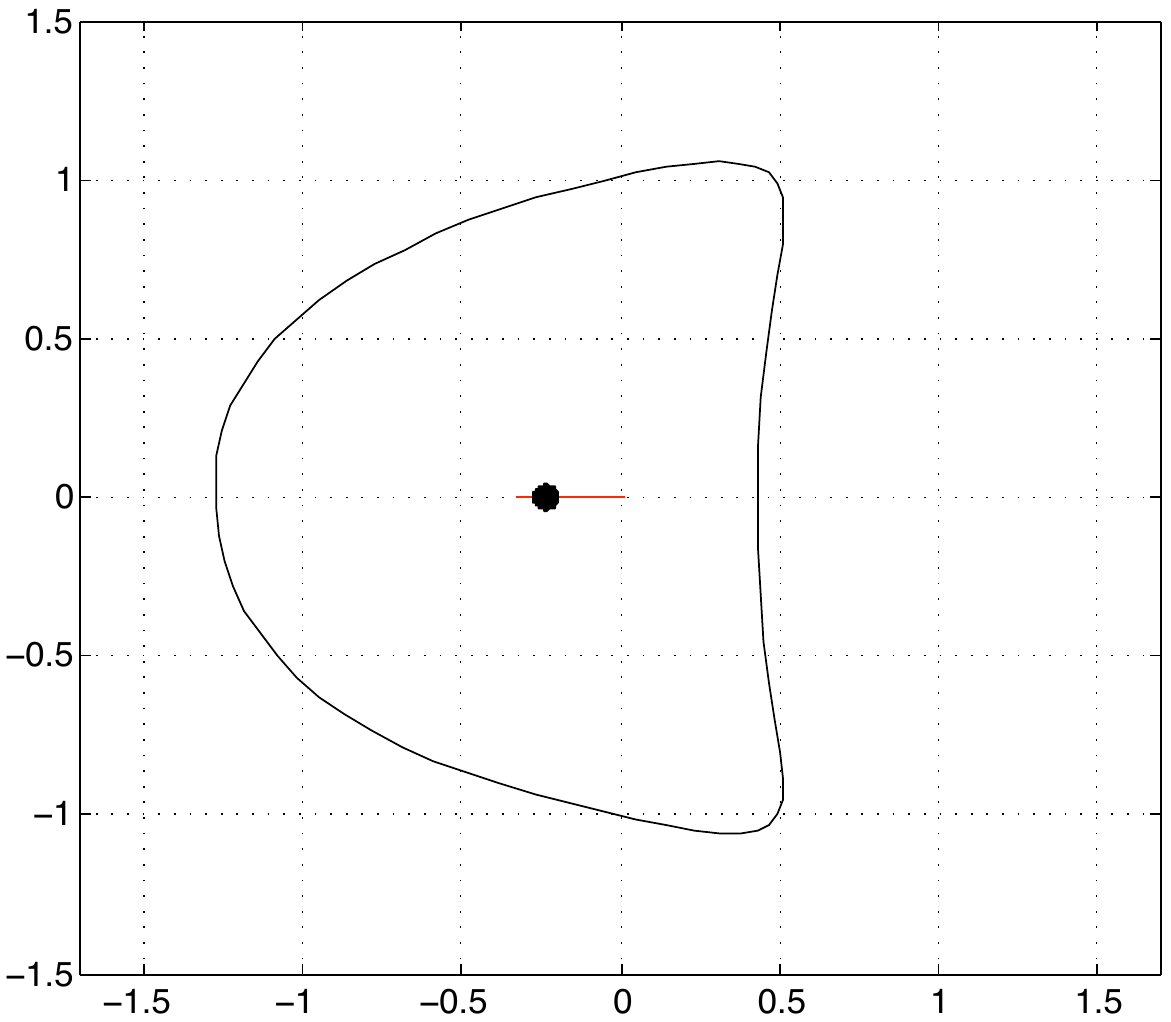}}
     &
     \subfigure
     {\includegraphics[width=.12\textwidth]{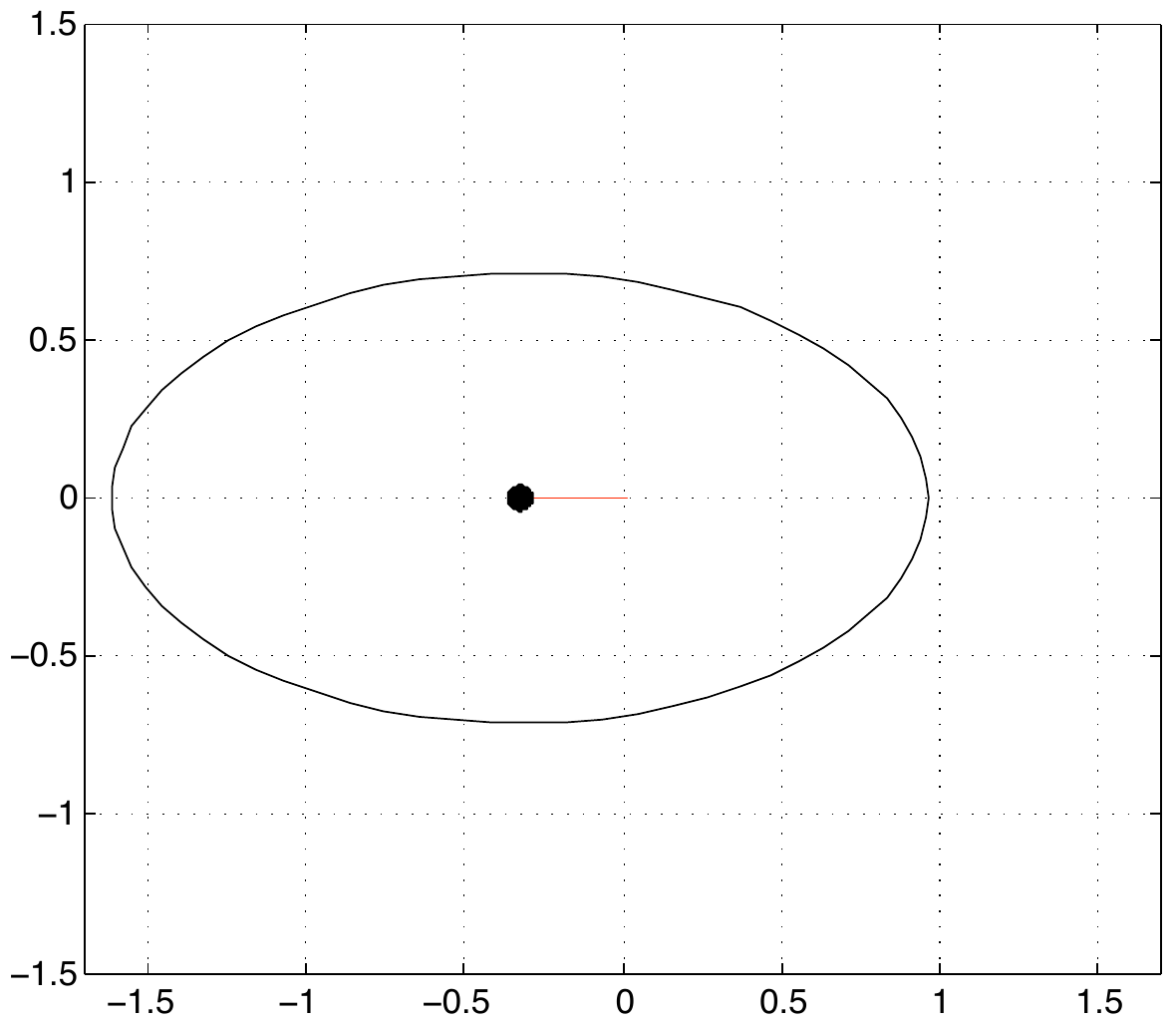}}\\
     \hline
     
     \subfigure
     {\includegraphics[width=.12\textwidth]{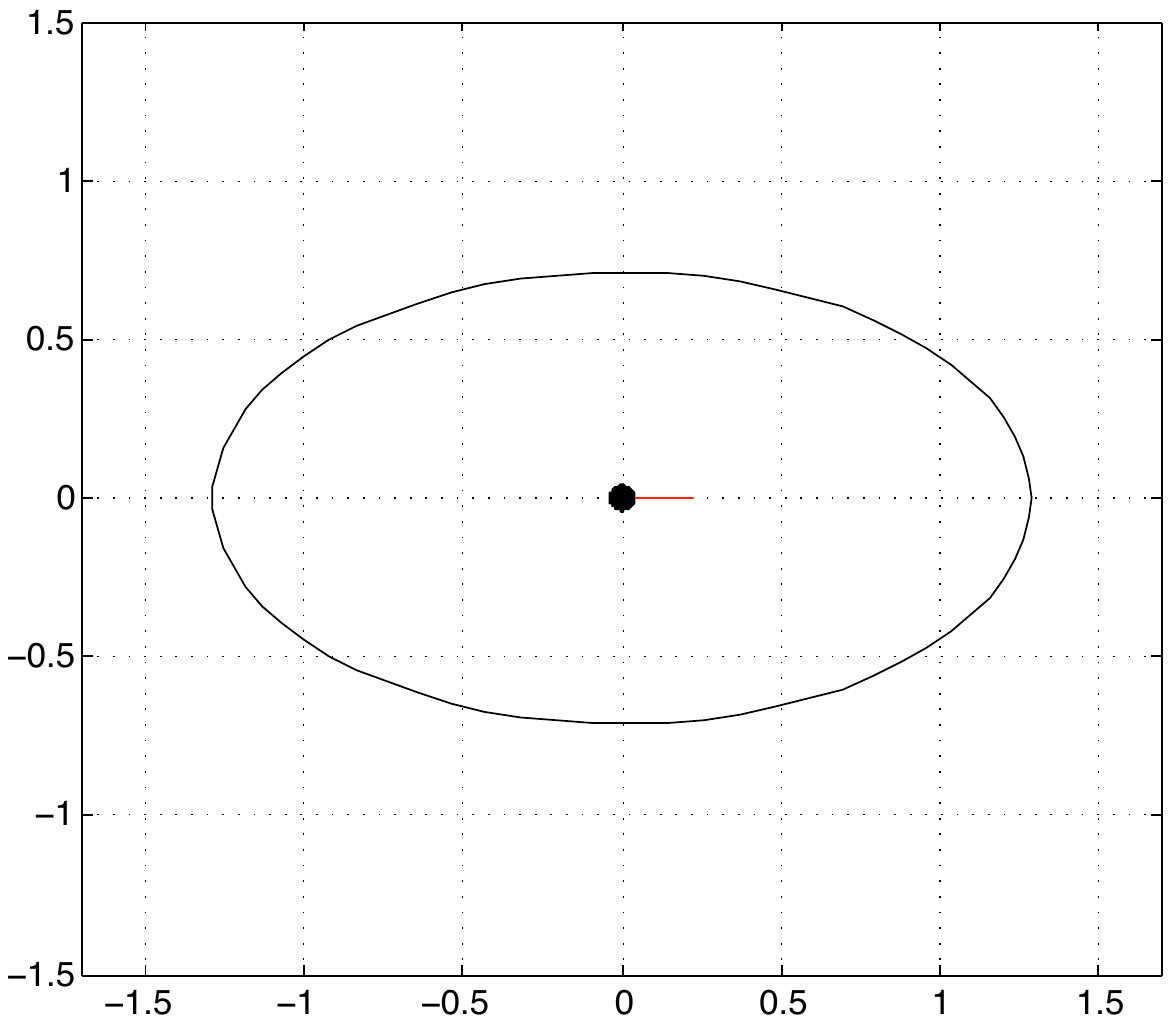}}
     &
     \subfigure
     {\includegraphics[width=.12\textwidth]{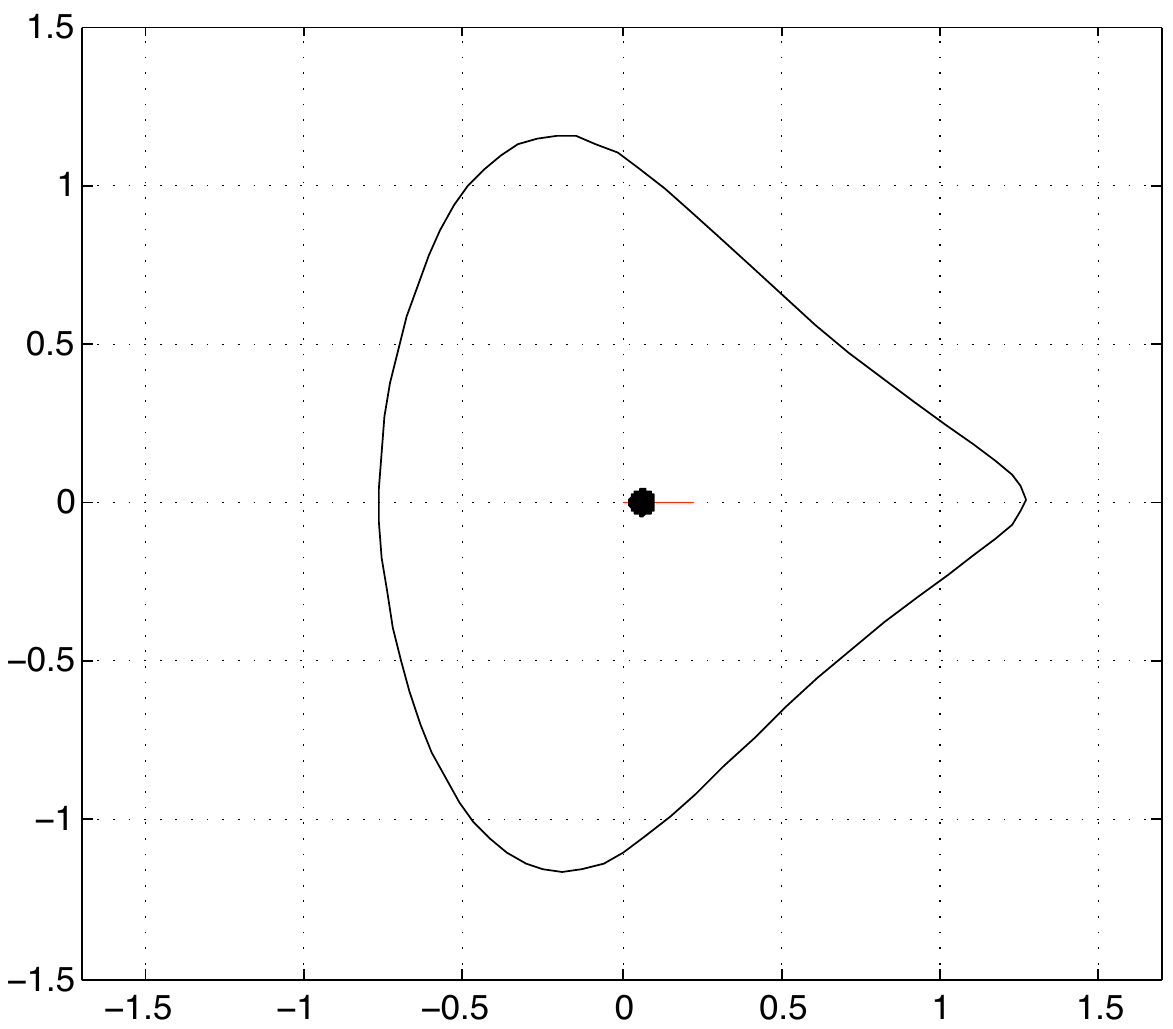}}
     &
     \subfigure
     {\includegraphics[width=.12\textwidth]{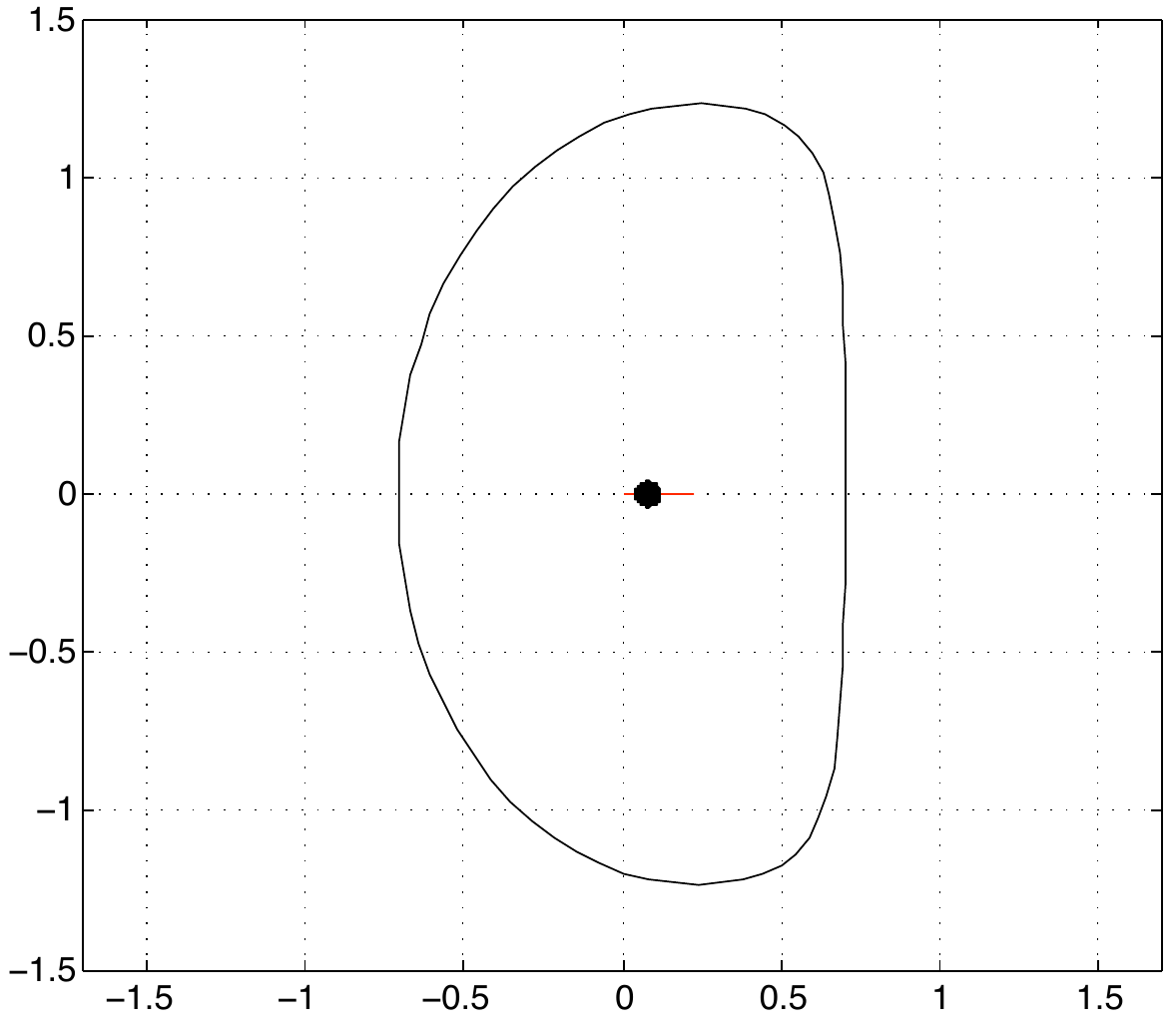}}
     &
     \subfigure
     {\includegraphics[width=.12\textwidth]{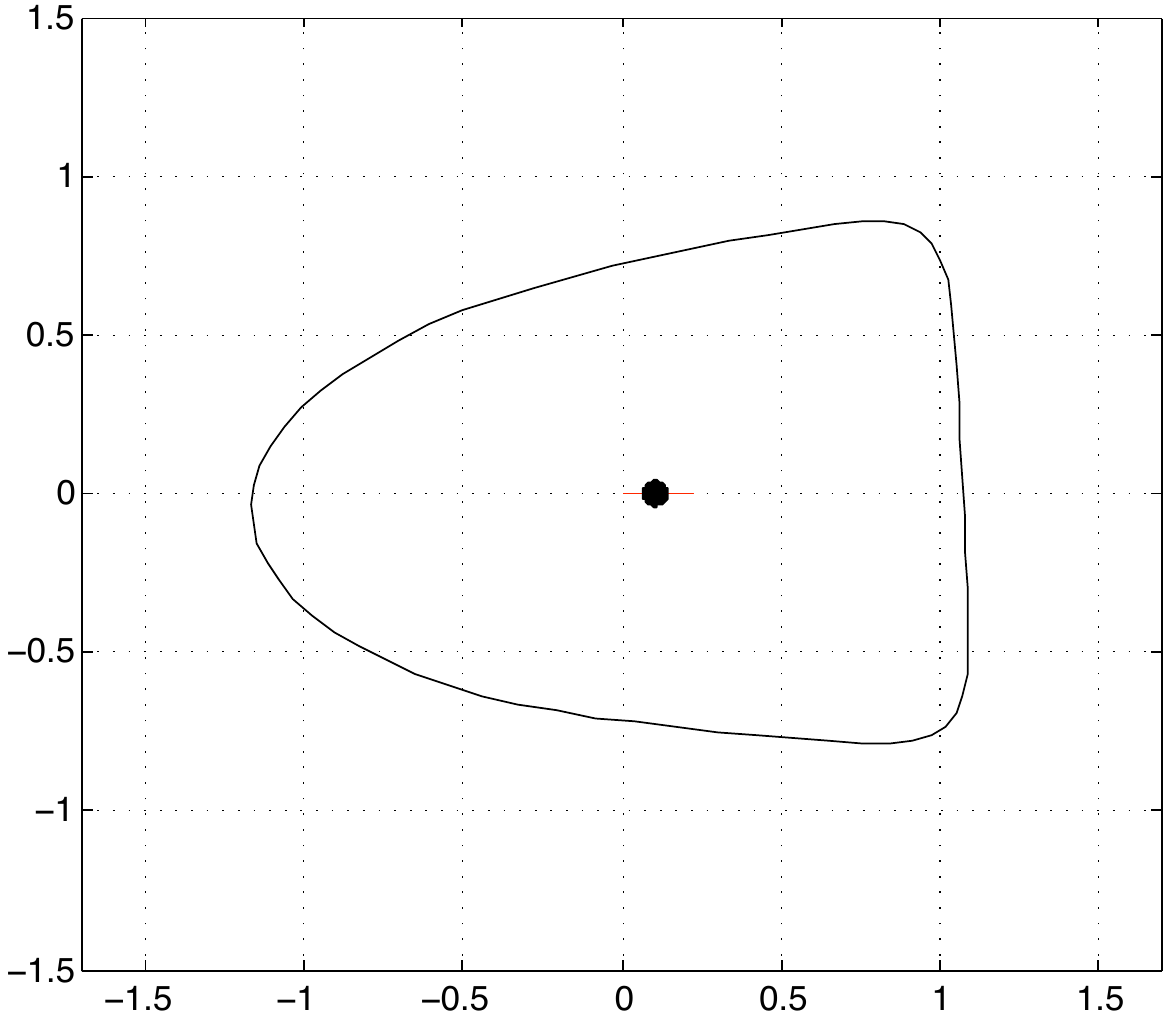}}
     &
     \subfigure
     {\includegraphics[width=.12\textwidth]{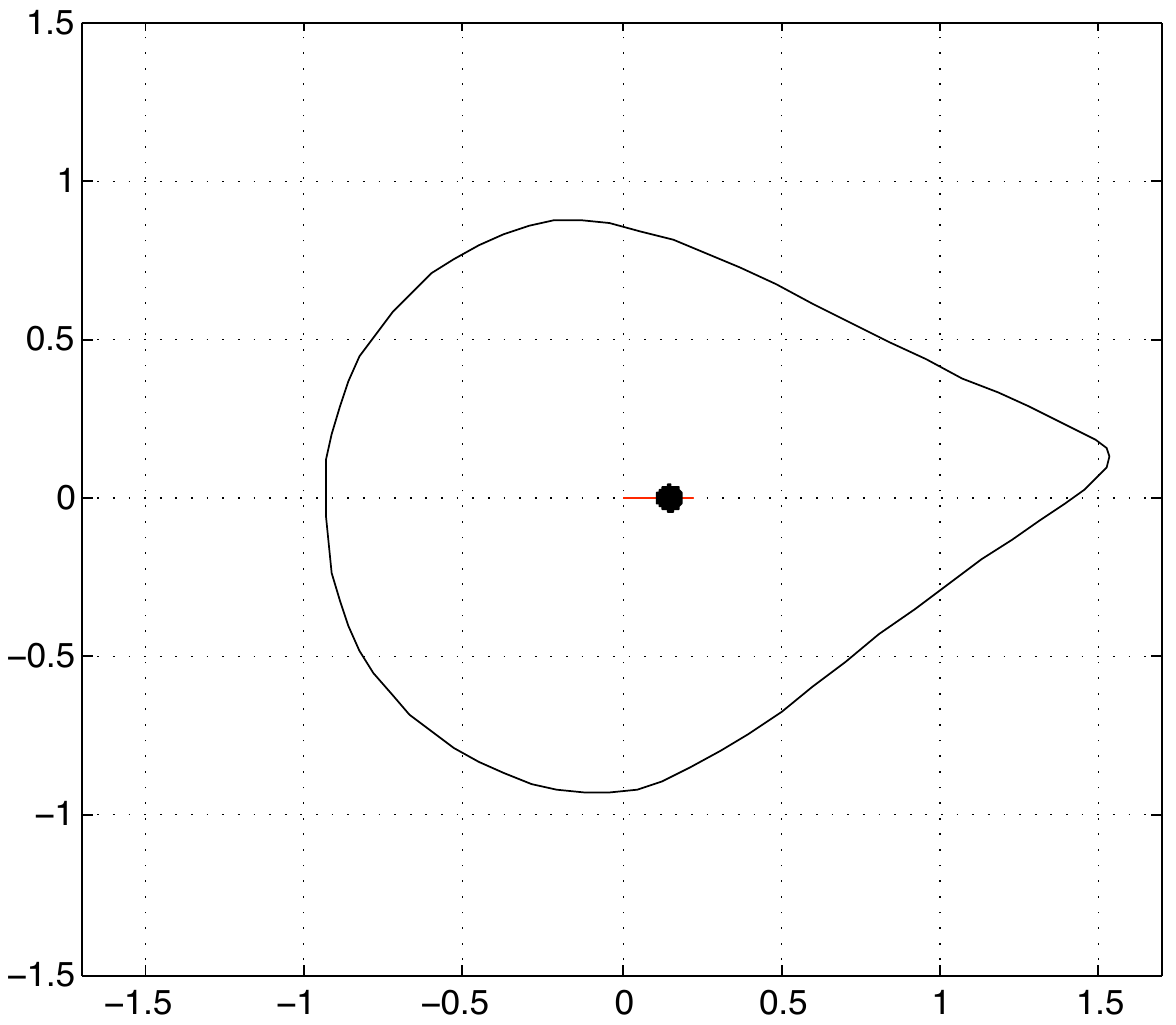}}
     &
     \subfigure
     {\includegraphics[width=.12\textwidth]{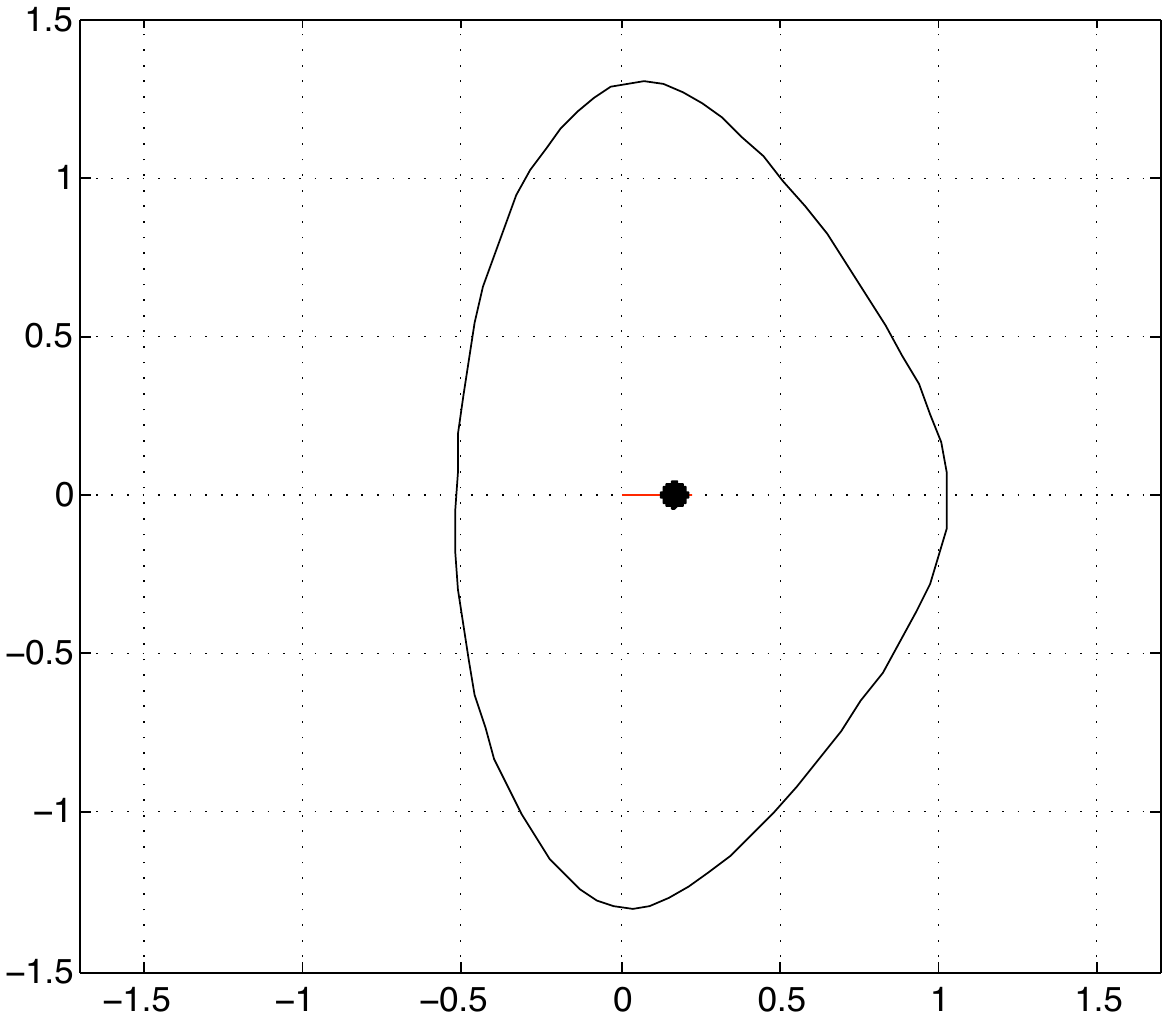}}
     &
     \subfigure
     {\includegraphics[width=.12\textwidth]{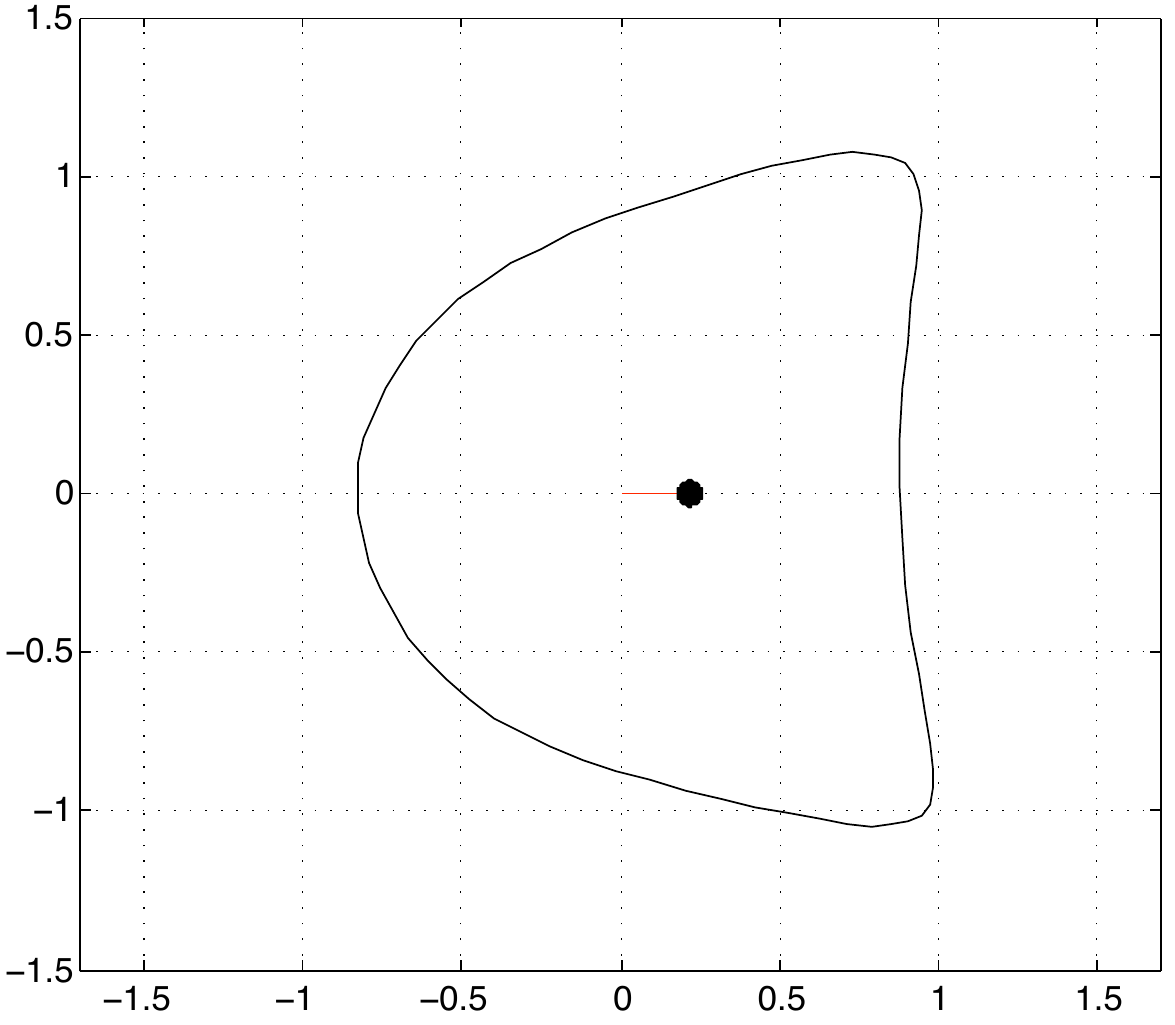}}
     &
     \subfigure
     {\includegraphics[width=.12\textwidth]{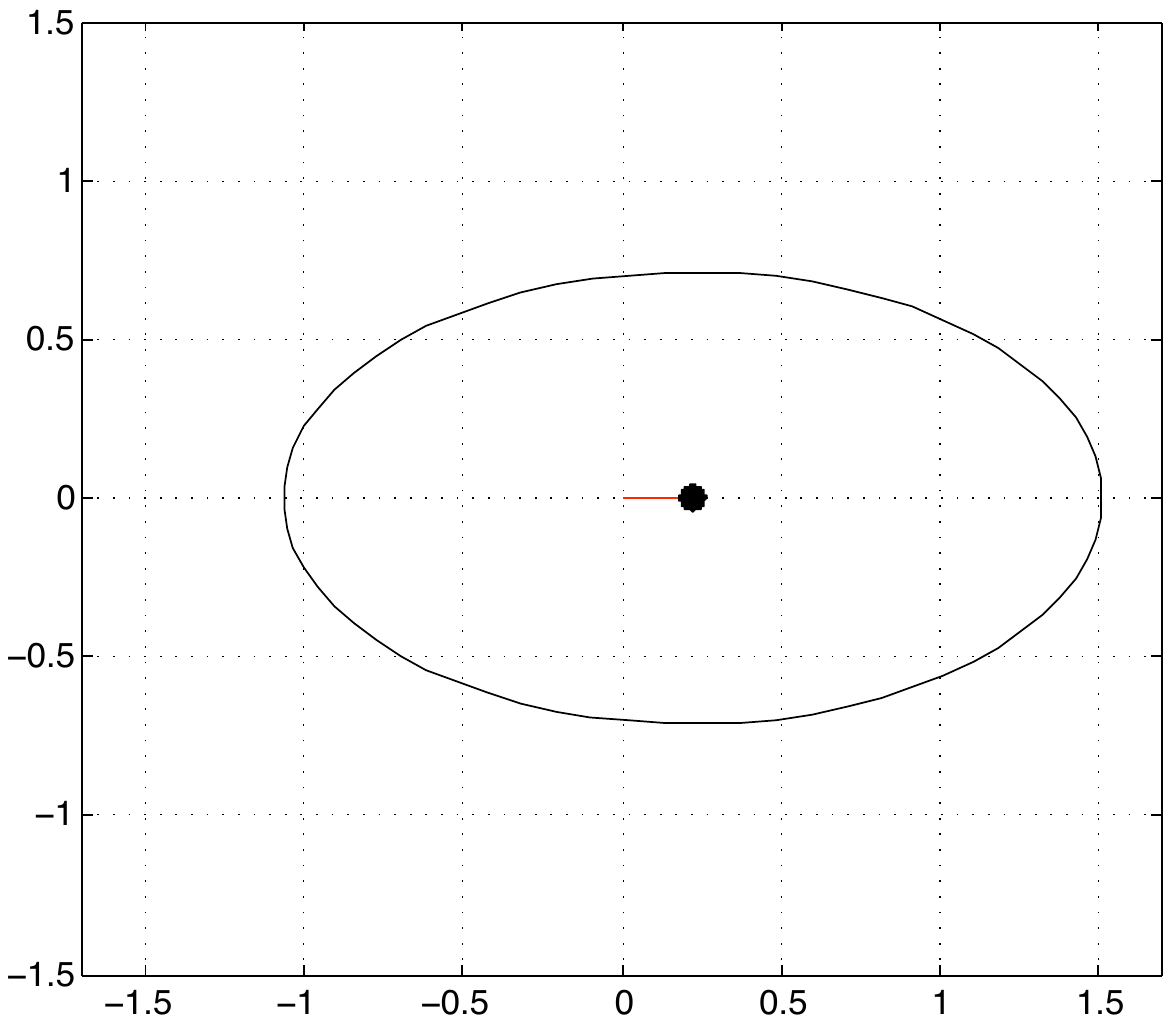}}\\
      \hline
      \small{$t=0$}&\small{$t=4\pi/7$}&\small{$t=8\pi/7$}&\small{$t=12\pi/7$}&\small{$t=16\pi/7$}&\small{$t=20\pi/7$}&\small{$t=24\pi/7$}&\small{$t=4\pi$}\\
      \hline
      \end{tabular}
      \caption{\label{fig:22} On the first row, the animal swims forward while it swims backward on the second row. The shape-changes seem to be similar in both cases.}
      \end{figure}
      \noindent Obviously, the shape-changes are actually different but the difference cannot be observed at a (time and space) macro-scale. The controls are set to be in both cases: $\alpha_1(t)=t$, $\alpha_2(t)=\pi/6$, $\alpha_3(t)=\pi/12$, $\alpha_4(t)=\pi/12$ and 
      $h_k(t)=0$ ($k=1,2,3$) for all $t\in[0,4\pi]$ but $\alpha_5(t)=0$ in the first case while $\alpha_5(t)=-10000t$ in the second one. It means that in the second case and with respect to the first case, there is superimposed shape-changes with very high frequency and very low amplitude making the animal swimming backward. We refer to the web page \url{http://www.iecn.u-nancy.fr/~munnier/page_amoeba/control_index.html} for relating animations and further explanations.

\appendix
\section{Additional material relating to the modeling}
\label{additional-material}
\subsection{Banach spaces of complex series}
\label{banach:series}
\subsubsection{Infinite dimensional spaces}
Remember that we denote the complex sequences $\mathbf c:=(c_k)_{k\geq 1}$ with $c_k=a_k+ib_k$, $i^2=-1$ and $a_k,b_k\in\mathbf R$ for all $k\geq 1$. 
All the sequences we consider in this article live in the spaces:
\begin{align*}\mathcal S&:=\big\{(c_k)_{k\geq 1}\in{\mathbf C}^{\mathbf N}\,:\,\sum_{k\geq 1}k(|a_k|+|b_k|)<+\infty\big\},\\
\mathcal T&:=\big\{(c_k)_{k\geq 1}\in{\mathbf C}^{\mathbf N}\,:\,\sum_{k\geq 1}k(|a_k|^2+|b_k|^2)<+\infty\big\},
\end{align*}
which are Banach spaces once endowed with their natural norms. We denote $\mathcal B$ the unitary open ball of $\mathcal S$ and
$$\mathcal D:=\Big\{\mathbf c:=(c_k)_{k\geq 1}\in\mathcal S\,:\,\sup_{z\in\partial D}\Big|\sum_{k\geq 0}(k+1) c_{k+1} z^{k}\Big|<1\Big\}.$$
According to the inequality $\|\mathbf c\|_{\mathcal T}\leq\|\mathbf c\|_{\mathcal S}$ for all $\mathbf c\in \mathcal S$, we deduce that $\mathcal S\subset\mathcal T$.
If we define $\mathbf a^j=(a^j_k)_{k\geq 1}$ and $\mathbf b^j=(b^j_k)_{k\geq 1}\in\mathcal S$ ($j\geq 1$), the complex sequences such that $a^j_k=\delta^j_k$ (the Kronecker symbol) and $b^j_k=i\delta^j_k$, then $\{\mathbf a^j,\,\mathbf b^j, j\geq 1\}$ is a Schauder basis of $\mathcal T$ and $\mathcal S$. Using the Hilbert structure of $\mathcal T$ to express the duality product, the dual of $\mathcal S$ can be identified with:
$$\mathcal S':=\big\{(c_k)_{k\geq 1}\in{\mathbf C}^{\mathbf N}\,:\,\max\{|a_k|,|b_k|,\,k\geq 1\}<+\infty\big\}.$$
We introduce next:
$$\begin{array}{rcl}
F:\mathcal S&\to&\mathcal S',\\
\mathbf c&\to&F(\mathbf c)
\end{array}\quad\text{and}\quad
\begin{array}{rcl}
G:\mathcal S&\to&\mathcal S',\\
\mathbf c&\to&G(\mathbf c),
\end{array}
$$
where, for all $\widetilde{\mathbf c}:=(\widetilde c_k)_{k\in{\mathbf N}}$, $\widetilde c_k=\widetilde a_k+i\widetilde b_k$, we  set: 
\begin{equation}
\label{def:fg}
\langle F({\mathbf c}),\widetilde{\mathbf c}\rangle:=\sum_{k\geq 1}\frac{1}{k+1}(\widetilde a_kb_k-\widetilde b_ka_k)~\text{ and }~
\langle G({\mathbf c}),\widetilde{\mathbf c}\rangle:=\sum_{k\geq 1}k(\widetilde a_ka_k+\widetilde b_kb_k).
\end{equation}
\subsubsection{Finite dimensional spaces}
The projector $\Pi_N$ is defined for any integer $N\geq 1$ and any complex sequence $\mathbf c=(c_k)_{k\geq 1}$ by:
\begin{equation}
\label{defpiN}
(\Pi_N(\mathbf c))_k=\begin{cases} c_k&\text{if }k\leq N,\\
0&\text{if }k>N.
\end{cases}
\end{equation}
We denote $\mathcal S_N:=\Pi_N(\mathcal S)$ and $\mathcal T_N:=\Pi_N(\mathcal T)$. Throughout the paper, we will identify $\mathcal S_N$ and $\mathcal T_N$ with, respectively, $S_N$ and $T_N$ that are nothing but $\mathbf C^N$ (or $\mathbf R^{2N}$),
but endowed respectively with the norms:
$$
\|\mathbf c\|_{S_N}:=\sum_{k=1}^N k(|a_k|+|b_k|)\quad\text{and}\quad
\|\mathbf c\|_{T_N}:=\left[\sum_{k=1}^N k(|a_k|^2+|b_k|^2)\right]^{1/2},
$$
for all $\mathbf c:=(c_k)_{1\leq k\leq N}\in{\mathbf C}^N$.
We recall the standard inequalities, for all integer $N\geq 1$:
\begin{equation}
\label{standard_ineq}
\|\mathbf c\|_{T_N}\leq \|\mathbf c\|_{S^1_N}\leq\sqrt{N(N+1)} \|\mathbf c\|_{T_N},\quad(\mathbf c\in{\mathbf C}^N).
\end{equation}
Remember that for all integers $N\geq 1$ and all $\mu>0$, we have defined
$\mathcal E_N(\mu):=\{\mathbf c\in\mathcal S_N\,:\,\|\mathbf c\|_{\mathcal T}=\mu\}$. As for the space $\mathcal S_N$ and $S_N$, we have identified $\mathcal E_N(\mu)$ with 
$$E_N(\mu):=\big\{(a_1,b_1,a_2,b_2,\ldots,a_N,b_N)\in\mathbf R^{2N}\,:\,\sum_{k =1}^N k(a_k^2+b_k^2) = \mu^2\big\},$$
which is diffeomorphic to the $2N-1$ dimensional Euclidian sphere. 
For any $\mathbf c, \widetilde{\mathbf c}\in\mathbf C^N$, the definition \eqref{def:fg} of $F$ and $G$ leads us to introduce:
\begin{equation}
\label{def:fgN}
\langle F_N({\mathbf c}),\widetilde{\mathbf c}\rangle:=\sum_{k=1}^N\frac{1}{k+1}(\widetilde a_kb_k-\widetilde b_ka_k)~\text{ and }~
\langle G_N({\mathbf c}),\widetilde{\mathbf c}\rangle:=\sum_{k=1}^Nk(\widetilde a_ka_k+\widetilde b_kb_k).
\end{equation}
\begin{rem}
\label{choce:mu}
According to \eqref{standard_ineq}, when $\mu<1/\sqrt{N(N+1)}$,  $\mathcal E_N(\mu)$ and $E_N(\mu)$ are both included in the unitary ball of their ambient spaces. In particular, for such values of $N$ and $\mu$ and for any control function $t\in[0,T]\mapsto\mathbf c(t)\in \mathcal E_N(\mu)$ there is no difference between being Physically allowable in the sense of Definition~\ref{alow:cont} and being Mathematically allowable in the sense of Definition~\ref{allowable_finite}.
\end{rem}
\subsection{Neumann boundary value problem}
\label{neumann:boun}
Note that this section is self-contained and independent, including the notation. We recall some results about the well-posedness of a Neumann boundary value problem 
in an exterior domain. 

Let $\Omega$ be the exterior of a compact in ${\mathbf R}^2$, assume that $\Omega$ is Lipschitz continuous, connected and consider the following general problem:
\begin{equation}
\label{nemann_formal}
-\Delta u=0\quad\text{in }\Omega,\qquad\partial_n u=g\quad\text{on }\partial\Omega,
\end{equation}
where $g$ is a given function in $L^2(\partial \Omega)$ satisfying the so-called compatibility condition:
\begin{equation}
\label{compatibility}
\int_{\partial \Omega}g\,{\rm d}\sigma=0,
\end{equation}
and $n$ is the unitary normal to $\partial\Omega$ directed toward the exterior of $\Omega$. The compatibility condition leads us to introduce $L^2_N(\partial\Omega):=\{g\in L^2(\partial\Omega)\text{ s.}\,\text{t. }\eqref{compatibility}\text{ holds}\}.$
We denote by ${\mathcal D}'(\Omega)$ the space of distributions, define
the {\it weight} function $\rho(x):=[\sqrt{1+|x|^2}\log(2+|x|^2)]^{-1},$
and introduce the quotient weighted Sobolev space:
$$
H^1_N(\Omega):=\{\psi\in{\mathcal D}'(\Omega)\,:\,\rho\psi\in L^2(\Omega),\,
\partial_{x_i}\psi\in L^2(\Omega),\,\forall\,i=1,2\}/{\mathbf R}.
$$
The quotient means that two functions differing only by an additive constant are equal in this space. It is a Hilbert space once endowed with the scalar product and associated norm:
$$(u,v)_{H^1_N(\Omega)}:=\int_{\Omega}\nabla u\cdot\nabla v\,{\rm d}x,\qquad
\|u\|_{H^1_N(\Omega)}:=(u,u)_{H^1_N(\Omega)}^{1/2}.$$
The variational formulation of \eqref{nemann_formal} is:
\begin{equation}
\label{neumann:varia}
\int_{\Omega}\nabla u\cdot\nabla v\,{\rm d}x= \int_{\partial\Omega}g v\,{\rm d}\sigma,\quad\forall\,v\in H^1_N(\Omega),
\end{equation}
and Lax-Milgram Theorem ensures that this problem admits one unique solution $u\in H^1_N(\Omega)$.
\begin{theorem}
The linear operator:
\begin{equation}
\label{def:operator:N}
\Delta_N^{-1}:g\in L^2_N(\partial\Omega)\mapsto u\in H^1_N(\Omega),
\end{equation} where $u$ is the solution of \eqref{neumann:varia}, is well-defined and continuous (i.e., $\Delta^{-1}_N\in \mathcal L(L^2_N(\partial\Omega),H^1_N(\Omega))$).
\end{theorem}
\subsection{Regularity of the potential function $\xi^d$}
\label{regularity_of_xid}
According to \eqref{def:xi_d} and \eqref{express:poten:complex}, we can state that for all $z\in\partial D$ and $\mathbf c\in\mathcal D$:
\begin{equation*}
\partial_n\xi^a_k(\mathbf c)(z):=b^a_{k,0}+\langle b^a_{k,1},\mathbf c\rangle(z),\qquad
\partial_n\xi^b_k(\mathbf c)(z):=b^b_{k,0}+\langle b^b_{k,1},\mathbf c\rangle(z),
\end{equation*}
where we have defined for all $z\in\partial D$:
\begin{alignat*}{3}
b^a_{k,0}(z)&:=-\Re(z^{k+1}),&\qquad&\langle b^a_{k,1},\mathbf c\rangle(z)&:=\!\!\!\!\!\sum_{j\geq 1-k}(j+k)\Re\left(\frac{c_{j+k}}{z^j}\right)-ka_k,\\
b^b_{k,0}(z)&:=-\Im(z^{k+1}),&&\langle b^b_{k,1},\mathbf c\rangle(z)&:=\!\!\!\!\!\sum_{j\geq 1-k}(j+k)\Im\left(\frac{c_{j+k}}{z^j}\right)-kb_k.
\end{alignat*}
It is clear that for any fixed $\mathbf c\in\mathcal D$, all of the series converge normally in $z$ on $\partial D$ and hence their sums define continuous functions on $\partial D$. However, we will rather consider them as functions of $L^2_N(\partial D)$. For all $\mathbf c\in\mathcal D$, $\widetilde{\mathbf c}=(\widetilde a_k+i\widetilde b_k)_{k\geq 1}\in\mathcal S$ and $z\in\partial D$, we have the following decomposition of the boundary data of $\langle \partial_n\xi^d(\mathbf c),\widetilde{\mathbf c}\rangle$:
\begin{equation*}
\langle \partial_n\xi^d(\mathbf c),\widetilde{\mathbf c}\rangle(z):=\sum_{k\geq 1}\widetilde a_k b^a_{k,0}(z)+
\widetilde b_k b^b_{k,0}(z)
+\widetilde a_k\langle b^a_{k,1},\mathbf c\rangle(z)+
\widetilde b_k\langle b^b_{k,1},\mathbf c\rangle(z).
\end{equation*}
In this form and based on simple estimates, we can easily prove that $\partial_n\xi^d\in\mathcal P(\mathcal D,\mathcal L(\mathcal S,L^2_N(\partial D)))$. We invoke next the continuity of the linear operator $\Delta^{-1}_N$ defined in \eqref{def:operator:N}, to conclude that  $\xi^d\in\mathcal P(\mathcal D,\mathcal L(\mathcal S,H^1_N(\Omega)))$.
\section{A brief survey of the Orbit Theorem}\label{SEC_Appendix_Orbit_Th}
In this Appendix, we aim to recall the statement of the Orbit Theorem used in Paragraph~\ref{SEC-Track} to prove some trackability properties.   
The material presented below is now considered as a classical part of geometric control theory. 
It has been introduced in the beginning of the 20th century independantly by Rashevsky (1938) and Chow (1939) following the ideas of Caratheodory (1909). The results were unified and generalized by Jurdjevic and  Sussmann in the 70's. In the following, we have chosen a very simplified presentation (restrained to the symmetric analytic case) of the exposition of \cite{agrachev}. Many other very good textbooks present this material with detailed proofs and comments, see for instance \cite{Isidori}, \cite{Jurdjevic}, \cite{MR1867362} or the research papers \cite{MR0321133} and \cite{MR0331185}. 
 
Throughout this section, $M$ is a real analytic manifold, and $\cal G$ a set of analytic vector fields on $M$. We do not assume in general that the fields from $\cal G$ are complete.
\subsection{Attainable sets}
 Let $f$ be an element of $\cal G$ and $q^{\ast}$ be an element of $M$. 
The Cauchy problem
\begin{equation}\label{EQ-Cauchy}
 \dot{q} = f(q),\qquad q(0)  =  q^{\ast},\end{equation}
admits a solution defined on the open interval $I(f,q^{\ast})$ containing $0$. For any real $t$ in $I(f,q^{\ast})$ we denote the value of the solution of (\ref{EQ-Cauchy}) at time $t$ by $e^{tf}(q^{\ast})$. We denote by $I(f,q^{\ast})^+=I(f,q^{\ast})\,\cap\, ]0,+\infty[$ the positive elements of  $I(f,q^{\ast})$.
   
For any element $q_0$ in $M$ and any positive real number $T$, we define the \emph{attainable set at time $T$} of $\cal G$ from $q_0$ by the set ${\cal A}_{q_0}(T)$ of all points of $M$ that can be attained with $\cal G$ using piecewise constants controls in time $T$
\begin{multline*} {\cal A}_{q_0}(T)= \Big\{ e^{t_p f_p}\circ e^{t_{p-1} f_{p-1}} \circ \cdots\circ e^{t_1 f_1} (q_0)\,:\,
 p \in \mathbf{N},\,  f_i \in {\cal G},\\ t_i\in I(f_i, e^{t_{i-1} f_{i-1}}\circ \cdots \circ e^{t_1 f_1}(q_0))^+,\, t_1+\cdots+t_p=T \Big \},
\end{multline*}
the times $t_i$ and the fields $f_i$ 
being chosen in such a way that every written quantity exists. We define also the \emph{orbit} of $\cal G$ trough $q_0$ by the set ${\cal O}_{q_0}$ of all points of $M$ that can be attained with $\cal G$ using piecewise constant controls, \emph{at any positive or negative time}
\begin{multline*} {\cal A}_{q_0}(T)= \Big\{ e^{t_p f_p}\circ e^{t_{p-1} f_{p-1}} \circ \cdots\circ e^{t_1 f_1} (q_0)\,:\, p \in \mathbf{N},\, f_i \in {\cal G}, \\
t_i\in I(f_i, e^{t_{i-1} f_{i-1}}\circ \cdots e^{t_1 f_1}(q_0)) \Big\}.
\end{multline*}
Of course, if $\cal G$ is a cone, that is if $ \lambda f \in {\cal G}$ for any positive $\lambda$ as soon as $f $ belongs to $\cal G$, the set ${\cal A}_{q_0}(T)$ does not depend on the positive $T$ but only on $q_0$.
If $\cal G$ is assumed to be symmetric, that is if $-f$ belongs to $\cal G$ as soon as $f$ belongs to $\cal G$, then the orbit of $\cal G$ trough a point $q_0$ is  the union of all attainable sets at positive time of $\cal G$ from $q_0$. 

\subsection{Lie algebra of vector fields}

If $f_1$ and $f_2$ are two vector fields on $M$ and $q$ is a point of $M$, the \emph{Lie bracket} $[f_1,f_2](q)$ of $f_1$ and $f_2$ at a point $q$ is the derivative at $t=0$ of the curve $t \mapsto \gamma(\sqrt{t})$ where $\gamma$ is defined by $\gamma(t):=e^{-t f_2} e^{-t f_1} e^{t f_2} e^{t f_1} (q)$ for $t$ small enough. The Lie bracket of $f_1$ and $f_2$ at a point $q$ is an element of the tangent space $T_{q}M$ of $M$ at the point $q$. The Lie bracket is bilinear and skew-symmetric in $f_1$ and $f_2$, and measures the non-commutativity of the fields $f_1$ and $f_2$ (see \cite[Prop 2.6]{agrachev}).
\begin{prop}
For any $f_1$, $f_2$ in $\cal G$, we have the equivalence:
$$ e^{t_1 f_1} e^{t_2 f_2} = e^{t_2 f_2} e^{t_1 f_1} \Leftrightarrow [f_1,f_2]=0$$
for all times $t_1$ and $t_2$ (if any) for which the expressions written in the left hand side of the above equivalence make sense. 
\end{prop}
Lie brackets of vectors fields are easy to compute with the following formulas (see \cite[Prop 1.3]{agrachev} and \cite[Exercise 2.2]{agrachev}).
\begin{prop} \label{PRO_calculLieBracket}
For any $f_1$, $f_2$ in $\cal G$, for any $q$ in $M$, 
$$[f_1,f_2](q)=\frac{df_2}{dq}f_1(q)-\frac{df_1}{dq}f_2(q).$$
\end{prop}
Further, we have the useful property:
 \begin{prop} \label{PRO_LieBracket_Multiple}
 Let $f_1$ and $f_2$ be two smooth vector fields on $M$, and let $a,b:M\rightarrow \mathbf{R}$ be two smooth functions.
 Then  $$[aX,bY]=a b[X,Y] +\left (\frac{db}{dq}X \right )Y -\left (\frac{da}{dq}Y \right )X.$$
 \end{prop}
 From the Lie brackets, we can define the Lie algebra:
 \begin{defn}
 The \emph{Lie algebra} of ${\cal G}$ is the linear span of all Lie brackets, of any length, of the elements of $\cal G$
 $$\mathrm{Lie}~ {\cal G}=\mathrm{span}\big\{ [f_1,[\ldots[f_{k-1},f_k]\ldots]],\, k \in \mathbf{N},\, f_i \in {\cal G} \big\},$$
 which is a subset of all the vector fields on $M$. 
 \end{defn}
We denote by $\mathrm{Lie}_q {\cal G}:=\big\{g(q),\, g \in \mathrm{Lie}~{\cal G}\big\}$ the evaluation $\mathrm{Lie}_q {\cal G}$ of the Lie algebra generated by $\cal G$ at a point $q$ of $M$.

\subsection{The Orbit Theorem}

The Orbit Theorem describes the differential structure of the orbit trough a point (see for instance \cite[Th 5.1]{agrachev} for a proof).
\begin{theorem}[Orbit Theorem]\label{THE_Orbit}  For any $q$ and $q_0$ in $M$:
\begin{enumerate} 
\item\label{conc1} ${\cal O}(q_0)$ is a connected immersed submanifold of $M$.
\item\label{conc2} If $q \in {\cal O}(q_0)$, then  $T_q{\cal O}(q_0)=\mathrm{Lie}_{q}{\cal G}$.
\end{enumerate}
\end{theorem}

\begin{rem}
The conclusion \eqref{conc1} of the Orbit Theorem holds true even if $M$ and $\cal G$ are only assumed to be smooth (and not analytic). The conclusion \eqref{conc2} is false in general when $\cal G$ is only assumed to be smooth. 
\end{rem}
 The Orbit Theorem has many consequences, among them the following useful properties (see \cite[Th 5.2]{agrachev} for a proof and further discussion).

\begin{theorem}[Rashevsky-Chow]
If $\mathrm{Lie}_q{\cal G}=T_qM$ for every $q$ in $M$, then the orbit of $\cal G$ through $q$ is equal to $M$.  
\end{theorem}

\begin{prop}\label{PRO_CompleteControl}
If $\cal G$ is a symmetric cone such that $\mathrm{Lie}_q{\cal G}=T_qM$ for every $q$ in $M$, then the attainable set at any positive time of any point of $M$ is equal to $M$.  
\end{prop}

\section{Acknowledgements}

It is a pleasure for the authors to thank Mario Sigalotti for his contribution to the proof of Proposition~\ref{PRO_CalculAlgebreLieDim4} and Denzil Millichap for many corrections and suggestions. 

{\small
\bibliographystyle{abbrv}
\bibliography{biblio9}
}
\end{document}